\documentclass[10pt]{amsart}
\usepackage{amssymb}
\usepackage{amsbsy}
\usepackage{amscd}
\usepackage[mathscr]{eucal}
\usepackage{verbatim}
\usepackage{version}
\usepackage{color}
\usepackage[matrix,arrow]{xy}

\def\cal{\mathcal}
\def\Bbb{\mathbb}
\def\frak{\mathfrak}

\newenvironment{NB}{
\color{red}{\bf NB}. \footnotesize 
}{}
\excludeversion{NB}
\newenvironment{NB2}{
\color{blue}{\bf NB}. \footnotesize
}{}

\newcommand{\bb} {\mathbb}
\newcommand{\bl} {\mathbf}
\newcommand{\frk}{\mathfrak}

\newcommand{\wt}{\widetilde}

\newcommand{\simto}{\xrightarrow{\ \sim\ }}

\newcommand{\ev}{\operatorname{ev}}

\newcommand{\Eq}  {\operatorname{Eq}}
\newcommand{\FM}  {\operatorname{FM}}
\newcommand{\GL}  {\operatorname{GL}}
\newcommand{\LieO}{\operatorname{O}}

\newcommand{\SL}  {\operatorname{SL}}

\newcommand{\Pic}{\operatorname{Pic}}

\newcommand{\Supp}{\operatorname{Supp}}
\newcommand{\ch}{\operatorname{ch}}

\newcommand{\Coh}{\operatorname{Coh}}
\newcommand{\divisor}{\operatorname{div}}
\newcommand{\Spe}{\operatorname{Spe}}
\newcommand{\Jac}{\operatorname{Jac}}

\newcommand{\Ext}{\operatorname{Ext}}
\newcommand{\Hom}{\operatorname{Hom}}

\newcommand{\codim}{\operatorname{codim}}
\newcommand{\im}{\operatorname{im}}
\newcommand{\Aut}{\operatorname{Aut}}
\newcommand{\rk}{\operatorname{rk}}

\newcommand{\NS}{\operatorname{NS}}
\newcommand{\coker}{\operatorname{coker}}

\newcommand{\Alb}{\operatorname{Alb}}
\newcommand{\Hilb}{\operatorname{Hilb}}

\newcommand{\Amp}{\operatorname{Amp}}

\newcommand{\Div}{\operatorname{Div}}

\newcommand{\alg}{\operatorname{alg}}

\newcommand{\Stab}{\operatorname{Stab}}

\newcommand{\Sym}{\operatorname{Sym}}
\newcommand{\Mov}{\operatorname{Mov}}
\newcommand{\Nef}{\operatorname{Nef}}

\font\b=cmr10 scaled \magstep5
\def\bigzerou{\smash{\lower1.7ex\hbox{\b 0}}}

\numberwithin{equation}{section}

\theoremstyle{plain}
 \newtheorem{thm}{Theorem}[section]
 \newtheorem{lem}[thm]{Lemma}
 \newtheorem{prop}[thm]{Proposition}
 \newtheorem{cor}[thm]{Corollary}
 
\theoremstyle{definition}
 \newtheorem{defn}[thm]{Definition}
 
\theoremstyle{remark}
 \newtheorem{rem}[thm]{Remark}
 \newtheorem{ex}[thm]{Example}
 
\begin{document}

\title{Bridgeland's stability and the positive cone
of the moduli spaces of stable objects on an abelian 
surface.}
\author{K\={o}ta Yoshioka}
\address{Department of Mathematics, Faculty of Science,
Kobe University,
Kobe, 657, Japan
}
\email{yoshioka@math.kobe-u.ac.jp}

\thanks{
The author is supported by the Grant-in-aid for 
Scientific Research (No.\ 22340010, No.\ 26287007), JSPS}


\begin{abstract}
We shall study the chamber structure of positive cone
of the albanese fiber of the moduli spaces of stable objects
on an abelian surface 
via the chamber structure of stability conditions.
\end{abstract}

\maketitle
%\tableofcontents

\section{Introduction}

The space of stability conditions on an abelian surfaces $X$
is studied by Bridgeland in \cite{Br:3}.
In particular, he completely described
a connected $\Stab(X)^*$ consisting of stability conditions
$\sigma$
such that the structure sheaves of points $k_x$ $(x \in X)$
are stable of a fixed phase $\phi$.
In the space of stability conditions, there is a natural action of
the universal cover
$\widetilde{\GL}^+(2,{\Bbb R})$ of $\GL^+(2,{\Bbb R})$.
In our situation,
$\Stab(X)^*/\widetilde{\GL}^+(2,{\Bbb R})$ is isomorphic
to $\NS(X)_{\Bbb R} \times \Amp(X)_{\Bbb R}$
as stated in \cite[sect. 15]{Br:3}.
In particular, if $\NS(X)={\Bbb Z}H$, then
$\Stab(X)^*/\widetilde{\GL}^+(2,{\Bbb R})$ is isomorphic
to the upper half plane ${\Bbb H}$.
For the stability conditions $\sigma_{(\beta,\omega)}$
corresponding to $(\beta,\omega) \in 
\NS(X)_{\Bbb R} \times \Amp(X)_{\Bbb R}$,
moduli spaces of 
$\sigma_{(\beta,\omega)}$-semi-stable objects are extensively studied in 
\cite{MYY:2011:1}, \cite{MYY:2011:2} and \cite{YY2}.
In particular, the projectivity of the moduli spaces are proved
for a general $\sigma_{(\beta,\omega)}$.
We also constructed ample line bundles on the moduli spaces.  
%Applying these results on the moduli of Bridgeland stable objects,
As a consequence of these results, 
we also got some results
on the moduli spaces of Gieseker semi-stable sheaves. 
Indeed for a parameter $(\beta,\omega)=(\beta,tH)$ $(t \gg 0)$ called 
the large volume limit,  
Bridgeland stability coincides with Gieseker
stability.
For the study of Gieseker stability on abelian surfaces,
Fourier-Mukai transforms
are very important tool, though
Gieseker stability is not preserved in general.
For the proof of projectivity of 
the moduli space of Bridgeland semi-stable objects,
we constructed a Fourier-Mukai transform which induces 
an isomorphism to a moduli space of Gieseker semi-stable objects.
In this sense, Bridgeland stability is regarded as a minimal generalization 
of Gieseker stability preserved by Fourier-Mukai transforms.
%In this sense,
%Bridgeland stability is a right notion for the study of   
%Gieseker stable sheaves at least on abelian surfaces.

In this note, we continue to study the moduli spaces of Bridgeland
semi-stable objects. 
In particular, we shall study the birational geometry
of the moduli spaces.  
Before explaining our main results,
we prepare some notation and explain some results in \cite{MYY:2011:2}.
For the algebraic cohomology groups
$H^*(X,{\Bbb Z})_{\alg}:=
{\Bbb Z} \oplus \NS(X) \oplus {\Bbb Z}\varrho_X$, let
$\langle\;\;,\;\; \rangle$ be the Mukai pairing, where
$\varrho_X$ is the fundamental class of $X$.
For $x=x_0+x_1+x_2 \varrho_X$ with $x_0,x_2 \in {\Bbb Z}$
and $x_1 \in \NS(X)$,
we also write $x=(x_0,x_1,x_2)$.
For $E \in {\bf D}(X)$, $v(E)=\ch(E)$ denotes the Mukai vector of $E$. 
For $v \in H^*(X,{\Bbb Z})_{\alg}$,
$M_{(\beta,\omega)}(v)$ denotes the moduli space of 
$\sigma_{(\beta,\omega)}$-semi-stable objects
$E$ with $v(E)=v$.
$M_{(\beta,\omega)}(v)$ is a projective scheme
if $(\beta,\omega)$ is general
(\cite[Thm. 1.4]{MYY:2011:2}).
If $v$ is primitive and $\langle v^2 \rangle \geq 6$, 
then as a Bogomolov factor,
we have an irreducible symplectic manifold
$K_{(\beta,\omega)}(v)$ 
which is deformation equivalent to
the generalized Kummer variety constructed by Beauville \cite{B:1}.
$K_{(\beta,\omega)}(v)$ is a fiber of the albanese map
of $M_{(\beta,\omega)}(v)$.  
We also have an isometry
$$
\theta_{v,\beta,\omega}:v^\perp \cap H^*(X,{\Bbb Z})_{\alg}
\to \NS(K_{(\beta,\omega)}(v)) 
$$
where $\NS(K_{(\beta,\omega)}(v))$ is equipped with 
the Beauville-Fujiki form.
For a Mukai vector $v$ of a coherent sheaf (i.e., $v=v(F), F \in \Coh(X)$),
$M_H^\beta(v)$ denotes the moduli space of $\beta$-twisted
semi-stable sheaves $E$ with 
$v(E)=v$. 
If $\beta=0$, then we denote it by $M_H(v)$.
Since $M_H^\beta(v)=M_{(\beta,tH)}(v)$ $(t \gg 0)$,
a fiber $K_H^\beta(v)$ of the albanese map
is $K_{(\beta,tH)}(v)$.

In \cite[sect. 5.3]{MYY:2011:2}, we relate 
the ample cone of $K_{(\beta,\omega)}(v)$ to a chamber structure
of $\NS(X)_{\Bbb R} \times \Amp(X)_{\Bbb R}$.
%In particular, we get a description of the ample cone
%of $K_H(v)$, where $K_H(v)$ is the Bogomolov factor of $M_H(v)$.
In this note, we refine this correspondence.
For a Mukai vector $v \in H^*(X,{\Bbb Z})_{\alg}$, 
we shall construct a map from our space of stability conditions
$\NS(X)_{\Bbb R} \times \Amp(X)_{\Bbb R}$
to the positive cone $P^+(v^\perp)_{\Bbb R}$ of $v^\perp$. 
%For each ray ${\Bbb R}_{>0}\zeta$ $(\zeta \in P^+(v^\perp))$,
%there is a parameter $(\beta,\omega)$
This map is surjective
up to the action of ${\Bbb R}_{>0}$ on $P^+(v^\perp)_{\Bbb R}$.
%(Proposition \ref{prop:positive}).
More precisely, we slightly extend the map in order to treat
the boundary of positive cone. 
In order to state the precise
statement (Proposition \ref{prop:positive-intro}), 
we need more notation.

We fix a norm $|| \;\; ||$ on $\NS(X)_{\Bbb R}$. 
For the closure $\overline{\Amp(X)}_{\Bbb R}$ of the ample cone
of $X$,
we set 
\begin{equation*}
\begin{split}
C(\overline{\Amp(X)}_{\Bbb R}):=&
\{ x \in \overline{\Amp(X)}_{\Bbb R} \mid || x ||=1 \},\\
\overline{\frak H}:= & \NS(X)_{\Bbb R} \times C(\overline{\Amp(X)}_{\Bbb R})
\times {\Bbb R}_{\geq 0}.
\end{split}
\end{equation*}
Then we have an embedding
$\NS(X)_{\Bbb R} \times \Amp(X)_{\Bbb R} \to \overline{\frak H}$
by sending $(\beta,\omega)$ to
$(\beta,\omega/||\omega||,||\omega||)$.

For $v=(r,c_1,a) \in H^*(X,{\Bbb Z})_{\alg}$,
we set
\begin{equation*}
\overline{P^+(v^\perp)}_{\Bbb R}:=
\left\{x \in H^*(X,{\Bbb Z})_{\alg} \otimes {\Bbb R} \left|
\begin{aligned}
x \in v^\perp, \langle v^2 \rangle \geq 0, \\
\langle x,rH_0+(H_0,c_1)\varrho_X \rangle>0 
\end{aligned}
\right. \right\},
\end{equation*}
where $H_0$ is an ample divisor on $X$.
$\overline{P^+(v^\perp)}_{\Bbb R}$ is the closure of 
the positive cone $P^+(v^\perp)_{\Bbb R}$
of $v^\perp$. 

For $(\beta,H,t) \in \overline{\frak H}$, we set
\begin{equation*}
\xi(\beta,H,t):=
\left(r\frac{t^2 (H^2)}{2}+\langle e^\beta,v \rangle \right)
(H+(\beta,H)\varrho_X)
-(c_1-r \beta,H)\left(e^\beta-\frac{t^2(H^2)}{2}\varrho_X \right).
\end{equation*} 
Then $\xi(\beta,H,t) \in \overline{P^+(v^\perp)}_{\Bbb R}$. 
\begin{prop}[{Proposition. \ref{prop:positive}}]\label{prop:positive-intro}
We have a surjective map
\begin{equation*}
\begin{matrix}
\Xi:& \overline{\frak H} & \to &
\overline{P^+(v^\perp)}_{\Bbb R}/{\Bbb R}_{>0} 
\\
& (\beta,H,t) & \mapsto & {\Bbb R}_{>0}\xi(\beta,H,t). 
\end{matrix}
\end{equation*}
Moreover if $tH$ is ample, then 
$\xi(\beta,H,t)$ belongs to the positive cone of $v^\perp$.
\begin{NB}
If $\langle \xi(\beta,H,t)^2 \rangle=0$, then
$t=0$ or $(H^2)=0$ by Lemma \ref{lem:non-perp}.
\end{NB}
\end{prop}

\begin{NB}
\begin{defn}
We set
\begin{equation*}
\begin{split}
P^+(v^\perp)_k:=& \{x \in H^*(X,{\Bbb Z})_{\alg} \otimes k \mid
x \in v^\perp, \langle x^2 \rangle> 0, 
\langle x,rH_0+(rH_0,\delta)\varrho_X \rangle>0 \},\\
\overline{P^+(v^\perp)}_k:=& 
\{x \in H^*(X,{\Bbb Z})_{\alg} \otimes k \mid
x \in v^\perp, \langle v^2 \rangle \geq 0, 
\langle x,rH_0+(rH_0,\delta)\varrho_X \rangle>0 \},
\end{split}
\end{equation*}
where $k={\Bbb Q},{\Bbb R}$ and $H_0 \in \Amp(X)_{\Bbb Q}$.
\end{defn}
\end{NB}

We introduce the wall and chamber structures on 
$\overline{\frak H}$ and 
$\overline{P^+(v^\perp)}_{\Bbb R}$% these spaces 
and
show that they correspond each other.
By using these descriptions,
we also study the movable cone of 
$K_{(\beta,\omega)}(v)$.

Let ${\frak W}$ be the set of Mukai vectors $v_1$
such that 
\begin{equation}\label{eq:wall-cond0}
\langle v_1,v-v_1 \rangle>0,\;
\langle v_1^2 \rangle \geq 0,\; \langle(v-v_1)^2 \rangle \geq 0.
\end{equation}
Then we have a chamber structure on $P^+(v^\perp)_{\Bbb R}$ 
by the set of walls $\{v_1^\perp \mid v_1 \in {\frak W} \}$.

\begin{thm}[{Theorem \ref{thm:general:movable}}]\label{thm:movable-intro}
Assume that $(\beta,H,t) \in \overline{\frak H}$
satisfies $\xi(\beta,H,t) \not \in \cup_{v_1 \in {\frak W}} v_1^\perp$.  
Let ${\frak I}$ be the set of primitive and isotropic Mukai
vectors $u$ with $\langle u,v \rangle=0,1,2$.
Let ${\cal D}(\beta,tH)$ be the connected component
of $P^+(v^\perp)_{\Bbb R} \setminus \cup_{u \in {\frak I}}u^\perp$
containing $\xi(\beta,H,t)$.
Then 
\begin{equation*}
\overline{\Mov(K_{(\beta,tH)}(v))}_{\Bbb R}=
\theta_{v,\beta,tH}(\overline{{\cal D}(\beta,tH)}).
 \end{equation*}
Moreover 
$$
\theta_{v,\beta,tH}(H^*(X,{\Bbb Z})_{\alg} \cap
\overline{{\cal D}(\beta,tH)}) \subset
\Mov(K_{(\beta,tH)}(v)).
$$ 
\end{thm}

In the movable cone of 
$K_{(\beta,tH)}(v)$, Hassett and Tschinkel 
\cite[Thm. 7, Prop. 17]{HT} introduced the chamber structure.
The chamber structure of ${\cal D}(\beta,tH)$
by $\{v_1^\perp \mid v_1 \in {\frak W} \}$
corresponds to the chamber structure of 
the interior of
$\Mov(K_{(\beta,tH)}(v))$ via $\theta_{v,\beta,tH}$.

As an application of our results, 
we get a result on the birational structure of $M_H^\beta(v)$.
\begin{prop}[{Proposition \ref{prop:Hilb}}]\label{prop:Hilb-intro}
Let $(X,H)$ be a polarized abelian surface
and $v$ a Mukai vector such that 
$2\ell:=\langle v^2 \rangle \geq 6$.
Then $M_H^\beta(v)$ is birationally equivalent
to $\Pic^0(Y) \times \Hilb_Y^{\ell}$
if and only if there is an isotropic Mukai
vector $w \in H^*(X,{\Bbb Z})_{\alg}$ with $\langle v,w \rangle=1$,
where $Y$ is an abelian surface.   
\end{prop}
This result also follows from a characterization of
the generalized Kummer variety by Markman and Mehrotra \cite{Ma-Me}.
Proposition \ref{prop:Hilb-intro} gives an affirmative solution 
of a conjecture of  
Mukai \cite{Mukai:1980}.

\begin{cor}[{Corollary \ref{cor:Hilb}}]\label{cor:Mukai}
Let $(X,H)$ be a principally polarized abelian surface
with $\NS(X)={\Bbb Z}H$.
Let $v=(r,dH,a)$ be a Mukai vector with
$\ell:=d^2-ra \geq 3$.
Then $M_H^\beta(v)$ is birationally equivalent
to $X \times \Hilb_X^{\ell}$
if and only if
the quadratic equation 
$$
rx^2+2dxy+a y^2=\pm 1
$$
has an integer valued solution.  
\end{cor}

We also study the location of walls.
If $\rk \NS(X) \geq 2$, we show that the stabilizer of $v$ in
the group of autoequivalences is infinite.
Hence if there is a wall, then we can generate infinitely many walls
by the action of autoequivalences.
We also show that there is an example of $X$ and $v$ such that
there is no wall, which implies that the ample cone of $K_H^\beta(v)$
is the same as the positive cone and the autoequivalences act
as automorphisms of $M_H^\beta(v)$.

The study of the movable cone
is motivated by recent works
\cite{ABCH} and \cite{BM}.
They studied the movable cones of the moduli spaces
for the projective plane and a $K3$ surface 
by analyzing the chamber structure of Bridgeland's stability.
For an irreducible symplectic manifold,
Markman \cite{Ma} studied the movable cone extensively. 
In particular, he obtained a numerical characterization 
of the movable cone. 
In this sense, our result 
(Theorem \ref{thm:general:movable}) gives concrete examples of his results.
In particular, we give a moduli-theoretic explanation
of birational models of $K_H^\beta(v)$.

Let us briefly explain the contents of this note.
In section \ref{sect:pre}, we introduce some notations
and recall known results on irreducible symplectic manifolds.
In particular, we define our parameter space of stability condition
and the wall for stability conditions. We also give a characterization 
of the walls in terms of Mukai lattice 
(Proposition \ref{prop:existence-of-wall}). 
In section \ref{sect:FM}, we shall study the cohomological
action of the autoequivalences of ${\bf D}(X)$, which will be used 
to study the set of walls.
We first treat the case where $\rk \NS(X)=1$.
In this case, we can use the 2 by 2 matrices description
of the cohomological action of the Fourier-Mukai transforms
in \cite{YY1}.
We then describe the stabilizer group $\Stab(v)$ 
of a Mukai vector $v$.
By using it, we shall construct many autoequivalences
fixing $v$ for all abelian surfaces.
 
In section \ref{sect:positive-cone}, we relate our space 
of stability condition with the positive cone of the moduli spaces.
We first construct a map from the space of stability conditions to
the positive cone.
Then we describe the nef cone of the moduli spaces.
In subsection \ref{subsect:movable}, 
we study the divisorial contractions 
of the moduli spaces. Then we get the description 
of movable cones (Theorem \ref{thm:general:movable}).

In section \ref{sect:rk1}, as an example, we treat
the case where $\NS(X)={\Bbb Z}H$.
In this case, the boundaries of $\overline{P^+(v^\perp)}_{\Bbb R}$ 
are spanned by two isotropic vectors $v_\pm$.
For a Mukai vector $v=(r,dH,a)$, 
we show that $v_\pm$ are not defined over ${\Bbb Q}$
if and only if $\sqrt{\langle v^2 \rangle/(H^2)} \not \in {\Bbb Q}$.
For the rank 1 case, this condition is equivalent to 
the existence of infinitely many walls \cite{YY2}.
According to Markman's solution \cite{Ma} 
of the movable cone conjecture
of Kawamata and Morrison (\cite{Ka}, \cite{Mo}), 
we have infinitely many walls
under this condition. 
By our correspondence of the space of stability conditions
and the positive cone, we see that the accumulation points
correspond to the two boundaries ${\Bbb R}_{>0}v_\pm$
which are the accumulation points set of walls.  
Thus we get an explanation of 
the existence of accumulation points in terms of
the positive cone.
For the general cases,
if $\sqrt{\langle v^2 \rangle/(H^2)} \not \in {\Bbb Q}$, then we show that
infinitely many Fourier-Mukai transforms preserve
$v$ as in the rank 1 case. 
So there are infinitely many walls if there is a wall.
However as in the case where $\rk \NS(X) \geq 2$, 
we have an abelian surface and a Mukai vector $v$ 
such that there is no wall for $v$.
In section \ref{sect:Markman}, we shall explain how our result 
on the movable cone follows
from Markman's general theory. 
In appendix, we shall study the base of Lagrangian fibrations.

After we wrote the first version of this note,
Bayer and Macri \cite{BM:2} completed their study of the birational geometry
of moduli spaces over $K3$ surfaces. In particular, they completely 
described the nef cone and the movable cone of the moduli spaces.
Moreover the results are generalized to deformations of the moduli spaces
\cite{BHT}, \cite{Mon}.

{\it Acknowledgement.}
The author would like to thank Eyal Markman very much for his explanation
of his results in detail.

\section{Preliminaries}\label{sect:pre}

\subsection{Notation.}
We denote the category of coherent sheaves on $X$ by $\Coh(X)$ and
the bounded derived category of $\Coh(X)$ by ${\bf D}(X)$.
A Mukai lattice of $X$ consists of 
$H^{2*}(X,{\Bbb Z}):=\bigoplus_{i=0}^2 H^{2i}(X,{\Bbb Z})$ 
and an integral bilinear form
$\langle\;\;,\;\; \rangle$ on 
$H^{2*}(X,{\Bbb Z})$:
$$
\langle x_0+x_1+x_2 \varrho_X,y_0+y_1+y_2 \varrho_X \rangle:=
(x_1,y_1)-x_0 y_2-x_2 y_0  \in {\Bbb Z},
$$
where $x_1,y_1 \in H^{2}(X,{\Bbb Z})$, 
$x_0,x_2,y_0,y_2 \in {\Bbb Z}$ and $\varrho_X \in H^4(X,{\Bbb Z})$ 
is the fundamental class of $X$.
We also introduce the algebraic Mukai lattice as
the pair of $H^*(X,{\Bbb Z})_{\alg}:=
{\Bbb Z} \oplus \NS(X) \oplus {\Bbb Z}$
and $\langle\;\;,\;\; \rangle$ on 
$H^*(X,{\Bbb Z})_{\alg}$.
For $x=x_0+x_1+x_2 \varrho_X$ with $x_0,x_2 \in {\Bbb Z}$
and $x_1 \in H^2(X,{\Bbb Z})$,
we also write $x=(x_0,x_1,x_2)$. 
For $E \in {\bf D}(X)$,
$v(E):=\ch(E)$ denotes the Mukai vector of $E$.

For ${\bf E} \in {\bf D}(X \times Y)$,
we set 
$$
\Phi_{X \to Y}^{{\bf E}}(x):={\bf R} p_{Y*}({\bf E} \otimes p_X^*(x)),\;
x \in {\bf D}(X),
$$
where $p_X,p_Y$ are projections from $X \times Y$ to
$X$ and $Y$ respectively. 
%Two smooth projective varieties $Y_1$ and $Y_2$ are said to be 
%Fourier-Mukai partners if there is an equivalence 
%$\bl{D}(Y_1)\simeq\bl{D}(Y_2)$.
%We denote by $\FM(X)$ the set of Fourier-Mukai partners of $X$. 
The set of equivalences between $\bl{D}(X)$ and $\bl{D}(Y)$ 
is denoted by $\Eq(\bl{D}(X),\bl{D}(Y))$. 
%For $Y,Z\in\FM(X)$, we set 
We set 
\begin{align*}
&\Eq_0(\bl{D}(Y),\bl{D}(Z))
:=
 \left\{ \left. \Phi_{Y \to Z}^{\bl{E}[2k]} 
   \in \Eq(\bl{D}(Y),\bl{D}(Z))
   \right| \bl{E} \in \Coh(Y \times Z),\, k \in {\bb{Z}} 
 \right\},
\\
&\cal{E}(Z)
:=
% \bigcup_{Y\in\FM(Z)}\Eq_0(\bl{D}(Y),\bl{D}(Z)),
\bigcup_{Y}\Eq_0(\bl{D}(Y),\bl{D}(Z)),
\qquad
\cal{E}
:=
 \bigcup_{Z}\cal{E}(Z)
=\bigcup_{Y,Z }\Eq_0(\bl{D}(Y),\bl{D}(Z)).
\end{align*}
Note that $\cal{E}$ is a groupoid with respect to the composition 
of the equivalences.

As we explained in the introduction,
$\Stab(X)^*/\widetilde{\GL}^+(2,{\Bbb R})$ is isomorphic
to $\NS(X)_{\Bbb R} \times \Amp(X)_{\Bbb R}$.
%For $(\beta,\omega) \in \NS(X)_{\Bbb R} \times \Amp(X)_{\Bbb R}$,
Let us briefly explain a stability condition $\sigma_{(\beta,\omega)}$
associated to
$(\beta,\omega) \in \NS(X)_\bb{Q} \times \Amp(X)_{\Bbb Q}$.
Let ${\frak T}_{(\beta,\omega)}$ be a full subcategory
of $\Coh(X)$ generated by torsion sheaves and
$\mu$-stable torsion free sheaves $E$ with
$(c_1(E)-\rk E \beta,\omega)>0$, and
let ${\frak F}_{(\beta,\omega)}$ be  
a full subcategory
of $\Coh(X)$ generated by $\mu$-stable torsion free sheaves $E$ with
$(c_1(E)-\rk E \beta,\omega) \leq 0$.
$({\frak T}_{(\beta,\omega)},{\frak F}_{(\beta,\omega)})$ is a
torsion pair of $\Coh(X)$.
Let ${\frak A}_{(\beta,\omega)}$ be its tilting.
Thus,
$$
{\frak A}_{(\beta,\omega)}:=
\left\{E \in {\bf D}(X) \left|
\begin{aligned}
& H^i(E)=0,\; i \ne -1,0,\\
& H^{-1}(E) \in {\frak F}_{(\beta,\omega)},\;
H^0(E) \in {\frak T}_{(\beta,\omega)}
\end{aligned}
\right. \right\}.
$$
%Thus ${\frak A}_{(\beta,\omega)}$ is generated by
%${\frak T}_{(\beta,\omega)}$ and ${\frak F}_{(\beta,\omega)}[1]$.
Let $Z_{(\beta,\omega)}:{\bf D}(X) \to \bb{C}$ is a group homomorphism 
called the stability function.
In terms of the Mukai lattice 
$(H^*(X,\bb{Z})_{\alg}, \langle \;\;,\;\; \rangle)$, 
$Z_{(\beta,\omega)}$ is given by 
\begin{align*}
 Z_{(\beta,\omega)}(E)=\langle e^{\beta+\sqrt{-1}\omega},v(E) \rangle,
 \quad
 E \in {\bf D}(X).
\end{align*}
Then $Z_{(\beta,\omega)}(E) \in {\Bbb H} \cup {\Bbb R}_{<0}$ 
for $0 \ne E \in \frk{A}_{(\beta,\omega)}$.
We define the phase $\phi_{(\beta,\omega)}(E) \in (0,1]$
of $ 0\ne E \in \frk{A}_{(\beta,\omega)}$ by
$$
Z_{(\beta,\omega)}(E)=|Z_{(\beta,\omega)}(E)|
e^{\pi \sqrt{-1}\phi_{(\beta,\omega)}(E)}.
$$ 
Then  
$(\frk{A}_{(\beta,\omega)},Z_{(\beta,\omega)})$ 
is the stability condition $\sigma_{(\beta,\omega)}$.
In particular, $k_x$ is a stable object of the phase
$\phi_{(\beta,\omega)}(k_x)=1$.

\begin{defn}
\begin{enumerate}
\item[(1)]
An object $0 \ne E \in {\frak A}_{(\beta,\omega)}$
is $\sigma_{(\beta,\omega)}$-semi-stable
if 
$$
\phi_{(\beta,\omega)}(F) \leq \phi_{(\beta,\omega)}(E) 
$$
for all proper subobject $F \ne 0$ of $E$.
If the inequality is strict, then
$E$ is $\sigma_{(\beta,\omega)}$-stable.
\item[(2)]
An object $0 \ne E \in {\bf D}(X)$ is 
$\sigma_{(\beta,\omega)}$-semi-stable (resp.
$\sigma_{(\beta,\omega)}$-stable), if
there is an integer $n$ such that 
$E[-n] \in {\frak A}_{(\beta,\omega)}$ and
$E[-n]$ is $\sigma_{(\beta,\omega)}$-semi-stable (resp.
$\sigma_{(\beta,\omega)}$-stable).   
\end{enumerate}
\end{defn}

\subsection{A parameter space of stability conditions.}
\label{subsect:parameter-space}

For an abelian surface $X$,
the ample cone $\Amp(X)$ is described as
$$
\Amp(X)=\{ x \in \NS(X) \mid (x^2)> 0, (x,h) >0 \},
$$
where $h \in \NS(X)$ is an ample class of $X$.
We set
$$
\overline{\Amp(X)}_k:=\{ x \in \NS(X)_k \mid (x^2) \geq 0, (x,h) >0 \},
$$
where $k={\Bbb Q},{\Bbb R}$.
For a cone $V \subset {\Bbb R}^m$,
we set $C(V):=(V \setminus \{ 0 \})/{\Bbb R}_{>0}$.
We fix a norm $|| \;\; ||$ on ${\Bbb R}^m$ and identify
$C(V)$ with
$\{ x \in V \mid ||x ||=1 \}$. 
Then we have a bijection 
$V \setminus \{ 0 \} \to C(V) \times {\Bbb R}_{>0}$   
by sending $x \in V \setminus \{ 0 \}$ to
$(x/||x ||,|| x ||)$.

We have a map
\begin{equation*}
\begin{matrix}
C(\overline{\Amp(X)}_{\Bbb R}) 
\times {\Bbb R}_{\geq 0}
& \to & \overline{\Amp(X)}_{\Bbb R} \cup \{ 0 \}\\
(L,t) & \mapsto & tL
\end{matrix}
\end{equation*}
which is bijective over $\overline{\Amp(X)}_{\Bbb R}$
and the fiber over $0$ is 
$\overline{\Amp(X)}_{\Bbb R} \times \{ 0\}$.
Thus $C(\overline{\Amp(X)}_{\Bbb R}) 
\times {\Bbb R}_{\geq 0}$ is a partial compactification of 
$\overline{\Amp(X)}_{\Bbb R}$.

We set 
\begin{equation*}
\begin{split}
{\frak H}:=\NS(X)_{\Bbb R} \times C(\Amp(X)_{\Bbb R}) 
\times {\Bbb R}_{> 0},\\
\overline{\frak H}:=\NS(X)_{\Bbb R} \times 
C(\overline{\Amp(X)}_{\Bbb R}) 
\times {\Bbb R}_{\geq 0}.
\end{split}
\end{equation*}
We have an identification
\begin{equation}\label{eq:embed-H}
\begin{matrix}
{\frak H}& \to & \NS(X)_{\Bbb R} \times \Amp(X)_{\Bbb R}\\
(\beta,H,t) & \mapsto & (\beta,tH)
\end{matrix}
\end{equation}
and
these spaces are our parameter space of stability
conditions and its partial compactification.

Let us introduce a wall and chamber structure on 
$\overline{\frak H}$.
\begin{defn}[{cf. \cite[Defn. 2.7]{YY2}}]\label{defn:wall}
Let $v$ be a Mukai vector.
\begin{enumerate}
\item[(1)]
For a Mukai vector $v_1$ satisfying
\begin{equation}\label{eq:wall-cond}
\langle v_1,v-v_1 \rangle>0,\;
\langle v_1^2 \rangle \geq 0,\; \langle(v-v_1)^2 \rangle \geq 0,
\end{equation}
we define the wall $W_{v_1}$ as
\begin{equation}\label{eq:wall}
W_{v_1}:=\{(\beta,H,t) \in  \overline{\frak H} \mid  
{\Bbb R} Z_{(\beta,tH)}(v_1)={\Bbb R} Z_{(\beta,tH)}(v)\}.
\end{equation}
\item[(2)]
${\frak W}$ denotes the set of Mukai vectors
$v_1$ satisfying \eqref{eq:wall-cond}.
\item[(3)]
A chamber for stabilities is a connected component of 
$\overline{\frak H}
\setminus  \cup_{v_1 \in {\frak W}} W_{v_1}$.
\item[(4)]
We also have
a wall and chamber structure on $\NS(X)_{\Bbb R} \times \Amp(X)_{\Bbb R}$
via \eqref{eq:embed-H},
which is the same as was introduced in \cite{YY2}, \cite{MYY:2011:2}.
\item[(5)]
We say that $(\beta,H,t) \in {\frak H}$ (resp.
$(\beta,\omega) \in \NS(X)_{\Bbb R} \times \Amp(X)_{\Bbb R}$) 
is general, if it is in a chamber.
\end{enumerate}
\end{defn}

As we explained in \cite{YY2},
\cite[Prop. 5.7]{MYY:2011:2} implies that 
if $(\beta,H,t) \in W_{v_1}$, there is a properly 
$\sigma_{(\beta,tH)}$-semi-stable
object $E$ with $v(E)=v$.
In general, $W_{v_1}$ may be an empty set.
We have the following characterization for the non-emptiness
of the wall whose proof is given in subsection 
\ref{subsect:positive-cone}.

\begin{prop}\label{prop:existence-of-wall}
Let $v_1$ be a Mukai vector satisfying \eqref{eq:wall-cond}.
Then $W_{v_1} \cap {\frak H} \ne \emptyset$ 
if and only if 
\begin{equation}\label{eq:wall-cond2}
\langle v,v_1 \rangle^2>\langle v^2 \rangle
\langle v_1^2 \rangle.
\end{equation}
\end{prop}

\begin{cor}\label{cor:existence-of-wall}
Let $w$ be an isotropic Mukai vector.
If $\langle v^2 \rangle/2>\langle w,v \rangle>0$,
then $w$ satisfies \eqref{eq:wall-cond} and
$W_w \cap {\frak H}$ is non-empty.
In particular, if $\langle w,v \rangle=1,2$ and
$\langle v^2 \rangle \geq 6$, then 
$w$ satisfies \eqref{eq:wall-cond} and
$W_w \cap {\frak H} \ne \emptyset$.
\end{cor}

We set $v_1:=(r_1,\xi_1,a_1)$.
Then the defining equation of $W_{v_1}$ is
\begin{equation}\label{eq:W}
\begin{split}
& \det
\begin{pmatrix}
a-(\xi,\beta)+r \frac{(\beta^2)-t^2(H^2)}{2} &
a_1-(\xi_1,\beta)+r_1 \frac{(\beta^2)-t^2(H^2)}{2}\\
-(\xi-r \beta,H) & -(\xi_1-r_1 \beta,H)
\end{pmatrix}\\
=&
(\xi_1-r_1 \beta,H)a-(\xi-r\beta,H)a_1+
(r_1 \xi-r \xi_1,\beta)(\beta,H)\\
& +(\xi,\beta)(\xi_1,H)-(\xi_1,\beta)(\xi,H)-
(r_1 \xi-r \xi_1,H)\frac{(\beta^2)-t^2 (H^2)}{2}=0.
\end{split}
\end{equation}

\begin{lem}\label{lem:d_beta=0}
\begin{enumerate}
\item[(1)]
If $r_1 \xi-r \xi_1 \ne 0$, then
$$
W_{v_1} \not \supset \{(\beta,H,t) \in \overline{\frak H}
\mid (\xi-r\beta,H)=0 \}.
$$
\item[(2)]
If $r_1 \xi-r \xi_1=0$, then
$$
W_{v_1}=
\{(\beta,H,t) \in \overline{\frak H}
\mid (\xi-r\beta,H)=0 \}.
$$
\end{enumerate}
\end{lem}

\begin{proof}
(1)
Assume that $r_1 \xi-r \xi_1 \ne 0$.
Then we can take $H \in \Amp(X)_{\Bbb Q}$ with
$(r_1 \xi-r \xi_1,H) \ne 0$.
We take $\beta \in \NS(X)_{\Bbb Q}$ with 
$(\xi-r\beta,H)=0$.
Then we have $(\xi_1-r_1 \beta,H) \ne 0$.
Since $Z_{(\beta,tH)}(v) \ne 0$,
\eqref{eq:W} implies that 
$(\beta,H,t) \not \in W_{v_1}$. 
Since the hypersurface $(\xi-r\beta,H)=0$ is irreducible,
we get the claim.

(2)
If $r_1 \xi-r \xi_1=0$, then \eqref{eq:W} implies that
$$
\left(\frac{r_1}{r} a-a_1 \right)(\xi-r\beta,H)=0.
$$ 
Since $v_1 \not \in {\Bbb Q}v$,
$W_{v_1}$ is defined by $(\xi-r\beta,H)=0$.
\end{proof}

\begin{rem}\label{rem:d_beta=0}
The assumption of Lemma \ref{lem:d_beta=0} (2) is equivalent to
$\varrho_X^\perp \cap v^\perp =v_1^\perp \cap v^\perp$.
Indeed 
$\varrho_X^\perp \cap v^\perp =v_1^\perp \cap v^\perp$ is equivalent to
${\Bbb Q}v+{\Bbb Q}v_1={\Bbb Q}v +{\Bbb Q}\varrho_X$.
Since $v_1 \not \in {\Bbb Q}v$, it is equivalent to
$v_1 \in {\Bbb Q}v +{\Bbb Q}\varrho_X$.
\begin{NB}
Assume that $r \ne 0$.
Then (2) of Lemma \ref{lem:d_beta=0} is equivalent to
$\varrho_X^\perp \cap v^\perp =v_1^\perp \cap v^\perp$.
Indeed if $\varrho_X^\perp=v_1^\perp$ on $v^\perp$, then
$\varrho_X-\langle v,\varrho_X \rangle v/\langle v^2 \rangle$ and
$v_1-\langle v,v_1 \rangle v/\langle v^2 \rangle$ are linearly
dependent over ${\Bbb Q}$.
Since $r \ne 0$, we see that
$v_1 \in {\Bbb Q}v +{\Bbb Q}\varrho_X$.
Conversely if $v_1 \in {\Bbb Q}v +{\Bbb Q}\varrho_X$, then
${\Bbb Q}v+{\Bbb Q}v_1={\Bbb Q}v +{\Bbb Q}\varrho_X$.
Hence $\varrho_X^\perp \cap v^\perp =v_1^\perp \cap v^\perp$.
\end{NB}
\end{rem}

\subsection{Facts on irreducible symplectic manifolds.}\label{subsect:hyper}

For a smooth projective manifold $M$,
$\Amp(M)_{k} \subset \NS(M)_{k}$ denotes the ample cone
of $M$ and $\Nef(M)_{k} \subset 
\NS(M)_{k}$ denotes the nef cone of numerically effective divisors on $M$, 
where $k={\Bbb Q}, {\Bbb R}$.

\begin{defn}\label{defn:cones}
Let $M$ be a smooth projective manifold.
\begin{enumerate}
%\item[(1)]
%$\Nef(M)$ denotes the nef cone of numerically effective divisors.
\item[(1)]
\begin{enumerate}
\item
A divisor $D$ on $M$ is movable, if
the base locus of $|D|$ has codimension $\geq 2$.
\begin{NB}
$\coker(H^0({\cal O}_M(D)) \otimes {\cal O}_M \to {\cal O}_M(D))$
is at least of codimension 2.
\end{NB}
\item
$\Mov(M)_k \subset \NS(X)_k$ $(k={\Bbb Q},{\Bbb R})$ 
denotes the cone generated by movable divisors
and $\overline{\Mov(M)}_{\Bbb R}$ the closure
in $\NS(X)_{\Bbb R}$.
\begin{NB}
$D \in \NS(X)_{\Bbb Q}$ belongs to $\Mov(M)_{\Bbb Q}$
if and only if $lD$ is movable for a sufficiently large $l$.  
\end{NB}
\end{enumerate}
\item[(2)]
For an irreducible symplectic manifold $M$,
$q_M$ denotes the Beauville-Fujiki form on
$H^2(M,{\Bbb Z})$.
Then the positive cone is defined as
$$
P^+(M)_k:=\{ x \in \NS(X)_k \mid q_M(x,x)>0, q_M(x,h)>0 \}
$$ 
where $k={\Bbb Q}, {\Bbb R}$ and
$h$ is an ample divisor on $M$. 
We also set
$$
\overline{P^+(M)}_k
:=\{ x \in \NS(X)_k \mid q_M(x,x) \geq 0, q_M(x,h)>0 \}.
$$ 
\end{enumerate}
\end{defn}

\begin{rem}
By the definition,
$\Mov(M)_{\Bbb Q}=\Mov(M)_{\Bbb R} \cap \NS(M)_{\Bbb Q}$.
\end{rem}
We note that $\Mov(M)_{\Bbb Q}$ is contained in
$\overline{P^+(M)}_{\Bbb Q}$ by works of Huybrechts
(\cite{H}, \cite[Thm. 7]{HT}).
There is a different argument in \cite[Lem. 6.22]{Ma} based on 
 results of Boucksom \cite{Bo2}.

\subsection{Moduli spaces}

\begin{defn}
A Mukai vector $v:=(r,\xi,a) \in H^*(X,{\Bbb Z})_{\alg}$
is positive, if 
\begin{enumerate}
\item[(i)] $r>0$ or 
\item[(ii)] $r=0$ and $\xi$ is effective 
or 
\item[(iii)]
$r=\xi=0$ and $a>0$.
\end{enumerate}  
\end{defn}

\begin{defn}
Let $v \in H^*(X,{\Bbb Z})_{\alg}$ be a Mukai vector. 
\begin{enumerate}
\item[(1)]
If $v$ is positive, then
let $M_H^\beta(v)$ be the moduli space of $\beta$-twisted 
semi-stable sheaves $E$
on $X$ 
with $v(E)=v$.
If $\beta=0$, then we also denote $M_H^\beta(v)$ by
$M_H(v)$. 
\item[(2)]
$M_{(\beta,\omega)}(v)$ denotes the moduli space of 
$\sigma_{(\beta,\omega)}$-semi-stable objects $E$ with $v(E)=v$.
\end{enumerate}
\end{defn}

\begin{rem}
\begin{enumerate}
\item[(1)]
If $H$ is general in $\Amp(X)$, then
$M_H^\beta(v)$ does not depend on the choice of $\beta$. 
\item[(2)]
If $v$ is positive, then
$M_{(\beta+sH,tH)}(v)=M_H^\beta(v)$ for some $(s,t)$.
Thus twisted semi-stability is a special case of
Bridgeland semi-stability.
\end{enumerate}
\end{rem}

Assume that $v$ is primitive and $(\beta,\omega)$ is general
with respect to $v$.
We fix $E_0 \in M_{(\beta,\omega)}(v)$.
Let 
$$
\Phi_{X \to \widehat{X}}^{{\bf P}}:{\bf D}(X) \to {\bf D}(\widehat{X})
$$
be the Fourier-Mukai transform by the Poincare line bundle ${\bf P}$
on $X \times \widehat{X}$,
where $\widehat{X}:=\Pic^0(X)$ is the dual of $X$.
Then we have an albanese map
${\frak a}: M_{(\beta,\omega)}(v) \to 
 X \times \widehat{X}$
by
\begin{equation*}
{\frak a}(E):=(\det(\Phi_{X \to \widehat{X}}^{{\bf P}}(E-E_0)),
\det (E-E_0)) \in X \times \widehat{X}
\end{equation*}
(\cite[Rem. 4.10]{MYY:2011:2}).
${\frak a}$ is an \'{e}tale locally trivial fibration.

\begin{defn}\label{defn:theta}
Assume that $v$ is primitive and $\langle v^2 \rangle \geq 6$.
\begin{enumerate}
\item[(1)]
$K_{(\beta,\omega)}(v)$ denotes a fiber of the albanese map
$M_{(\beta,\omega)}(v) \to X \times \widehat{X}$.
If $v$ is positive, then we also denote
a fiber of ${\frak a}:M_H^\beta(v) \to X \times \widehat{X}$ by 
$K_H^\beta(v)$.
\item[(2)]
\begin{equation}\label{eq:theta}
\theta_{v,\beta,\omega}:v^{\perp} \to H^2(M_{(\beta,\omega)}(v),{\Bbb Z})
\to H^2(K_{(\beta,\omega)}(v),{\Bbb Z})
\end{equation}
denotes the Mukai's homomorphism.
If there is a universal family
${\bf E}$ on $M_{(\beta,\omega)}(v)$, e.g.,
there is a Mukai vector $w$ with
$\langle v,w \rangle=1$,
then 
$$
\theta_{v,\beta,\omega}(x)=
c_1(p_{M_{(\beta,\omega)}(v)*}
(\ch ({\bf E})p_X^*(x^{\vee})))_{|K_{(\beta,\omega)}(v)},
$$ 
where $p_X, p_{M_{(\beta,\omega)}(v)}$ are projections
from $X \times M_{(\beta,\omega)}(v)$ to $X$ and
$M_{(\beta,\omega)}(v)$ respectively. 
\end{enumerate}
\end{defn}

\begin{thm}%[{\cite[Thm. 0.1]{Y:7},\cite[Thm.3.3.3, Rem. 3.3.4]{MYY:2011:2}}]
[{\cite[Prop. 5.16]{MYY:2011:2}}]
For $v \in H^*(X,{\Bbb Z})_{\alg}$, $M_{(\beta,\omega)}(v)$ is a smooth
projective symplectic manifold which is deformation equivalent
to $\Hilb_X^{\langle v^2 \rangle/2} \times X$.  
Assume that $\langle v^2 \rangle \geq 6$.
\begin{enumerate}
\item[(1)]
$K_{(\beta,\omega)}(v)$ is an irreducible symplectic manifold
of $\dim K_{(\beta,\omega)}(v)=\langle v^2 \rangle-2$
which is deformation equivalent to the generalized Kummer variety
constructed by Beauville \cite{B:1}.
\item[(2)]
$$
\theta_{v,\beta,\omega}:(v^\perp, \langle\;\;,\;\; \rangle)
\to (H^2(K_{(\beta,\omega)}(v),{\Bbb Z}),q_{K_{(\beta,\omega)}(v)})
$$
 is an isometry of Hodge structure. 
\end{enumerate}
\end{thm}

\section{Fourier-Mukai transforms on abelian surfaces.}\label{sect:FM}
 
\subsection{Cohomological Fourier-Mukai transforms}\label{subsect:FM}
We collect some results on the Fourier-Mukai
transforms on abelian surfaces $X$ with $\rk \NS(X)=1$.
Let $H_X$ be the ample generator of $\NS(X)$.
We shall describe the action of Fourier-Mukai transforms 
on the cohomology lattices in \cite{YY1}.
For $Y \in \FM(X)$, we have $(H_Y^2)=(H_X^2)$.
We set $n:= (H^2_X)/2$. 
In \cite[sect. 6.4]{YY1}, we constructed an isomorphism of lattices
\begin{align*}
\iota_X: 
(H^*(X,\bb{Z})_{\alg},\langle\,\,,\,\,\rangle) 
\simto
(\Sym_2(\bb{Z}, n),B),
\quad
(r,dH_X,a) 
\mapsto 
\begin{pmatrix}
 r & d\sqrt{n} \\ d\sqrt{n} & a
\end{pmatrix},
\end{align*}
where $\Sym_2(\bb{Z}, n)$ is given by 
\begin{align*}
 \Sym_2(\bb{Z},n):=
 \left\{\begin{pmatrix} x &y \sqrt{n} \\ y\sqrt{n}&z\end{pmatrix}\, 
 \Bigg|\, x,y,z\in \bb{Z}\right\},
\end{align*}
and the bilinear form $B$ on $\Sym_2(\bb{Z},n)$ is given by 
\begin{align*}
B(X_1,X_2) := 2ny_1 y_2-(x_1 z_2+z_1 x_2)
\end{align*}
for
$X_i =\begin{pmatrix}x_i & y_i \sqrt{n} \\ y_i \sqrt{n} &z_i \end{pmatrix}
 \in\Sym_2(\bb{Z},n)$
($i=1,2$).

Each $\Phi_{X \to Y}$ gives an isometry 
\begin{align}
 \iota_{Y} \circ \Phi^H_{X \to Y} \circ \iota_{X}^{-1}
\in \LieO(\Sym_2(\bb{Z}, n)),
\end{align}
where $\LieO(\Sym_2(\bb{Z}, n))$ is the isometry group
of the lattice $(\Sym_2(\bb{Z}, n),B)$. 
Thus we have a map
$$
\eta:
{\cal E} \to \LieO(\Sym_2(\bb{Z}, n))
$$
which preserves the structures of 
multiplications. 

\begin{defn}\label{defn:G}
We set
\begin{align*}
&\widehat{G} := 
\left\{
 \begin{pmatrix} a \sqrt{r} & b \sqrt{s}\\
c \sqrt{s} & d \sqrt{r}
\end{pmatrix}
%\in\GL(2,\bb{R})\,
 \Bigg|\, 
\begin{aligned}
 a,b,c,d,r,s \in \bb{Z},\, r,s>0\\ 
rs=n, \, adr-bcs = \pm1
\end{aligned}
\right\},
\\
&G := \widehat{G}\cap \SL(2,\bb{R}).
\end{align*}
\end{defn}

We have a right action $\cdot$ of $\widehat{G}$ on the lattice
$(\Sym_2(\bb{Z}, n),B)$:
 
\begin{align}
\label{eq:action_cdot}
\begin{pmatrix}
r &d  \sqrt{n}\\d  \sqrt{n}&a \end{pmatrix}
\cdot g := 
 {}^t g
 \begin{pmatrix}r &d  \sqrt{n}\\d  \sqrt{n}&a \end{pmatrix} 
 g,\;
g \in \widehat{G}.
\end{align}
Thus we have an anti-homomorphism:
$$
\alpha:\widehat{G}/\{\pm 1\} \to \LieO(\Sym_2(n,{\Bbb Z})).
$$

\begin{thm}[{\cite[Thm. 6.16, Prop. 6.19]{YY1}}]
Let
$\Phi \in \Eq_0(\bl{D}(Y),\bl{D}(X))$ be an equivalence.
\begin{enumerate}
\item[(1)]
$v_1:=v(\Phi({\cal O}_Y))$ and $v_2:=\Phi(\varrho_Y)$
are positive isotropic Mukai vectors with 
$\langle v_1,v_2 \rangle=-1$ and we can write 
\begin{equation*}
\begin{split}
& v_1=(p_1^2 r_1,p_1 q_1 H_Y, q_1^2 r_2),\quad
v_2=(p_2^2 r_2,p_2 q_2 H_Y, q_2^2 r_1),\\
& p_1,q_1,p_2,q_2, r_1, r_2 \in {\Bbb Z},\;\;
p_1,r_1,r_2 >0,\\
& r_1 r_2=n,\;\; p_1 q_2 r_1-p_2 q_1 r_2=1.
\end{split}
\end{equation*}
\item[(2)]
We set
\begin{equation*}
\theta(\Phi):=\pm
\begin{pmatrix}
p_1 \sqrt{r_1} & q_1 \sqrt{r_2}\\
p_2 \sqrt{r_2} & q_2 \sqrt{r_1}
\end{pmatrix}
\in G/\{\pm 1\}.
\end{equation*}
Then $\theta(\Phi)$ is uniquely determined by 
$\Phi$ and  
we have a map
\begin{equation*}
\theta:{\cal E} \to G/\{\pm 1\}.
\end{equation*}
\item[(3)]
The action of $\theta(\Phi)$ on $\Sym_2(n,{\Bbb Z})$
is the action of $\Phi$ on $H^*(X,{\Bbb Z})_{\alg}$: 
\begin{equation*}
\iota_X \circ \Phi(v)
=\iota_Y(v)\cdot 
\theta(\Phi).
\end{equation*}
\begin{NB}
\begin{equation*}
\iota_X \Phi\iota_Y^{-1}
\begin{pmatrix}r &d  \sqrt{n}\\d  \sqrt{n}&a \end{pmatrix} 
=
{}^t (\theta(\Phi))
\begin{pmatrix}r &d  \sqrt{n}\\d  \sqrt{n}&a \end{pmatrix} 
\theta(\Phi)
\end{equation*}
\end{NB}
Thus we have the following commutative diagram:
\begin{align}
\label{diag:groups}
\xymatrix{    
{\cal E}  \ar[d]_-{\theta} \ar[dr]_{\eta} %& \ar[d]^{\alpha}
&
%     \ar@{<-^{)}}[d]
\\
\widehat{G}/\{\pm 1\} \ar[r]_-{\alpha}
&
\LieO(\Sym_2(n,{\Bbb Z}))   
}
\end{align}

\end{enumerate}
\end{thm}
From now on, we identify the Mukai lattice
$H^*(X,{\Bbb Z})_{\alg}$ with $\Sym_2(n,{\Bbb Z})$
via $\iota_X$. 
Then for $g \in \widehat{G}$ and $v \in H^*(X,{\Bbb Z})_{\alg}$,
$v \cdot g$ means
$\iota_X(v \cdot g)=\iota_X(v) \cdot g$.    

\begin{NB}
Then the action of the cohomological FMT 
$\Phi^H : H^{\ev}(Y,\bb{Z})_{\alg}\to H^{\ev}(X,\bb{Z})_{\alg}$ 
can be written as follows (\cite[sect. 6.4]{YY}). 
Let $\widehat{H}$ be the ample generator of $\NS(Y)$. 
For a Mukai vector $v=(r,d\widehat{H},a)\in H^{\ev}(Y,\bb{Z})_{\alg}$, 
the image $\Phi^H(v)$ is given by
\begin{align*}
\Phi^H(v)=v\cdot g,\quad 
g := \begin{pmatrix}
p_1\sqrt{r_1}&q_1\sqrt{r_2}\\p_2\sqrt{r_2}&q_2\sqrt{r_1}
\end{pmatrix}\in G.
\end{align*}
Here the action $\cdot$ is given by
\begin{align}
\label{eq:action_cdot}
 (r,d H,a)\cdot g := (r',d' H,a'),\quad
 {}^t g
 \begin{pmatrix}r &d  \sqrt{n}\\d  \sqrt{n}&a \end{pmatrix} 
 g
=\begin{pmatrix}r'&d' \sqrt{n}\\d' \sqrt{n}&a'\end{pmatrix}.
\end{align}
\end{NB}

For an isotropic Mukai vector
$v=(x^2,\frac{xy}{\sqrt{n}}H_X,y^2)=x^2 e^{\frac{y}{x \sqrt{n}}H}$,
$v \cdot g=({x'}^2,\frac{x'y'}{\sqrt{n}}H_X,{y'}^2)$,
where $(x',y')=(x,y)g$.

We also need to treat the composition of a Fourier-Mukai transform 
and the dualizing functor ${\cal D}_X$.
For a
Fourier-Mukai transform
 $\Phi \in \Eq_0(\bl{D}(X),\bl{D}(Y))$,
we set
\begin{align*}
\theta(\Phi \circ \cal{D}_X)
:=
\begin{pmatrix}
1 & 0\\
0 & -1 
\end{pmatrix}
\theta(\Phi)
\in \widehat{G}/\{\pm 1\}.
\end{align*}
Then the action of 
$\theta(\Phi \circ \cal{D}_X)$ on
$\Sym_2({\Bbb Z},n)$ is the same as the action of
$\Phi \circ \cal{D}_X$.

\begin{NB}
\begin{lem}[{\cite[Lemma 6.18]{YY1}}]
\label{fct:matrix}
%\begin{enumerate}
%\item
%\begin{align*}
%\theta(\cal{D}_X)=
%\pm\begin{pmatrix} 1 & 0 \\ 0 & -1 \end{pmatrix}
%\in \GL(2,\bb{R})/\{\pm 1\}.
%\end{align*}
%We also have $\theta(\Phi_{X \to Y}^{\bl{E}}\cal{D}_X)\in
%\GL(2,\bb{R})/\{\pm 1\}$ 
%for an FMT $\Phi_{X \to Y}^{\bl{E}^\bullet} \in \Eq(\bl{D}(X),\bl{D}(Y))$.
%
%\item
If $\theta(\Phi_{X \to Y}^{\bl{E}})=
\begin{pmatrix} a & b\\ c & d \end{pmatrix}$,
then 
\begin{align*}
\theta(\Phi_{Y \to X}^{\bl{E}})=
\pm\begin{pmatrix}d & b\\c & a \end{pmatrix},\quad
\theta(\Phi_{Y \to X}^{\bl{E}^{\vee}[2]})=
\pm\begin{pmatrix}d & -b\\-c & a \end{pmatrix},\quad
\theta(\Phi_{X \to Y}^{\bl{E}^{\vee}[2]})=
\pm\begin{pmatrix}a & -b\\-c & d \end{pmatrix}.
\end{align*}
%\end{enumerate}
\end{lem}
\end{NB}

\subsection{A stabilizer subgroup.}\label{subsect:stab}

We keep the notation in subsection \ref{subsect:FM}.
In particular, we assume that $\rk \NS(X)=1$. 
Let $v:=(r,dH,a)$ be a primitive Mukai vector with $r \ne 0$.
We shall study the stabilizer of 
$\pm v \in H^*(X,{\Bbb Z})_{\alg}/\{\pm 1\}$ in 
$\widehat{G}$.
\begin{NB}
\begin{lem}
Let $X, Y \in {\Bbb Z}$ be a solution of
$X^2-(n\ell) Y^2=1$. 
We set
$$
A:=
\begin{pmatrix}
X-dn Y & -aY \sqrt{n}\\
rY \sqrt{n} & X+dn Y
\end{pmatrix}.
$$
Then
$A$ belongs to the stabilizer of $v$.
\end{lem}

\begin{proof}
\begin{equation*}
\begin{split}
& \begin{pmatrix}
X-dn Y & rY \sqrt{n}\\
-aY \sqrt{n} & X+dn Y
\end{pmatrix}
\begin{pmatrix}
r & d \sqrt{n}\\
d \sqrt{n} & a
\end{pmatrix}
\begin{pmatrix}
X-dn Y & -aY \sqrt{n}\\
rY \sqrt{n} & X+dn Y
\end{pmatrix}
\\
=& 
\begin{pmatrix}
X-dn Y & rY \sqrt{n}\\
-aY \sqrt{n} & X+dn Y
\end{pmatrix}
\begin{pmatrix}
rX & (dX+\ell Y)\sqrt{n}\\
(dX-\ell Y)\sqrt{n} & aX
\end{pmatrix}\\
=& 
\begin{pmatrix}
r(X^2-n \ell Y^2) & d(X^2-n \ell Y^2)\\
d(X^2-n \ell Y^2)& a(X^2-n \ell Y^2)
\end{pmatrix}
=(X^2-n \ell Y^2) A
\end{split}
\end{equation*}
\end{proof}

\begin{cor}
If $\sqrt{n \ell} \not \in {\Bbb Q}$, $\Stab(v)$ is infinite.
In particular, $d/r \pm \sqrt{\ell/n}$ are the accumulation
point of walls. 
\end{cor}

\end{NB}
Assume that
\begin{equation*}
\begin{pmatrix}
x  & z\\
y & w
\end{pmatrix}
\begin{pmatrix}
r & d \sqrt{n}\\
d \sqrt{n} & a
\end{pmatrix}
\begin{pmatrix}
x & y\\
z & w
\end{pmatrix}
=\epsilon
\begin{pmatrix}
r & d \sqrt{n}\\
d \sqrt{n} & a
\end{pmatrix}
\end{equation*}
and $xw-yz=\epsilon$.
\begin{NB}
\begin{equation*}
\begin{split}
& \begin{pmatrix}
x  & z\\
y & w
\end{pmatrix}
\begin{pmatrix}
r & d \sqrt{n}\\
d \sqrt{n} & a
\end{pmatrix}
\begin{pmatrix}
x & y\\
z & w
\end{pmatrix}
\\
= &
\begin{pmatrix}
r x^2+2d \sqrt{n}xz +a z^2 & rxy+d \sqrt{n}(xw+zy)+az w\\
rxy+d \sqrt{n}(xw+zy)+az w & ry^2+2d\sqrt{n}yw+a w^2
\end{pmatrix}.
\end{split}
\end{equation*}
\end{NB}
Then we have

\begin{equation}\label{eq:xyzw}
\begin{split}
& rx^2+2d\sqrt{n}xz+a z^2= \epsilon r,\\
& ry^2+2d\sqrt{n}yw+a w^2=\epsilon a,\\
& rxy+d\sqrt{n}(xw+zy)+azw=\epsilon d \sqrt{n},\\
& xw-yz=\epsilon.
\end{split}
\end{equation}
Hence 
$$
y(rx+2d\sqrt{n} z)+(az)w=
rxy+2d\sqrt{n} zy+azw=0.
$$
We note that
$$
(rx+2d\sqrt{n} z,az) \ne (0,0)
$$
by
$$
x(rx+2d\sqrt{n} z)+z(az)=\epsilon r \ne 0.
$$
We set 
$y:=-\lambda az$ and $w:=\lambda(rx+2d\sqrt{n} z)$.
Then 
$$
\epsilon =xw-yz=\lambda(rx^2+2d\sqrt{n} xz+az^2)=\lambda \epsilon r.
$$
Hence $\lambda=1/r$.
Therefore 
\begin{equation}\label{eq:yz}
y=-\frac{a}{r}z,\; w=x+2d\sqrt{n}\frac{z}{r}.
\end{equation}
Conversely for $x,z$ with 
\begin{equation}\label{eq:xz}
rx^2+2d\sqrt{n} xz+az^2=\epsilon r,
\end{equation}
we define $y,w$ by \eqref{eq:yz}.
Then \eqref{eq:xyzw} are satisfied.
We note that \eqref{eq:xz} is written as
\begin{equation*}
(x+d \sqrt{n} \tfrac{z}{r})^2-\ell (\tfrac{z}{r})^2=\epsilon.
\end{equation*} 
We set $X:=x+d \sqrt{n} \tfrac{z}{r}$ and 
$Z:=\frac{z}{r}$. Then
\begin{equation}\label{eq:det}
X^2-\ell Z^2=\epsilon
\end{equation}
and 
\begin{equation*}
\begin{pmatrix}
x & y \\
z & w
\end{pmatrix}
=
\begin{pmatrix}
X-d\sqrt{n}Z & -a Z\\
rZ & X+d\sqrt{n} Z.
\end{pmatrix}
=XI_2+ZF,
\end{equation*}
where
\begin{equation}
I_2:=
\begin{pmatrix}
1 & 0\\
0 & 1
\end{pmatrix},\;
F:=
\begin{pmatrix}
-d \sqrt{n} & -a\\
r & d \sqrt{n}
\end{pmatrix}.
\end{equation}
We have $F^2=\ell I_2$. 
We set
\begin{equation*}
\begin{split}
\Stab_0(v):= &\{g \in \widehat{G} \mid g(v)=(\det g) v \}\\
=& \{g \in \widehat{G} \mid g=XI_2+ZF \}.
\end{split}
\end{equation*}
$\Stab_0(v)$ is a normal subgroup of 
$\Stab(v)$ of index 2.
Indeed for $g \in \Stab(v)$,
we have $g(v)=\eta(g) (\det g) v$ and 
$\eta(gg')=\eta(g)\eta(g')$, $g, g' \in \Stab(v)$.
Thus $\ker \eta=\Stab_0(v)$.  
We get a homomorphism
\begin{equation*}
\begin{matrix}
\varphi:& \Stab_0(v) & \to & {\Bbb R}\\
& XI_2+ZF & \mapsto & X+Z \sqrt{\ell}.
\end{matrix}
\end{equation*}  

\begin{prop}\label{prop:stab}
Assume that $\sqrt{n \ell} \not \in {\Bbb Q}$. 
\begin{enumerate}
\item[(1)]
For $X I_2+ZF \in \Stab_0(v)$,
$X+Z \sqrt{\ell}$ is an algebraic integer such that
$(X+Z\sqrt{\ell})^2 \in {\Bbb Q}(\sqrt{n\ell})$.
\item[(2)]
$\im \varphi \cong {\Bbb Z} \oplus {\Bbb Z}_2$. 
\item[(3)]
$\varphi$ is injective or $\ker \varphi=\{\sqrt{\ell}^{-1} F \}$
if $\sqrt{\ell}^{-1}F \in \Stab_0(v)$.
In particular, if $n=1$ and $\ell>1$, then
$\varphi$ is injective. 
\end{enumerate}
\end{prop}

\begin{proof}
We set $\alpha:=X+Z\sqrt{\ell}$.
Assume that $XI_2+ZF \in \widehat{G}$. Then
$$
x^2, y^2, xw,yz,
\frac{xy}{\sqrt{n}}, \frac{xz}{\sqrt{n}}, 
\frac{yw}{\sqrt{n}}, \frac{zw}{\sqrt{n}}, \in {\Bbb Z}.
$$
Hence
\begin{equation*}
\begin{split}
2(X^2+d^2 n Z^2)= & x^2+w^2 \in {\Bbb Z},\\
X^2-d^2 nZ^2=& xw \in {\Bbb Z},\\
r^2 \frac{XZ}{\sqrt{n}}= & r\frac{xz}{\sqrt{n}}+dz^2 \in {\Bbb Z},
\end{split}
\end{equation*}
which impy that
\begin{equation}\label{eq:XZ-quadratic}
X^2+\ell Z^2, \frac{XZ}{\sqrt{n}} \in {\Bbb Q}.
\end{equation}
We note that $\alpha$ satisfies the equation
\begin{equation*}
\alpha^2-2X \alpha+\epsilon=0.
\end{equation*}
Since $2X=x+w$ is an algebraic integer, $\alpha$ is an algebraic integer.
By \eqref{eq:XZ-quadratic},
$$
\alpha^2=(X^2+\ell Z^2)+2\frac{XY}{\sqrt{n}} \sqrt{n \ell}
\in {\Bbb Q}(\sqrt{n \ell}).
$$
Thus (1) holds.

%The proof of (2) is similar to that of
%Corollary 5.6.
(2)
We first prove that
$\im \varphi \ne \{ \pm 1 \}$.
We take a solution $(p,q) \in {\Bbb Z}^{\oplus 2}$ of 
$p^2-n\ell q^2=1$ such that $q \ne 0$.
We set $X:=p$ and $Z:=q \sqrt{n}$.
Then 
\begin{equation}
X I_2+ZF=
\begin{pmatrix}
p-dnq & -aq \sqrt{n}\\
r q \sqrt{n} & p+dnq
\end{pmatrix}
\end{equation}
satisfies all the requirements.
Therefore $\im \varphi \ne \{ \pm 1 \}$.
By the Dirichlet unit theorem,
the torsion part of $\im \varphi$ is 
$\{ \pm 1 \}$. 
Since $\alpha^2$ is a unit of 
the ring of integers of ${\Bbb Q}(\sqrt{n\ell})$,
the rank of $\im \varphi$ is 1, which implies the 
claim.

If $\alpha=X+Z\sqrt{\ell}=1$, then
$\alpha^2=1$ implies that $XZ=0$. 
If $Z=0$, then $X=1$. If $X=0$, then $Z\sqrt{\ell}=1$.
Therefore the first part of the claims holds.

Assume that $n=1$ and $\ell>1$.
Then 
$$
\frac{1}{\sqrt{\ell}}F=
\frac{1}{\sqrt{\ell}}
\begin{pmatrix}
-d \sqrt{n} & -a \\
r & d \sqrt{n}
\end{pmatrix}.
$$
Hence $\ell \mid a^2,\ell \mid r^2,\ell \mid d^2 n$.
Since $v$ is primitive, $a^2,r^2,d^2$ are
relatively prime.
Hence $\ell \mid n$, which is a contradiction.
Therefore the second part also holds.  
\end{proof}

We set

\begin{equation*}
\Stab_0(v)^*:=
\left\{
\left.
\begin{pmatrix}
x & y\\
z & w
\end{pmatrix}
\in \Stab_0(v) \right| xw-yz=1,\; y \in \sqrt{n}{\Bbb Z}
\right\}.
\end{equation*}
All elements of $\Stab_0(v)^*$ come from
autoequivalences of ${\bf D}(X)$
(see Lemma \ref{lem:n} below).
$\Stab_0(v)/\Stab_0(v)^*$ is a finite group of type
$({\Bbb Z}/2{\Bbb Z})^{\oplus k}$.

If
\begin{equation*}
A:=\begin{pmatrix}
x & y \\
z & w
\end{pmatrix}
\in \Stab(v) \setminus \Stab_0(v),
\end{equation*}
then

\begin{equation*}
A
=
\begin{pmatrix}
x & \frac{1}{r}(az+2d \sqrt{n} x)\\
z & -x
\end{pmatrix}.
\end{equation*}
In particular, $A^2= \pm I_2$.

\begin{ex}
Assume that $n=1$.
Then 
$$
XI_2+Z F \in \GL(2,{\Bbb Z})=\widehat{G}
\Longleftrightarrow 
X \pm dZ, aZ,rZ \in {\Bbb Z}.
$$
Assume that $2 \mid r$ and $2 \mid a$.
Then the primitivity of $v$ implies that
$2 \nmid d$.
Hence $\ell=d^2-ra \equiv 1 \mod 4$.
Then ${\frak O}:={\Bbb Z}[\tfrac{1+\sqrt{\ell}}{2}]$
is the ring of integers.
We note that
$X \pm dZ, aZ,rZ \in {\Bbb Z}$ imply that
$2dZ, aZ,rZ \in {\Bbb Z}$.
Since $\gcd(r/2,a/2,d)=1$,
we have $2Z \in {\Bbb Z}$.
Then $X-dZ \in {\Bbb Z}$ implies that
$X-Z \in {\Bbb Z}$. Therefore 
$X+Z \sqrt{\ell} \in {\frak O}$ with 
\eqref{eq:det}.
Conversely for 
$X+Z \sqrt{\ell} \in {\frak O}$ with \eqref{eq:det},
we have $X \pm dZ, aZ,rZ \in {\Bbb Z}$.
Therefore the fundamental unit of ${\frak O}$
is the generator of $\im \varphi$.
\end{ex}

\subsection{The case where $\rk \NS(X) \geq 2$.}\label{subsect:FM:rank2}

Assume that $\rk \NS(X) \geq 2$.
Let $v=(r,\xi,a)$ be a Mukai vector with
$\langle v^2 \rangle=2\ell$.
By using Proposition \ref{prop:stab},
we shall construct many autoequivalences
preserving $v$. 
Assume that $\xi \in \Amp(X)$ and set $\xi=dH$,
where $H$ is a primitive and ample divisor.
Let $L:={\Bbb Z} \oplus {\Bbb Z}H \oplus {\Bbb Z}\varrho_X$ 
be a sublattice of $H^*(X,{\Bbb Z})_{\alg}$. 
\begin{lem}\label{lem:n}
Let $v_0=(p^2 n,pq H,q^2)$ be a primitive and isotropic
Mukai vector with $n=(H^2)/2$ and $\gcd(pn,q)=1$.
\begin{enumerate}
\item[(1)]
$M_H(v_0) \cong X$.
\item[(2)]
For an isotropic vector $v_1 \in H^*(X,{\Bbb Z})_{\alg}$ with
$\langle v_0,v_1 \rangle=1$,
there is a Fourier-Mukai transform
$\Phi:{\bf D}(X) \to {\bf D}(X)$
such that $\Phi(-\varrho_X)=v_0$ and $\Phi(v({\cal O}_X))
=v_1$.
Moreover we have the following.
\begin{enumerate}
\item
$\Phi$ is unique up to
the action of $\Aut(X) \times \Pic^0(X) \times 2{\Bbb Z}$,
where $2k \in  2{\Bbb Z}$ acts as the shift
functor $[2k]:{\bf D}(X) \to {\bf D}(X)$. 
\item
If $v_1 \in L$, then
we can take $\Phi$ such that $\Phi(L)=L$
and $\Phi_{|L^{\perp}}$ is the identity.
\end{enumerate}
\end{enumerate}
\end{lem}

\begin{proof}
(1)
\begin{NB}
If $\NS(X)={\Bbb Z}H$, the claim is shown in \cite[Lem. 7.3]{YY1}.
It is easy to see that
the same proof works for our assumption.
\end{NB}
We fix a stable sheaf $E$ with $v(E)=v_0$.
Let ${\bf P}$ be the Poincar\'{e} line bundle on $X \times \Pic^0(X)$.
Then we have a surjective homomorphism
$\Pic^0(X) \to M_H(v_0)$ by sending $y \in \Pic^0(X)$
to $E \otimes {\bf P}_{|X \times \{ y\}} \in M_H(v_0)$
(\cite{Mu0}).
So we get an isomorphism
$\Pic^0(X)/\Sigma(E) \to M_H(v_0)$,
where 
$$
\Sigma(E):=\{y \in \Pic^0(X) \mid E 
\otimes {\bf P}_{|X \times \{ y \}}
\cong E \}.
$$
Let $T_x:X \to X$ be the translation by $x$.
For a divisor $D$ on $X$,
$\phi_{D}:X \to \Pic^0(X)$ denotes the homomorphism
such that $\phi_D(x)=T^*_x {\cal O}_X(D) \otimes 
{\cal O}_X(-D)$.
We set $K(D):=\ker \phi_D$.
If $(D^2)>0$, then $\phi_D$ is finite and 
$\# K(D)=d^2$, where $d:=(D^2)/2$.
For $D=pqH$,
\cite[Thm. 7.11]{Mu0} implies that
$\phi_{pq H}(X_{p^2 n})=\Sigma(E)$, where 
$X_m$ denotes the set of $m$-torsion points of $X$, 
which is the kernel
of $m:X \to X$
the multiplication by $m \in {\Bbb Z}_{>0}$.
Hence we have a morphism $\phi:X \to M_H(v_0)$
such that 
$$
\phi(x)=E \otimes T_x^* ({\cal O}_X(pqH)) \otimes 
{\cal O}_X(-pqH),\; x \in X
$$
 and $\phi$ induces
an isomorphism 
$$
X/(X_{p^2 n}+K(pqH)) \cong \widehat{X}/\Sigma(E) \cong M_H(v_0).
$$
Since $K(pq H)=(pq)^{-1}(K(H))$,
$n(K(H))=0$ and
$(pn,q)=1$, we have a sequence of isomorphisms
$$
X/(X_{p^2 n}+K(pqH)) \overset{p^2n}{\to} X/p^2 nK(pqH) 
=X/q^{-1}(0)
\overset{q}{\to} X.
$$
\begin{NB}
We note that
$K(pqH)=\frac{1}{pq}V_1/V_1 \oplus \frac{1}{pqn}V_2/V_2$,
where $V_1$ and $V_2$ are isomorphic to ${\Bbb Z}^{\oplus 2}$.
Then $p^2 n K(pqH)=
\frac{p^2 n}{pq}V_1/V_1 \oplus \frac{p^2 n}{pqn}V_2/V_2
=\frac{p n}{q}V_1/V_1 \oplus \frac{p}{q}V_2/V_2=
\frac{1}{q}V_1/V_1 \oplus \frac{1}{q}V_2/V_2=X_{q}$.
\end{NB}
Therefore $M_H(v_0) \cong X$.

\begin{NB}
Wrong argument:
Assume that $v_0=(n,qH,q^2)$.
Let ${\bf P}$ be the Poincar\'{e} line bundle on $\Pic^0(X) \times X$.
We have a morphism $\pi:\Pic^0(X) \to X$ such that
$\pi^*(E)=L^{\oplus n}$ and $\pi_*(L)=E$.
\begin{NB2}
It is not correct.
\end{NB2}
Let $G$ be the kernel of $\pi$.
Then ${\bf E}:=\pi_*({\bf P}\otimes L)$ is a flat family
of stable sheaves with $v({\bf E}_{|X \times \{ x \}})=v_0$. 
We note that $\Hom({\bf E}_{|X \times \{ x \}},{\bf E}_{|X \times \{ y \}})
=\Hom(\pi^*({\bf E}_{|X \times \{ x \}}),
{\bf P}_{|X \times \{ y \}} \otimes
L)^G
=\Hom({\bf P}_{|X \times \{ x \}}^{\oplus n},
{\bf P}_{|X \times \{ y \}})^G=0$ if $x \ne y$.
Hence $X \to M_H(v_0)$ is isomorphic. 

In general, for $v_0=(p^2 s,pqH,q^2 t)$,
we take a covering $\pi:Y \to X$ such that $\pi^*(E)=L^{\oplus p^2 s}$.
Let ${\bf P}$ be the Poincar\'{e} line bundle on
$Y \times \Pic^0(Y)$.
Then $\pi_*({\bf P} \otimes p_1^*(L))$ parameterizes
stable sheaves with the Mukai vector $v_0$
and $\Pic^0(Y) \to M_H(v_0)$ is isomorphic.
\end{NB}

(2) By our assumption, we have a universal family
${\bf E}$ on $X \times M_H(v_0)$. By (1), we have an isomorphism
$$
X \times M_H(v_0) \to X \times X.
$$ 
Thus we may assume that
${\bf E} \in \Coh(X \times X)$.
Then we have an equivalence 
$$
\Phi:{\bf D}(X) \to {\bf D}(X)
$$
such that $\Phi({\Bbb C}_x)={\bf E}_{|X \times \{ x\}}[1]$
for $x \in X$.
%Let $E_1$ be a stable sheaf with $v(E_1[k])=v_1$ $(k \in {\Bbb Z})$. 
Since $\langle \Phi^{-1}(v_1),\Phi^{-1}(v_0) \rangle=1$,
$\rk \Phi^{-1}(v_1)=1$.
Let $p_2:X \times X \to X$ be the second projection.
Replacing ${\bf E}$ by ${\bf E} \otimes p_2^*(L)$ $(L \in \Pic(X))$,
we have $\Phi^{-1}(v_1)=v({\cal O}_X)$.
Thus the first claim holds.

(a) If $\Phi':{\bf D}(X) \to {\bf D}(X)$
also satisfies the same properties, then
$\Phi^{-1}\Phi'(\varrho_X)=\varrho_X$ and
$\Phi^{-1}\Phi'(v({\cal O}_X))=v({\cal O}_X)$.
Hence $\Phi^{-1}\Phi'$
is the Fourier-Mukai transform
whose kernel is ${\cal O}_{\Gamma} \otimes p_2^*(N)[2k]$,
where $N \in \Pic^0(X)$ and $\Gamma$ is the graph of
$g \in \Aut(X)$.
Hence 
$$
\Phi'(E)=\Phi(g_*(E \otimes N))[2k],\;E \in {\bf D}(X).
$$ 

(b)
We take a complex $E_1 \in{\bf D}(X)$ with $v(E_1)=v_1$.
For a $S$-flat coherent sheaf ${\bf F}$ on $X \times S$
with ${\bf F}_{|X \times \{s \}} \in M_H(v_0)$ ($s \in S$),
we set 
$$
{\cal L}_{\bf F}:=\det p_{S_*}({\bf F}^{\vee} \otimes E_1).
$$
If we replace ${\bf F}$ by ${\bf F} \otimes {\cal L}_{\bf F}^{\vee}$,
then we have ${\cal L}_{\bf F}={\cal O}_S$.  
We set ${\bf E}_1:=E \otimes {\bf P}$.
Then by the identification $X \cong M_H(v_0)$ in the proof of (1),
we have
$$
(1 \times \phi_{pqH})^*
({\bf E}_1 \otimes {\cal L}_{{\bf E}_1}^{\vee})
\cong (1 \times p^2 nq)^*({\bf E}),
$$
where 
${\bf E}$ is the universal family in (2) which is normalized to
satisfy ${\cal L}_{{\bf E}} \cong {\cal O}_X$.
Indeed $\Phi(v({\cal O}_X))=v(E_1)$ implies that
$c_1({\cal L}_{{\bf E}})=-c_1(\Phi_{X \to X}^{{\bf E}^{\vee}[1]}(E_1))
=-c_1(v({\cal O}_X))=0$. Hence 
${\bf E}':={\bf E} \otimes {\cal L}_{\bf E}^{\vee}$ also satisfies 
$v(\Phi_{X \to X}^{{\bf E}'[1]}({\cal O}_X))=v_1$.

We set $\beta:=\frac{q}{pn}H$.
For $G \in {\bf D}(X)$ with 
$$
v(G)=r e^\beta+a \varrho_X+dH+D
+(dH+D,\beta)\varrho_X,\;D \in H^\perp,
$$
Lemma \ref{lem:action} below implies that
\begin{equation}\label{eq:p^2 qn}
\begin{split}
(p^2 qn)^*(v(\Phi_{X \to X}^{{\bf E}^{\vee}}(G)))=&
\phi_{pqH}^*(v({\bf R}p_{\Pic^0(X)*}({\bf E}_1^{\vee} \otimes G)
\otimes {\cal L}_{{\bf E}_1}))\\
=& \phi_{pqH}^*
(v(\Phi_{X \to \Pic^0(X)}^{{\bf P}^{\vee}}(E^{\vee} \otimes G)
\otimes {\cal L}_{{\bf E}_1}))\\
=& 
(p^2 qn)^*(p^2 na+\tfrac{r}{p^2 n}\varrho_X-dH+D)
\otimes \phi_{pqH}^*(v({\cal L}_{{\bf E}_1})).
\end{split}
\end{equation}
In order to complete the proof of the claim,
we need to show that
$\phi_{pqH}^*({\cal L}_{{\bf E}_1}) \in (p^2 qn)^*(L)$.
\begin{NB}
For $D \in H^\perp$,
$D e^{\lambda H}=D$, $\lambda \in {\Bbb Q}$.
\end{NB}
We take integers $x,y$ with $y pn-xq= \pm 1$.
Then we have $v_1=(x^2,xyH,y^2 n)$. 
Hence 
$$
v_1=x^2 e^\beta \pm \frac{x}{pn}(H+(H,\beta)\varrho_X)+
\frac{1}{p^2 n}\varrho_X.
$$
Applying \eqref{eq:p^2 qn} to 
$v({\cal O}_X)=\Phi_{X \to X}^{{\bf E}^{\vee}}(v_1)$, we have
\begin{equation*}
\begin{split}
(p^2 qn)^*(v({\cal O}_X))
=&
(p^2 qn)^*(1+\tfrac{x^2}{p^2 n}\varrho_X \mp \frac{x}{pn}H)
\otimes \phi_{pqH}^*(v({\cal L}_{{\bf E}_1}))\\
=& (p^2 qn)^*(v({\cal O}_X(\mp\frac{x}{pn}H)))
\otimes \phi_{pqH}^*(v({\cal L}_{{\bf E}_1})).
\end{split}
\end{equation*}
Therefore the claim holds.

\begin{NB}

Let $({\cal X},{\cal H}) \to T$ be a general deformation
of $(X,H)$ such that $({\cal X}_0,{\cal H}_0)=(X,H)$. 
Under the assumption, we have a family of 
Mukai vectors $v_0,v_1 \in L$.
Hence we have a family of Fourier-Mukai transforms
$\Phi_t:{\bf D}({\cal X}_t) \to {\bf D}({\cal X}_t)$
such that $\Phi_0=\Phi$.
For a point $t \in T$ with $\NS({\cal X}_t)={\Bbb Z}{\cal H}_t$,
we have $\Phi_t(L)=L$.
Since the cohomological correspondence $\Phi_t$
is independent of $t$, we have $\Phi_0(L)=L$.
\end{NB}
\end{proof}

Let $m:X \times X \to X$ be the addition map.
Then
$$
m^*({\cal O}_X(pq H))
\otimes p_1^*({\cal O}_X(-pq H)) \otimes
 p_2^*({\cal O}_X(-pq H))) \cong (1 \times \phi_{pq H})^*({\bf P}).
$$
We shall compute 
$\phi_{pqH}^*(\Phi_{X \to \Pic^0(X)}^{{\bf P}^{\vee}}(w))$
for $w \in H^*(X,{\Bbb Z})$.

\begin{lem}\label{lem:action}
We set $\beta:=\frac{q}{pn}H$.
For 
$$
u:=re^\beta+a \varrho_X +(dH+D)+(dH+D,\beta)\varrho_X,\;
D \in \NS(X)_{\Bbb Q} \cap H^\perp,
$$  
we have
$$
\phi_{pqH}^*(\Phi_{X \to X}^{{\bf P}^{\vee}}(v_0^{\vee} u))=
(p^2 qn)^*(p^2 na+\tfrac{r}{p^2 n}\varrho_X-dH+D).
$$
\end{lem}

\begin{proof} 
Let $e_1,e_2,e_3,e_4$ be a basis of $H^1(X,{\Bbb Z})$
such that $c_1(H)=e_1 \wedge e_2+n e_3 \wedge e_4$.
In 
$$
H^*(X \times X,{\Bbb Z}) =H^*(X,{\Bbb Z}) \otimes H^*(X,{\Bbb Z}),
$$
we identify $e_i \otimes 1$ with $e_i$ and  
denote $1 \otimes e_i$ by $f_i$.
Then 
\begin{equation*}
\begin{split}
& c_1(m^*({\cal O}_X(H)) \otimes p_1^*({\cal O}_X(-H))
 \otimes p_2^*({\cal O}_X(-H)))\\
= &(e_1+f_1) \wedge (e_2+f_2)+n (e_3+f_3) \wedge (e_4+f_4)\\
& \quad -(e_1 \wedge e_2+n e_3 \wedge e_4)-
(f_1 \wedge f_2+n f_3 \wedge f_4)\\
= & e_1 \wedge f_2+f_1 \wedge e_2 +n (e_3 \wedge f_4
+f_3 \wedge e_4)=: \eta.
\end{split}
\end{equation*}
We denote the class by $\eta$.
Then we see that
\begin{equation*}
\begin{split}
\frac{\eta^2}{2!}= & 
-(e_3 \wedge e_4)^* \wedge (f_1 \wedge f_2)-n^2
(e_1 \wedge e_2)^* \wedge (f_3 \wedge f_4)\\
& +n((e_2 \wedge e_4)^* \wedge (f_2 \wedge f_4)+
(e_1 \wedge e_4)^* \wedge (f_1 \wedge f_4)\\
& +(e_2 \wedge e_3)^* \wedge (f_2 \wedge f_3)+
(e_1 \wedge e_3)^* \wedge (f_1 \wedge f_3)),\\
\frac{\eta^4}{4!}=&
n^2 (e_1 \wedge e_2 \wedge e_3 \wedge e_4) \wedge
(f_1 \wedge f_2 \wedge f_3 \wedge f_4),
\end{split}
\end{equation*}
where $\{ (e_i \wedge e_j)^* \mid i,j  \}$ is the dual basis of 
$\{ e_i \wedge e_j \mid i,j\}$ via the intersection pairing.
We note that $H^\perp$ is generated by
$$
e_1 \wedge e_2-n e_3 \wedge e_4,\;e_2 \wedge e_4,\;
e_1 \wedge e_4,\; e_2 \wedge e_3,\; e_1 \wedge e_3. 
$$
Then we see that
\begin{equation*}
\begin{split}
&\phi_{pq H}^*(\Phi_{X \to \Pic^0(X)}^{{\bf P}^{\vee}}(v_0^{\vee} u))\\
=& p_{2*}(p_1^*(p^2 n e^{-\beta}
(r e^\beta+a \varrho_X+(dH+D+(dH+D,\beta)\varrho_X))
e^{pq \eta})\\
=&
p_{2*}(p_1^*(p^2 n (r +a \varrho_X+dH+D))
e^{pq \eta})\\
=& p^2 n^2 (pq)^2(-dH+D)+p^2 n a+p^6 n^3 q^4 r \varrho_X. \\
\end{split}
\end{equation*}
\begin{NB}
$$
= p^2 n (pq)^2 (-y f_1 \wedge f_2-x n^2 f_3 \wedge f_4+n z)
+p^2 n a+p^6 n^3 q^4 r \varrho_X, 
$$
where
$\xi=x e_1 \wedge e_2+y e_3 \wedge e_4+z$, $(z,e_1 \wedge e_2)=
(z,e_3 \wedge e_4)=0$.
Thus the action on $z$ is multiplication by
$(p^2 qn)^2$, which is the same as the action of
$p^2 qn:X \to X$.
For $H$, we have $-(p^2 qn)^2 H$.
For $e_1 \wedge e_2-n e_3 \wedge e_4 \in H^\perp$, we also
have $(p^2 qn)^2(e_1 \wedge e_2-n e_3 \wedge e_4)$.
\end{NB}
\begin{NB}
If $(r,a,\xi)=(p^2n,0,0)$,
we get $v_0 \mapsto (p^2 qn)^4 \varrho_X$.
\end{NB}
Since $(p^2 qn)^*(x_i)=(p^2 qn)^{2i} x_i$
for $x_i \in H^{2i}(X,{\Bbb Q})$, we get the claim.
\end{proof}

By Lemma \ref{lem:n},
every member of $\Stab_0(v)^*$ comes from an autoequivalence
of ${\bf D}(X)$ and
we have a homomorphism
$$
\Stab_0(v)^* \to \LieO(H^*(X,{\Bbb Z})_{\alg}).
$$

\begin{NB}
By Lemma \ref{lem:n} (1),
there is an autoequivalence $\Phi$ of ${\bf D}(X)$
such that $\Phi(-\varrho_X)=-(z^2,zw/\sqrt{n} H,z^2)$
and $\Phi(1)=(x^2,xy/\sqrt{n} H,y^2)$.
Then (2) implies $L$ is preserved.
Since $\Phi_{|L}$ is an isometry which preserves the orientation
(\cf \cite[Lem. 6.3]{YY1}),
$\Phi(H)$ is uniquely determined.
Thus $\Phi_{|L}$ coincides with the action 
by the matrix $\begin{pmatrix} x & y\\ z & w \end{pmatrix}$.
\end{NB}

\begin{NB}
If $\Phi(\varrho_X)=\varrho_X$ and $\Phi(v({\cal O}_X))=v({\cal O}_X)$,
the kernel ${\bf E}$ of $\Phi$
is a family of point sheaves up to shift $[2k]$.
Hence $\Phi$ is the action of $X \times \Pic^0(X)$ up to
shift $[2k]$.
In particular, the cohomological action of
$\Phi$ is trivial.
\end{NB}

\begin{NB}
We set $n:=(H^2)/2$.
Then 
$$
S_H(v):=\left\{
\left.
\begin{pmatrix}
p-dnq & -aq \sqrt{n}\\
rq \sqrt{n} & p+dnq 
\end{pmatrix}
\right| p,q \in {\Bbb Z},
p^2-n\ell q^2=1
\right\}
$$
is a subgroup of $G_0$ which acts trivially on $v$.   
\end{NB}

\begin{lem}\label{lem:S_H}
For a Mukai vector $v=(r,\xi,a)$, assume that
$(\xi^2)>0$. We set $\xi:=dH$, where $H$ is a primitive ample divisor
and $d \in {\Bbb Z}$.
If $\sqrt{\langle v^2 \rangle/(\xi^2)} \not \in {\Bbb Q}$,
then there is an autoequivalence $\Phi:{\bf D}(X) \to {\bf D}(X)$
such that $\Phi$ acts on $L:={\Bbb Z} \oplus {\Bbb Z}H \oplus
{\Bbb Z}\varrho_X$ 
as an isometry of infinite order
and $\Phi(v)=v$.
\end{lem}

\begin{proof}
By the proof of Proposition \ref{prop:stab} (2),
$\Stab_0(v)^*$ contains an element $g$ of infinite order.
By Lemma \ref{lem:n},
we have a Fourier-Mukai transform
$\Phi:{\bf D}(X) \to {\bf D}(X)$ inducing the action of $g$ on $L$.
\end{proof}

\section{The space of stability conditions
 and the positive cone of the moduli spaces.}\label{sect:positive-cone}

\subsection{A polarization of $M_{(\beta,tH)}(v)$.}
  \label{subsect:polarization}
We shall explain a natural ample line bundle on the 
moduli space $M_{(\beta,tH)}(v)$, which is introduced in 
\cite{MYY:2011:2} and \cite{BM}.

For $v=(r,\xi,a)$, 
we set
\begin{equation*}
%\begin{split}
d_\beta:= \frac{(\xi-r \beta,H)}{(H^2)},\;
a_\beta:= -\langle e^\beta,v \rangle.
%\end{split}
\end{equation*}
Then 
$$
v=re^\beta+a_\beta \varrho_X+d_\beta H+D_\beta+
(d_\beta H+D_\beta,\beta)\varrho_X,\;
D_\beta \in \NS(X)_{\Bbb Q} \cap H^\perp.
$$
Let 
\begin{equation*}
\begin{split}
\xi_\omega:=& \frac{(\omega^2)}{2 d_\beta}
\left(r(H+(H,\beta)\varrho_X)+d_\beta (H^2)\varrho_X \right)\\
& \quad -
\frac{1}{d_\beta}
 \left(a_\beta(H+(H,\beta)\varrho_X)+d_\beta (H^2)e^\beta \right)
\in H^*(X,{\Bbb Z})_{\alg} \otimes {\Bbb R}
\end{split}
\end{equation*}
be a vector in \cite[Defn. 5.12]{MYY:2011:2}, where
$\omega=tH$.

%By taking a norm on $\NS(X)_{\Bbb R}$, 
%we regard $C(\overline{\Amp(X)}_{\Bbb R})$ as a subset
%of $\overline{\Amp(X)}_{\Bbb R}$.
\begin{defn}
For $(\beta,H,t) \in \NS(X)_{\Bbb R} \times
\overline{\Amp(X)}_{\Bbb R} \times {\Bbb R}_{\geq 0}$,
we set 
\begin{equation*}
\begin{split}
\xi(\beta,H,t):=&
d_\beta \xi_\omega\\
=&
\left(r\frac{t^2 (H^2)}{2}+\langle e^\beta,v \rangle \right)
(H+(\beta,H)\varrho_X)
-(\xi-r \beta,H)\left(e^\beta-\frac{t^2(H^2)}{2}\varrho_X \right).
\end{split}
\end{equation*} 
%\begin{equation*}
%\xi(\beta,H,t):=
%r \left( \frac{t^2(H^2)-((\beta-\delta)^2)}{2}+
%\frac{\langle v^2 \rangle}{2r^2} \right)
%\left(H+(\delta,H)\varrho_X \right)
%-r(\delta-\beta,H)
%\left(e^\beta-\frac{a_\beta}{r}\varrho_X \right).
%\end{equation*}
\end{defn}

Assume that $r \ne 0$.
We set $\delta:=\frac{\xi}{r}$.
Then $v=r e^\delta+a_\delta \varrho_X$.
For $\beta=\delta+sH+D$ with $D \in \NS(X)_{\Bbb R} \cap H^{\perp}$,
we set
\begin{equation*}
%\begin{split}
d_\beta=
\frac{r(\delta-\beta,H)}{(H^2)},\;
a_\beta=
a_\delta+\frac{((\beta-\delta)^2)}{2}r.
%\end{split}
\end{equation*}
\begin{NB}
Since 
$v=r e^\beta+(d_\beta H+D_\beta+(d_\beta H+D_\beta,\beta)\varrho_X)
+a_\beta \varrho_X$,
$\delta=\beta+(d_\beta H+D_\beta)/r$.
Hence $D=-D_\beta/r$.
\end{NB}
\begin{equation*}
\begin{split}
\xi(\beta,H,t)=&
r \left( \frac{t^2(H^2)-((\beta-\delta)^2)}{2}+
\frac{\langle v^2 \rangle}{2r^2} \right)
\left(H+(\delta,H)\varrho_X \right)
-r(\delta-\beta,H)
\left(e^\beta-\frac{a_\beta}{r}\varrho_X \right)\\
=& r \left( \frac{(\omega^2)-((\beta-\delta)^2)}{2}+
\frac{\langle v^2 \rangle}{2r^2} \right)
\left(H+(\delta,H)\varrho_X \right)\\
& -r(\delta-\beta,H)
\left((\beta-\delta+(\beta-\delta,\delta)\varrho_X)
+\left(e^\delta-\frac{a_\delta}{r}\varrho_X \right) \right).
\end{split}
\end{equation*}

\begin{NB}
Let 
\begin{equation*}
\xi_\omega:=\frac{r\frac{(\omega^2)}{2}-a_\beta}{d_\beta}
\left(H+\left(H,\frac{\xi}{r}\right)\varrho_X \right)-
(H^2)\left(e^\beta-\frac{a_\beta}{r}\varrho_X \right)
\in H^*(X,{\Bbb Z})_{\alg} \otimes {\Bbb R}
\end{equation*}
be a vector in \cite[sect. 5.3]{MYY:2011:2}.
Then we have
\begin{equation*}
\begin{split}
 d_\beta \xi_\omega=&
r \left( \frac{(\omega^2)-((\beta-\delta)^2)}{2}+
\frac{\langle v^2 \rangle}{2r^2} \right)
\left(H+(\delta,H)\varrho_X \right)
-r(\delta-\beta,H)
\left(e^\beta-\frac{a_\beta}{r}\varrho_X \right)\\
=& r \left( \frac{(\omega^2)-((\beta-\delta)^2)}{2}+
\frac{\langle v^2 \rangle}{2r^2} \right)
\left(H+(\delta,H)\varrho_X \right)\\
& -r(\delta-\beta,H)
\left((\beta-\delta+(\beta-\delta,\delta)\varrho_X)
+\left(e^\delta-\frac{a_\delta}{r}\varrho_X \right) \right).
\end{split}
\end{equation*}
\end{NB}

Assume that $r=0$.
%For $v=(0,\xi,a)=a_\beta \varrho_X+
%(d_\beta H+D_\beta+(d_\beta H+D_\beta,\beta)\varrho_X)$, 
Then we also have
\begin{equation*}
\begin{split}
\xi(\beta,H,t)=& -a_\beta (H+(H,\beta)\varrho_X)
-d_\beta (H^2)\left(e^\beta-\frac{t^2(H^2)}{2}\varrho_X \right)\\
=& ((\xi,\beta)-a)(H+(H,\beta)\varrho_X)-
(\xi,H)\left(e^\beta-\frac{t^2(H^2)}{2}\varrho_X \right).
\end{split}
\end{equation*}

\begin{lem}\label{lem:Im}
\begin{equation*}
{\Bbb R}_{>0}\xi(\beta,H,t)=
{\Bbb R}_{>0}\mathrm{Im}(Z_{(\beta,tH)}(v)^{-1} e^{\beta+\sqrt{-1}tH}).
\end{equation*}
\end{lem}

\begin{proof}
\begin{equation*}
\begin{split}
Z_{(\beta,tH)}(v)=& \frac{r}{2}t^2 (H^2)-a_\beta
+(\xi-r \beta,tH)\sqrt{-1}\\
%=& r \frac{t^2 (H^2) -((\delta-\beta)^2)}{2}
%+\frac{\langle v^2 \rangle}{2r} +r(\delta- \beta,tH)\sqrt{-1}.
\end{split}
\end{equation*}
For the complex conjugate $\overline{Z_{(\beta,tH)}(v)}$
of $Z_{(\beta,tH)}(v)$,
we have
\begin{equation*}
\begin{split}
& \mathrm{Im}(\overline{Z_{(\beta,tH)}(v)} e^{\beta+\sqrt{-1} tH})\\
=&
\left(r\frac{t^2 (H^2)}{2}-a_\beta \right)
t(H+(\beta,H)\varrho_X)
-(\xi-r \beta,tH)\left(e^\beta-\frac{t^2(H^2)}{2}\varrho_X \right)\\
%r\left( \frac{t^2 (H^2) -((\delta-\beta)^2)}{2}
%+\frac{\langle v^2 \rangle}{2r^2} \right)(tH+(tH,\beta)\varrho_X)-
%r(\delta- \beta,tH)\left( e^\beta-\frac{t^2 (H^2)}{2} \right)\\
%=& r\left( \frac{t^2 (H^2) -((\delta-\beta)^2)}{2}
%+\frac{\langle v^2 \rangle}{2r^2} \right)(tH+(tH,\delta)\varrho_X)
%+r\left( \frac{t^2 (H^2) -((\delta-\beta)^2)}{2}
%+\frac{\langle v^2 \rangle}{2r^2} \right)(tH,\beta-\delta)\varrho_X\\
%& -(\delta- \beta,tH)\left( e^\beta-\frac{t^2 (H^2)}{2} \right)\\
=& t\xi(\beta,H,t). 
\end{split}
\end{equation*}
Hence the claim holds.
\end{proof}

\begin{rem}
The expression of $\xi(\beta,H,t)$ in Lemma \ref{lem:Im} 
appeared in \cite{BM}.
The difference of the sign comes from the difference
of $\theta_v$ and $\theta_{v,\beta,tH}$.  
\end{rem}

\begin{rem}\label{rem:theta(rho)}
Assume that $r=0$ and $t>0$.
Then 
$$
\{\xi(\beta,H,st) \mid s \geq 1 \}=
\xi(\beta,H,t)+{\Bbb R}_{ \geq 0}\varrho_X
$$
 and
$$
\lim_{t \to \infty}\xi(\beta,H,t)/t^2=(\xi,H)\frac{(H^2)}{2}\varrho_X.
$$
If $\xi$ is effective, then
we have a morphism
$M_H^\beta(v) \to \Hilb_X^{\xi}$ by sending $E$ to the scheme-theoretic
support $\Div(E)$ and $\theta_v(\varrho_X)$ comes from
$\Hilb_X^\xi$. 
Since $M_{(\beta,tH)}(v)=M_H^\beta(v)$ for $t \gg 0$ and $(\xi,H)>0$,
$\theta_{v,\beta,tH}((\xi,H)\varrho_X)$ is base point free 
for $t \gg 0$.
\begin{NB}
If $(\xi,H)<0$, then
$E[1] \in M_H(-v)$ up to shift $[2k]$ 
for all $E \in M_{(\beta,tH)}(v)$.
\end{NB}
\end{rem}

We shall remark the behavior of $\xi(\beta,H,t)$ under the Fourier-Mukai
transforms.
Let 
$$
\Phi:{\bf D}(X) \to {\bf D}^{\alpha_1}(X_1)
$$
 be a twisted Fourier-Mukai transform such that
$$
\Phi(r_1 e^\gamma)=-\varrho_{X_1},\;
\Phi(\varrho_X)=-r_1 e^{\gamma'},
$$
where $\alpha_1$ is a representative of a suitable
Brauer class.
Then we can describe the cohomological Fourier-Mukai transform
as 
\begin{equation*}
\Phi(r e^\gamma+a \varrho_X+\xi+(\xi,\gamma)\varrho_X)
=-\frac{r}{r_1} \varrho_{X_1}-r_1 a e^{\gamma'}+
\frac{r_1}{|r_1|}
( \widehat{\xi}+(\widehat{\xi},\gamma')\varrho_{X_1}),
\end{equation*}
where 
$$
\xi \in \NS(X)_{\Bbb Q},\;
\widehat{\xi}:=
\frac{r_1}{|r_1|} 
c_1(\Phi(\xi+(\xi,\gamma)\varrho_X)) \in \NS(X_1)_{\Bbb Q}.
$$
\begin{rem}
By taking a locally free $\alpha_1$-twisted stable
sheaf $G$ with $\chi(G,G)=0$,
we have a notion of Mukai vector, thus, we have
a map (\cite[Rem. 1.2.10]{MYY:2011:1}):
$$
v_G:{\bf D}^{\alpha_1}(X_1) \to H^*(X_1,{\Bbb Q})_{\alg}.
$$
\end{rem}
For $(\beta,\omega) \in \NS(X)_{\Bbb R} \times \Amp(X)_{\Bbb R}$,
we set 
\begin{equation}\label{eq:tilde(beta)}
\begin{split}
\widetilde{\omega}:= & -\frac{1}{|r_1|}
\frac{\frac{((\beta-\gamma)^2)-(\omega^2)}{2}\widehat{\omega}-
(\beta-\gamma,\omega)(\widehat{\beta}-\widehat{\gamma})}
{\left(\frac{((\beta-\gamma)^2)-(\omega^2)}{2} \right)^2
+(\beta-\gamma,\omega)^2},\\
\widetilde{\beta}:= & \gamma'-\frac{1}{|r_1|}
\frac{\frac{((\beta-\gamma)^2)-(\omega^2)}{2}(\widehat{\beta}-\widehat{\gamma})
-(\beta-\gamma,\omega) \widehat{\omega}}
{\left(\frac{((\beta-\gamma)^2)-(\omega^2)}{2} \right)^2
+(\beta-\gamma,\omega)^2}.
\end{split}
\end{equation}
Then $(\widetilde{\beta},\widetilde{\omega}) 
\in \NS(X_1)_{\Bbb R} \times \Amp(X_1)_{\Bbb R}$.

By \cite[sect. A.1]{MYY:2011:2},
we get the following commutative diagram:
\begin{equation*}
\xymatrix{
   {\bf D}(X) \ar[r] \ar[d]_{Z_{(\beta,\omega)}}
 & {\bf D}^{\alpha_1}(X_1) \ar[d]^{Z_{(\wt{\beta},\wt{\omega})}} \\
   {\Bbb C}  \ar[r]_{\zeta^{-1}} 
 & {\Bbb C}
}
%\begin{CD}
%{\bf D}(X) @>>> {\bf D}^{\alpha_1}(X_1)\\
%@V{Z_{(\beta,\omega)}}VV @VV{Z_{(\widetilde{\beta},\widetilde{\omega})}}V\\
%{\Bbb C} @>{\zeta^{-1}}>> {\Bbb C}
%\end{CD}
\end{equation*}
where 
$$
\zeta=-r_1 \left(
\frac{((\gamma-\beta)^2)-(\omega^2)}{2}
+\sqrt{-1}(\beta-\gamma,\omega) \right).
$$

\begin{NB}
\begin{prop}
Assume that $Z_{(\beta,tH)}(r_1 e^\gamma) \in {\Bbb R}_{>0}
Z_{(\beta,tH)}(v)$.
Then $\Phi(\xi(\beta,H,t)) \in {\Bbb R}_{>0}
(\widetilde{tH}+(\widetilde{tH},\widetilde{\beta})\varrho_{X_1})$.
\end{prop}

\begin{proof}
\begin{equation*}
\Phi( e^{\beta+\sqrt{-1}tH})=
\zeta e^{\widetilde{\beta}+\sqrt{-1}\widetilde{tH}}.
\end{equation*}

Since 
\begin{equation*}
\begin{split}
Z_{(\beta,tH)}(v)=& \frac{r}{2}t^2 (H^2)-a+(c_1-r \beta,tH)\sqrt{-1}\\
=& r \frac{t^2 (H^2) -((c_1/r-\beta)^2)}{2}
+\frac{\langle v^2 \rangle}{2r} +(c_1-r \beta,tH)\sqrt{-1}
\end{split}
\end{equation*}
and $Z_{(\beta,tH)}(r_1 e^\gamma)=\zeta$,
we have
\begin{equation*}
{\Bbb R}_{>0}\xi(\beta,H,t)=
{\Bbb R}_{>0}\mathrm{Im}(Z_{(\beta,tH)}(v)^{-1} e^{\beta+\sqrt{-1}tH})=
{\Bbb R}_{>0}\mathrm{Im}(\zeta^{-1} e^{\beta+\sqrt{-1}tH}).
\end{equation*}
\end{proof}
\end{NB}

\begin{prop}\label{prop:xi-FM}
For $(\beta,H,t) \in \NS(X)_{\Bbb R} \times \Amp(X)_{\Bbb R} 
\times {\Bbb R}_{>0}$, we have
${\Bbb R}_{>0}\Phi(\xi(\beta,H,t))=
{\Bbb R}_{>0}\xi(\widetilde{\beta},H_1,t_1)$, where
$t_1 H_1=\widetilde{tH}$.
\end{prop}

\begin{proof}
\begin{equation*}
\begin{split}
{\Bbb R}_{>0}\Phi(\xi(\beta,H,t))=&
{\Bbb R}_{>0}\Phi(\mathrm{Im}(Z_{(\beta,tH)}(v)^{-1} e^{\beta+\sqrt{-1}tH}))\\
=& {\Bbb R}_{>0} \mathrm{Im}(Z_{(\beta,tH)}(v)^{-1} 
\Phi(e^{\beta+\sqrt{-1}tH}))\\
=& {\Bbb R}_{>0} 
\mathrm{Im}(Z_{(\widetilde{\beta},\widetilde{tH})}(\Phi(v))^{-1} \zeta^{-1}
\zeta e^{\widetilde{\beta}+\sqrt{-1}\widetilde{tH}})\\
=& {\Bbb R}_{>0} \xi(\widetilde{\beta},H_1,t_1).
\end{split}
\end{equation*}
\end{proof}

\begin{rem}
We have a commutative diagram
\begin{equation*}
\xymatrix{
   v^\perp \ar[r]^{\Phi} \ar[d]_{\theta_{v,\beta,\omega}}
 & \Phi(v)^\perp \ar[d]^{\theta_{\Phi(v),\wt{\beta},\wt{\omega}}} \\
   \NS(K_{(\beta,\omega)}(v)) \ar[r]_{\Phi_*} 
 & \NS(K_{(\wt{\beta},\wt{\omega})}^\alpha(\Phi(v)))
}
%\begin{CD}
%{\bf D}(X) @>>> {\bf D}^{\alpha_1}(X_1)\\
%@V{Z_{(\beta,\omega)}}VV @VV{Z_{(\widetilde{\beta},\widetilde{\omega})}}V\\
%{\Bbb C} @>{\zeta^{-1}}>> {\Bbb C}
%\end{CD}
\end{equation*}
where $K_{(\wt{\beta},\wt{\omega})}^\alpha(\Phi(v))$ is the
Bogomolov factor of the moduli of 
$\sigma_{(\wt{\beta},\wt{\omega})}$-stable objects
$M_{(\wt{\beta},\wt{\omega})}^\alpha(\Phi(v))$.  
Then we have 
$$
{\Bbb R}_{>0}\Phi_*(\theta_{v,\beta,\omega}(\xi(\beta',H',t')))=
{\Bbb R}_{>0}\theta_{\Phi(v),\wt{\beta},\wt{\omega}}
(\xi(\widetilde{\beta'},H_1,t_1)),
$$
where $t_1 H_1=\widetilde{t'H'}$.
\end{rem}

\subsection{Stability conditions and the positive cone.}
\label{subsect:positive-cone}

Assume that $r \ne 0$.
We note that $\xi(\beta,H,t) \in \varrho_X^{\perp}$
if and only if
 $\delta-\beta \in H^{\perp}$.
In this case, we have 
$((\delta-\beta)^2) \leq 0$, which implies that
$$
\frac{(\omega^2)-((\beta-\delta)^2)}{2}+
\frac{\langle v^2 \rangle}{2r^2}> 0.
$$
Therefore we get the following claim.
\begin{lem}\label{lem:rho^perp}
$$
\xi(\beta,H,t) \in \varrho_X^\perp \Longleftrightarrow 
\xi(\beta,H,t)=r(\eta+(\eta,\delta)\varrho_X),\;
\eta \in \overline{\Amp(X)}_{\Bbb R}
$$
Moreover 
$$
\eta \in \Amp(X)_{\Bbb R}
\Longleftrightarrow H \in \Amp(X)_{\Bbb R}.
$$
%$\xi(\beta,H,t) \in \varrho_X^\perp$
%if and only if 
%$\xi(\beta,H,t)=r(\eta+(\eta,\delta)\varrho_X)$,
%$\eta \in \overline{\Amp(X)}_{\Bbb R}$.
%Moreover $\eta \in \Amp(X)_{\Bbb R}$ if and only if
%$tH \in \Amp(X)_{\Bbb R}$. 
\end{lem}

\begin{lem}\label{lem:non-perp}
Assume that $(\beta-\delta,H) \ne 0$ and set
\begin{equation*}
\eta:=\beta-\delta+\frac{1}{(\beta-\delta,H)}
\left(\frac{t^2(H^2)-((\beta-\delta)^2)}{2}+
\frac{\langle v^2 \rangle}{2r^2} \right)H.
\end{equation*}
Then
\begin{enumerate}
\item[(1)]
$(\eta^2) \geq \frac{\langle v^2 \rangle}{r^2}$.
Moreover 
$$
(\eta^2)=\frac{\langle v^2 \rangle}{r^2}
\Longleftrightarrow
\begin{cases}
\text{$(H^2)=0$ or }\\
\text{$t=0$ and $((\beta-\delta)^2)
=\frac{\langle v^2 \rangle}{r^2}$.}
\end{cases}
$$
\item[(2)]
$(\beta-\delta,H)(\eta,H')>0$ for $H' \in \Amp(X)_{\Bbb R}$
which is sufficiently close to $H$.
\end{enumerate}
\end{lem}

\begin{proof}
(1)
\begin{equation*}
\begin{split}
(\eta^2)
=& ((\beta-\delta)^2)+2\left(\frac{t^2(H^2)-((\beta-\delta)^2)}{2}
\frac{\langle v^2 \rangle}{2r^2} \right)\\
& \quad +\frac{(H^2)}{(\beta-\delta,H)^2}
\left(\frac{t^2(H^2)-((\beta-\delta)^2)}{2}+
\frac{\langle v^2 \rangle}{2r^2} \right)^2 \\
=& t^2 (H^2)+\frac{\langle v^2 \rangle}{r^2}+
\frac{(H^2)}{(\beta-\delta,H)^2}
\left(\frac{t^2(H^2)-((\beta-\delta)^2)}{2}+
\frac{\langle v^2 \rangle}{2r^2} \right)^2\\
\geq & \frac{\langle v^2 \rangle}{r^2}.
\end{split}
\end{equation*}
Moreover the equality holds if and only if
$(\beta,H,t)$ satisfy (i) or (ii). 

(2)
It is sufficient to prove
$(\beta-\delta,H)(\eta,H)>0$.
If $(H^2)=0$, then
$(\beta-\delta,H)(\eta,H)=(\beta-\delta,H)^2>0$.
Assume that $(H^2)>0$.
We set $\beta-\delta=sH+D$ ($s \in {\Bbb R}$, $D \in H^\perp$).
Then
\begin{equation*}
(\beta-\delta,H)(\eta,H)=
\frac{1}{2}
\left((\beta-\delta,H)^2-(H^2)(D^2)+t^2(H^2)^2+
\frac{(H^2) \langle v^2 \rangle}{r^2}
\right)>0.
\end{equation*}
\end{proof}

\begin{NB}
Since $d_\beta=-s$,
\begin{equation*}
 -s \xi_\omega=
r \left( (s^2+t^2)\frac{(H^2)}{2}-\frac{(D^2)}{2}+
\frac{\langle v^2 \rangle}{2r^2} \right)
\left(H+(\delta,H)\varrho_X \right)
+rs(H^2)(D+(D,\delta)\varrho_X)+rs(H^2)
\left(e^\delta+\frac{\langle v^2 \rangle}{2r^2}\varrho_X \right).
\end{equation*}
\begin{defn}
We set
\begin{equation}\label{eq:xi(s,t,D)}
\begin{split}
\xi(\beta,\omega):=&
-st \xi_\omega\\
=&
rt \left( (s^2+t^2)\frac{(H^2)}{2}-\frac{(D^2)}{2}+
\frac{\langle v^2 \rangle}{2r^2} \right)
\left(H+(\delta,H)\varrho_X \right)
+rst(H^2)(D+(D,\delta)\varrho_X)+rst(H^2)
\left(e^\delta+\frac{\langle v^2 \rangle}{2r^2}\varrho_X \right)\\
=& 
r \left( (\frac{s^2}{t^2}+1)\frac{(\omega^2)}{2}-\frac{(D^2)}{2}+
\frac{\langle v^2 \rangle}{2r^2} \right)
\left(\omega+(\delta,\omega)\varrho_X \right)
+r\frac{s}{t}(\omega^2)(D+(D,\delta)\varrho_X)+r\frac{s}{t}(\omega^2)
\left(e^\delta+\frac{\langle v^2 \rangle}{2r^2}\varrho_X \right),
\end{split}
\end{equation}
where $s=\frac{(\beta-\frac{\xi}{r},\omega)}{(\omega^2)}t$,
$D=\beta-\frac{\xi}{r}-
\frac{(\beta-\frac{\xi}{r},\omega)}{(\omega^2)}\omega$.
\end{defn}
\end{NB}

\begin{defn}
We set
\begin{equation*}
\begin{split}
P^+(v^\perp)_k:=& \{x \in H^*(X,{\Bbb Z})_{\alg} \otimes k \mid
x \in v^\perp, \langle x^2 \rangle> 0, 
\langle x,rH_0+(rH_0,\delta)\varrho_X \rangle>0 \},\\
\overline{P^+(v^\perp)}_k:=& 
\{x \in H^*(X,{\Bbb Z})_{\alg} \otimes k \mid
x \in v^\perp, \langle v^2 \rangle \geq 0, 
\langle x,rH_0+(rH_0,\delta)\varrho_X \rangle>0 \},
\end{split}
\end{equation*}
where $k={\Bbb Q},{\Bbb R}$ and $H_0 \in \Amp(X)_{\Bbb Q}$.
\end{defn}
We take a norm $||\;\;||$ on $\NS(X)_{\Bbb R}$
defined over ${\Bbb Q}$ and regard the cone
$C(\overline{\Amp(X)}_{\Bbb R})$ in
subsection \ref{subsect:parameter-space} as a subset of
$\overline{\Amp(X)}_{\Bbb R}$ 
(subsection \ref{subsect:parameter-space}):
\begin{equation*}
C(\overline{\Amp(X)}_{\Bbb R})=
\{ L \in \overline{\Amp(X)}_{\Bbb R} \mid 
||L||=1\}.
\end{equation*}
Under this identification, 
$\xi(\beta,H,t)$ is defined for
$$
(\beta,H,t) \in 
\overline{\frak H}=\NS(X)_{\Bbb R} \times C(\overline{\Amp(X)}_{\Bbb R})
\times {\Bbb R}_{\geq 0}.
$$
For the embedding
\begin{equation*}
\begin{matrix}
\NS(X)_{\Bbb R} \times \overline{\Amp(X)}_{\Bbb R}
& \to & \overline{\frak H}\\
(\beta,H) & \mapsto & (\beta,H/||H||,||H||),
\end{matrix}
\end{equation*} 
we have 
\begin{equation*}
\begin{split}
{\Bbb R}_{>0}\xi(\beta,H,t)=&
{\Bbb R}_{>0}\xi(\beta,H/||H||,||H||t),\;(\beta,H,t) 
\in \NS(X)_{\Bbb R} \times \overline{\Amp(X)}_{\Bbb R} 
\times {\Bbb R}_{\geq 0}.
%{\Bbb R}_{>0}\xi(\beta,H,0)=&
%{\Bbb R}_{>0}\xi(\beta,H/||H||,)
\end{split}
\end{equation*}

\begin{prop}\label{prop:positive}
We have a map
\begin{equation*}
\begin{matrix}
\Xi:& \overline{\frak H} & \to & 
C(H^*(X,{\Bbb Z})_{\alg} \otimes {\Bbb R})\\
& (\beta,H,t) & \mapsto & {\Bbb R}_{>0}\xi(\beta,H,t) 
\end{matrix}
\end{equation*}
whose image is the positive cone
$C^+:=C(\overline{P^+(v^\perp)}_{\Bbb R})$ of
$v^\perp$.
Moreover if $tH$ is ample, then 
$\xi(\beta,H,t)$ belongs to the interior of
$C^+$. 
\begin{NB}
If $\langle \xi(\beta,H,t)^2 \rangle=0$, then
$t=0$ or $(H^2)=0$ by Lemma \ref{lem:non-perp}.
\end{NB}
\end{prop}

\begin{proof}
By Proposition \ref{prop:xi-FM}, it is sufficient to prove
the claim for $r \ne 0$.
An element 
\begin{equation}\label{eq:u}
u=\zeta+(\zeta,\delta)\varrho_X
+y(e^\delta+\tfrac{\langle v^2 \rangle}{2r^2}\varrho_X)
\in v^\perp \cap H^*(X,{\Bbb Z})_{\alg} \otimes {\Bbb R}
\end{equation}
satisfies $\langle u^2 \rangle \geq 0$
and $\langle u,rH+(rH,\delta)\varrho_X \rangle>0$
if and only if
$(\zeta^2) \geq y^2 \frac{\langle v^2 \rangle}{r^2}$
and $(\zeta,rH)>0$.
We first assume that $y \ne 0$.
We set $\frac{\zeta}{y}:=x H+D$ ($x \in {\Bbb R}$, $D \in H^\perp$)
and
$$
\sigma_\pm:= \pm \sqrt{\frac{\langle v^2 \rangle-r^2(D^2)}{r^2(H^2)}}.
$$
Then the conditions are 
$x \geq \sigma_+$ if $ry>0$ and $x \leq \sigma_-$ if $ry<0$.
\begin{NB}
$0<(\zeta,rH)=rxy(H^2)$.
\end{NB} 
If $y=0$, then the condition is
$(\zeta^2) \geq 0$ and $(\zeta,rH) >0$,
that is, $r\zeta \in \overline{\Amp(X)}_{\Bbb R}$.

We shall first prove that $\im(\Xi)$ contains
$C^+$. 
\begin{NB}
We may not need this:
By Lemma \ref{lem:rho^perp}, we may assume that
$y \ne 0$.
\end{NB}
We shall find $(\beta,H,t)$ such that
$\beta=\delta+sH+D$, $(s \in {\Bbb R}, D \in H^\perp)$
and 
$$
{\Bbb R}_{>0}\xi(\beta,H,t)={\Bbb R}_{>0}u.
$$
We set
\begin{equation*}
\begin{split}
g(s,t):=&
\frac{(H^2)(s^2+t^2)+\frac{\langle v^2 \rangle}{r^2}-(D^2)}
{2s(H^2)}.\\
%\sigma_\pm:=& \pm \sqrt{\frac{\langle v^2 \rangle-r^2(D^2)}{r^2(H^2)}}.
\end{split}
\end{equation*}
Then $g(s,0)$ define continuous functions
from $(0,\infty)$ to 
$[\sigma_+,\infty)$ and from
$(-\infty,0)$ to $(-\infty,\sigma_-]$. 
Hence we can take $s \in {\Bbb R}$
with 
$x=g(s,0)$.
\begin{NB}
$\zeta=r(\beta-\delta,H)(g(s,0)H+D)$.
\end{NB}
For $\beta:=\delta+sH+D$,
we have 
\begin{equation*}
\begin{split}
\frac{\xi(\delta+sH+D,H,0)}{r(\beta-\delta,H)}=&
\frac{1}{(\beta-\delta,H)}
\left(-\frac{((\beta-\delta)^2)}{2}+\frac{\langle v^2 \rangle}{2r^2}
\right)(H+(H,\delta)\varrho_X)\\
& \quad +
(\beta-\delta+(\beta-\delta,\delta)\varrho_X)
+\left(e^\delta-\frac{a_\delta}{r}\varrho_X \right)\\
=& g(s,0)(H+(H,\delta)\varrho_X)+
(D+(D,\delta)\varrho_X)
+\left(e^\delta-\frac{a_\delta}{r}\varrho_X \right)=\frac{u}{y}.
\end{split}
\end{equation*}
%$\frac{\xi(\delta+sH+D,H,0)}{r(\beta-\delta,H)}=\frac{u}{y}$.
%In particular $x=g(s,0)$.
Since $r(\beta-\delta,H)y=rsy(H^2)$ and
$sx=sg(s,0)>0$, $rxy=(\zeta,rH)/(H^2)>0$ implies that 
$r(\beta-\delta,H)$ and $y$ have the same sign.
Thus 
$$
{\Bbb R}_{>0}\xi(\delta+sH+D,H,0)={\Bbb R}_{>0}u.
$$
If $y=0$, then Lemma \ref{lem:rho^perp} implies that
$$
\xi(\delta,r\zeta,t)=\frac{r^4 t^2(\zeta^2)+\langle v^2 \rangle}{2}
(\zeta+(\zeta,\delta)\varrho_X) \in {\Bbb R}_{>0}u.
$$
Hence $u \in \im \Xi$.  
Conversely for $\xi(\beta,H,t)$,
Lemma \ref{lem:rho^perp} or 
Lemma \ref{lem:non-perp} implies that
$\xi(\beta,H,t) \in C^+$.
Therefore the claim holds. 
\end{proof}

\begin{NB}
If $y \ne 0$, then we can take an arbitrary ample divisor
$H$ for $u$.  
\end{NB}

\begin{rem}\label{rem:positive}
If $u$ in \eqref{eq:u} belongs to
$H^*(X,{\Bbb Z})_{\alg} \otimes {\Bbb Q}$
and satisfies $\langle u^2 \rangle>0$,
then there is $(\beta,H,t)$ such that
$\Xi(\beta,H,t)=u$, 
$\beta,H \in \NS(X)_{\Bbb Q}$ and $t^2 \in {\Bbb Q}$:
Indeed if $u \in H^*(X,{\Bbb Z})_{\alg}$,
then we may assume that $H \in \NS(X)_{\Bbb Q}$.
For $g(s,0) \in {\Bbb Q}$,
we can take $(s',t')$ such that
$s',{t'}^2 \in {\Bbb Q}$ and $g(s',t')=g(s,0)$.
Hence the claim holds. 
\begin{NB}
$g(s,t)=g(s_0,0)$
if and only if 
$t^2(H^2)=-(H^2)s^2+((H^2)s_0^2+(\tfrac{\langle v^2 \rangle}{r^2}-(D^2)))
\tfrac{s}{s_0}-(\tfrac{\langle v^2 \rangle}{r^2}-(D^2))$.
Since the descriminant of the RHS is
$D:=((H^2)s_0-\tfrac{1}{s_0}(\tfrac{\langle v^2 \rangle}{r^2}-(D^2)))^2$
and 
$D=0$ if and only if $s_0=\sigma_\pm$,
if $|s|>|\sigma_+$, then 
we have $g(s,t)=g(s_0,0)$ with $s,t^2 \in {\Bbb Q}$. 
\end{NB}
\end{rem}

\begin{prop}\label{prop:wall-equation}
$(\beta,H,t) \in W_{v_1}$ (see Definition \ref{defn:wall})
if and only if 
$\xi(\beta,H,t) \in v^\perp \cap v_1^\perp$.
\end{prop}

\begin{proof}
We note that \eqref{eq:W} is written as
\begin{equation}\label{eq:W2}
\begin{split}
& \det
\begin{pmatrix}
a-(\xi,\beta)+r \frac{(\beta^2)-t^2(H^2)}{2} &
a_1-(\xi_1,\beta)+r_1 \frac{(\beta^2)-t^2(H^2)}{2}\\
-(\xi-r \beta,H) & -(\xi_1-r_1 \beta,H)
\end{pmatrix}\\
=&
\begin{pmatrix}
\frac{((\delta-\beta)^2)-t^2(H^2)}{2}-
\frac{\langle v^2 \rangle}{2r^2} &
\langle -e^\beta+\frac{t^2(H^2)}{2}\varrho_X,v_1 \rangle \\
-(\xi-r \beta,H) & -\langle H+(H,\beta)\varrho_X,v_1 \rangle
\end{pmatrix}\\
=&  \langle \xi(\beta,H,t),v_1 \rangle.
\end{split}
\end{equation}
Hence the claim holds.
\end{proof}

{\it Proof of Proposition \ref{prop:existence-of-wall}.}
By Proposition \ref{prop:wall-equation} and Proposition 
\ref{prop:positive}, it is sufficient to find the condition 
$\overline{P^+(v^\perp)}_{\Bbb R} \cap
v_1^\perp \ne \emptyset$.
We set $\eta:=\langle v^2 \rangle v_1-\langle v_1,v \rangle v
\in v^\perp$.
Then $v_1^\perp \cap v^\perp=\eta^\perp \cap v^\perp$.
Since the signature of $v^\perp$ is
$(1,\rk \NS(X))$,
the condition is $(\eta^2)<0$. 
Since
$$
\langle \eta^2 \rangle=
\langle v^2 \rangle
(\langle v_1^2 \rangle \langle v^2 \rangle-\langle v_1,v \rangle^2),
$$
we get the claim.
\qed

\begin{NB}

\begin{lem}
Assume that $a_1 b_1-\sum_i a_i b_i=0$,
$a_1^2-\sum_i a_i^2 >\epsilon>0$,
$b_1^2-\sum_i b_i^2 >-N$, $N>0$.
Then $b_1^2 \leq (N a_1^2-\epsilon N)/\epsilon$.
\end{lem}

\begin{proof}
\begin{equation*}
\begin{split}
(a_1 b_1)^2=(\sum_i a_i b_i)^2 \leq (\sum_i a_i^2)(\sum_i b_i^2)
\leq (a_1^2-\epsilon)(b_1^2+N).
\end{split}
\end{equation*}
Since $a_1 \ne 0$, we get the claim.
\end{proof}

\begin{prop}
Let $L$ be a Lorentzian lattice.
Let $B$ be a compact subset of $P^+(L)$.
Then 
$$
\{ v \in L \mid \langle v^2 \rangle \geq -N,\; (v,h)=0, h \in B \}
$$
is a finite set.
\end{prop}

\begin{proof}
Since $B$ is compact, $(h^2)> \epsilon>0$ for $h \in B$.
We take a orthogonal basis $e_1,e_2,...,e_n$ of $L$ 
such that $e_1 \in B \cap L$.
By the above lemma, 
$\langle v, e_1 \rangle$ is bounded.
Since $v':=v-\frac{\langle v, e_1 \rangle}{\langle e_1^2 \rangle} e_1$
belongs to $\frac{1}{M}e_1^\perp$ for some $M \in {\Bbb Z}_{>0}$
and
$\langle {v'}^2 \rangle=\langle v^2 \rangle-
\frac{\langle v,e_1 \rangle^2}{\langle e_1^2 \rangle}$ is bounded,
the choice of $v'$ is finite.
Hence the choice of $v$ is also finite.
\end{proof}

\end{NB}

\begin{defn}
\begin{enumerate}
\item[(1)]
A connected component $D$ of 
$P^+(v^\perp)_{\Bbb R} \setminus \cup_{v_1 \in {\frak W}} v_1^\perp$
is a chamber. 
\item[(2)]
$D(\beta,H,t)$ is a chamber such that
 $\xi(\beta,H,t) \in D(\beta,H,t)$.
\end{enumerate}
\end{defn}

$\overline{D(\beta,H,t)}$ consists of 
$x \in \overline{P^+(v^\perp)}_{\Bbb R}$ such that
$\langle \xi(\beta,H,t),\pm v_1 \rangle>0$
implies $\langle x,\pm v_1 \rangle \geq 0$, that is,
$\langle \xi(\beta,H,t),v_1 \rangle
\langle x,v_1 \rangle \geq 0$.

\begin{prop}[{\cite[Thm. 1.6]{MYY:2011:2}}]\label{prop:positive-cone}
We fix a general $(\beta,H,t)$.
\begin{enumerate}
\item[(1)]
Assume that $(\beta',H',t')$ belongs to a chamber and
$H' \in \Amp(X)_{\Bbb Q}$,
${t'}^2 \in {\Bbb Q}$, $\beta' \in \NS(X)_{\Bbb Q}$.
Then
$\theta_{v,\beta',t'H'}(\xi(\beta',H',t'))$ is an ample ${\Bbb Q}$-divisor
of $K_{(\beta',t'H')}(v)$.
\item[(2)]
We have a surjective map
\begin{equation*}
\varphi_{(\beta,H,t)}: \overline{\frak H} \to 
C(\overline{P^+(K_{(\beta,tH)}(v))}_{\Bbb R}) 
\end{equation*}
such that 
$$
\varphi_{(\beta,H,t)}((\beta',H',t')):={\Bbb R}_{>0} 
\theta_{v,\beta,tH}(\xi(\beta',H',t')).
$$
\item[(3)]
Let ${\cal C}$ be a chamber in
$\overline{\frak H}$. 
Then $\Xi({\cal C})$ is the chamber $D(\beta',H',t')$ 
$((\beta',H',t') \in {\cal C})$ in $C^+$ and
$$
\theta_{v,\beta,tH}(\overline{D(\beta,H,t)}) 
=\Nef(K_{(\beta,tH)}(v))_{\Bbb R}.
$$
%Thus $\Nef(K_{(\beta,tH)}(v))_{\Bbb R}$
%$((\beta,H,t) \in D)$ 
%corresponds
%to the closure of $D$ via $\varphi^{-1}$.
%In particular, the closure of positive cone has a chamber structure 
%by the chamber structure of the stability conditions.
\end{enumerate}
\end{prop}

\begin{proof}
(1)
\begin{NB}
Old proof:
We write $\beta'=s' H'+D'+\delta$ ($D' \in {H'}^\perp$). 
In the notation of \cite[sect. 5.3]{MYY:2011:2},
we note that $d_{\beta'}=d_{\delta+D'}-rs'>0$ if and only if
$d_{\beta'}-r \lambda=d_{\delta+D'}-r(s'+\lambda)>0$, where
$\lambda \in {\Bbb R}$ satisfies
$\phi_{(\beta',t'H')}(e^{\beta'+\lambda H'})=\phi_{(\beta',t'H')}(v)$. 
Purturbing $t'$ slightly, we may assume that
$\lambda \in {\Bbb Q}$.
\begin{NB2}
For $\beta'=s' H'+D'+\delta$,
we take $\lambda:=\lambda(s',t')$ such that
$\phi_{(\beta',t' H')}(e^{\beta'+\lambda H'})=\phi_{(\beta',t'H')}(v)$. 
For this $\lambda$, we consider the circle
$\phi_{(\beta(x),yH')}(e^{\beta(x)+\lambda H'})=\phi_{(\beta(x),yH')}(v)$,
in the $(x,y)$-plane, where $\beta(x)=xH'+D'+\delta$.
This circle passes through $(x,y)=(s'+\lambda,0)$ and deso not intersect
the line $x=d_{\delta+D'}/r$.
Hence 
$d_{\delta+D}-rs>0$ if and only if $d_{\delta+D}-r(s+\lambda)>0$. 
\end{NB2}
So the claim is a consequence of
\cite[Prop. 4.3.2]{MYY:2011:2}, if $d_\beta>0$.
\end{NB}
Since $(\beta',H',t')$ is general,
we may assume that $d_{\beta'}(v) \ne 0$.
If $d_{\beta'}(v)>0$, then the claim is a consequence of
\cite[Thm. 1.6]{MYY:2011:2}.
If $d_{\beta'}(v)<0$, then we apply 
\cite[Thm. 1.6]{MYY:2011:2} to $M_{(\beta',t' H')}(-v)$.
Since $\xi_\omega$ for $v$ is the same as that for $-v$
and $\theta_{-v,\beta',t' H'}=
-\theta_{v,\beta',t' H'}$, 
$-\theta_{v,\beta',t' H'}(\xi_{t' H'})$ is ample, which implies that
$d_{\beta'} \theta_{v,\beta',t' H'}(\xi_{t' H'})$ is ample.

(2)
Since $\theta_{v,\beta,tH}:\overline{P^+(v^\perp)}_{\Bbb R}
\to \overline{P^+(K_{(\beta,tH)}(v))}_{\Bbb R}$
is an isomorphism, the claim follows from
Proposition \ref{prop:positive}.
  
(3) By (1) and (2), $\theta_{\beta,H,t}(D(\beta,H,t))$ is contained in
		 $\Nef(K_{(\beta,tH)}(v))_{\Bbb R}$.
Then the claim follows from
\cite[Cor. 5.17 (2)]{MYY:2011:2}.
More precisely, we proved that
$\xi(\beta',H',t') \in P^+(v^\perp)_{\Bbb R} \cap v_1^\perp \cap 
\overline{D(\beta,H,t)}$ gives a contraction
under the assumption
$d_{\beta'}>0$. 
For the case where $d_{\beta'}<0$, we get the claim
by the same reduction in (1).
We next treat the case where $d_{\beta'}=0$.
\begin{NB}
If $r_1 \xi-r \xi_1 \ne 0$, then
there is $H \in \Amp(X)_{\Bbb Q}$ such that
$(\xi-r\beta,H)=0$ and $(r_1 \xi-r \xi_1,H) \ne 0$.
Since $(\xi_1-r_1 \beta,H) \ne 0$,
$(\beta,H,t) \not \in W_{v_1}$.
\end{NB}
If $r_1 \xi-r \xi_1 \ne 0$, then 
 $W_{v_1}$ does not contain the set $d_{\beta'}=0$ by
Lemma \ref{lem:d_beta=0} (1).
For a general point of $W_{v_1}$, we can apply
\cite[Cor. 5.17 (2)]{MYY:2011:2}.
Hence for any $\xi(\beta',H',t') \in P^+(v^\perp)_{\Bbb R} \cap v_1^\perp 
\cap \overline{D(\beta,H,t)}$,
$\theta_{v,\beta,tH}(\xi(\beta',H',t'))$ gives a contraction.

If $r_1 \xi-r \xi_1=0$, then
$W_{v_1}$ is the set $d_{\beta'}=0$ by Lemma \ref{lem:d_beta=0} (1).
In this case, $\theta_{v,\beta,tH}(\xi(\beta',H',t'))$ gives a contraction.
\end{proof}

\begin{cor}\label{cor:positive-cone}
For $(\beta',H',t') \in \overline{D(\beta,H,t)}$
with $\langle \xi(\beta',H',t')^2 \rangle>0$,
$\theta_{v,\beta,tH}(\xi(\beta',H',t'))$
gives a birational contraction of $K_{(\beta,tH)}(v)$.
\end{cor}

\begin{proof}
We first note that
the canonical bundle of $K_{(\beta,tH)}(v)$ is trivial.
By Proposition \ref{prop:positive-cone} (3)
and $\langle \xi(\beta',H',t')^2 \rangle>0$,
we can apply the base point free theorem
to get the claim.
\end{proof}

\begin{NB}
A generalization:
If $W$ is a codimension 0 wall, then
we have the Hilbert-Chow morphism as a contraction.
If $W$ is not a codimension 0 wall, then
we have a natural birational map
$M_{(\beta_-,\omega_-)}(v) \cdots \to M_{(\beta_+,\omega_+)}(v)$.
Then we see that $\theta_{v,\beta_-,\omega_-}(\xi(\beta,H,t))$
is big. Since it is also nef, 
by the base point free theorem,
we have a birational contraction of 
$M_{(\beta,tH)}(v)$. 
\end{NB}

The following result describe the exceptional locus
of the contraction.
\begin{prop}\label{prop:positive-cone:remark}
Assume that $(\beta', H',t') \in \overline{D(\beta,H,t)} \setminus 
D(\beta,H,t)$.
We set
\begin{equation*}
\begin{split}
M_{(\beta,tH)}(v)^*:= & \{E \in M_{(\beta,tH)}(v) \mid 
\text{ $E$ is $\sigma_{(\beta',t' H')}$-stable } \},\\
K_{(\beta,tH)}(v)^*:= & M_{(\beta,tH)}(v)^* \cap K_{(\beta,tH)}(v).
\end{split}
\end{equation*}
%and $K_{(\beta,tH)}(v)^*:=M_{(\beta,tH)}(v)^* \cap K_{(\beta,tH)}(v)$.
If $\langle \xi(\beta',H',t')^2 \rangle>0$, then
$\theta_{v,\beta,tH}(\xi(\beta',H',t'))$ is ample
over $K_{(\beta,tH)}(v)^*$.   
Thus the exceptional locus of the contraction by
$\theta_{v,\beta,tH}(\xi(\beta',H',t'))$ is contained in 
$K_{(\beta,tH)}(v) \setminus K_{(\beta,tH)}(v)^*$.
\end{prop}

\begin{proof}
Since $\langle \xi(\beta',H',t')^2 \rangle>0$,
we can take a general $(\beta_1,H_1,t_1)$ such that
$\beta_1, H_1 \in \NS(X)_{\Bbb Q}$, $t_1^2 \in {\Bbb Q}$ and
$$
\xi(\beta',H',t')=x \xi(\beta,H,t)+(1-x)\xi(\beta_1,H_1,t_1),\;
x \in (0,1).
$$
Since
\begin{equation*}
M_{(\beta,tH)}(v)^* 
\subset  M_{(\beta,tH)}(v) \cap M_{(\beta_1,t_1 H_1)}(v),
\end{equation*}
we have 
$$
\theta_{v,\beta,tH}(x)_{|M_{(\beta,tH)}(v)^*}=
\theta_{v,\beta_1,t_1 H_1 }(x)_{|M_{(\beta,tH)}(v)^*},\; x \in v^\perp,
$$
where 
$$
\theta_{v,\beta,tH}(x)
:=c_1(p_{M_{(\beta,tH)}(v)*}(\ch({\bf E})p_X^*(x^{\vee}))
 \in \NS(M_{(\beta,tH)}(v)),
$$
(cf. Definition \ref{defn:theta}).
Since $\theta_{v,\beta,tH}(\xi(\beta,H,t))$ and 
$\theta_{v,\beta_1, t_1 H_1}(\xi(\beta_1,H_1, t_1))$
are ample divisors on
$M_{(\beta,tH)}(v)$ and $M_{(\beta_1,t_1 H_1)}(v)$ respectively, 
$\theta_{v,\beta,tH}(\xi(\beta',H',t'))$ is ample
over $M_{(\beta,tH)}(v)^*$.   
\end{proof}

\subsection{The movable cone of $K_{(\beta,tH)}(v)$.}\label{subsect:movable}

\begin{defn}\label{defn:I}
We set
\begin{equation*}
\begin{split}
{\frak I}_k := & \left \{ w  \left|
\begin{aligned} 
& \text{ $w \in H^*(X,{\Bbb Z})_{\alg} $ is primitive}, \\
& \langle w^2 \rangle=0, \langle v,w \rangle=k
\end{aligned}
\right.
\right\},\\
{\frak I}:=& \bigcup_{k=0}^2 {\frak I}_k.
\end{split}
\end{equation*}

\begin{NB}
For $w \in {\frak I}_0$,
\eqref{eq:wall-cond} are not satisfied.
We extend the definition of $W_w$ to this case
by \eqref{eq:wall}. Then we have 
$W_w=\{(\lambda,0)  \}$, where $\lambda H=c_1(w)/\rk w$.
A chamber ${\frak C}$ for generic stabilities 
is a connected component of 
$\overline{\Bbb H} \setminus \cup_{w \in {\frak I}} W_w$.
\end{NB} 
\end{defn}

By the classification of walls in \cite{MYY:2011:1},
the following is obvious.
\begin{lem}\label{lem:wall-intersection}
For $v_1 \in {\frak I}_1$ and $w_1 \in {\frak W}$
with $w_1 \not \in 
\{ iv_1, v-iv_1 \mid 1 \leq i \leq \langle v^2 \rangle/2 \}$,
$$
v_1^\perp \cap w_1^\perp \cap P^+(v^\perp)_{\Bbb R}=\emptyset.
$$
\end{lem}

\begin{proof}
If $v_1^\perp \cap w_1^\perp \cap P^+(v^\perp)_{\Bbb R} \not =\emptyset$,
then we have $\xi \in H^*(X,{\Bbb Z})_{\alg}$ such that
$\xi \in v_1^\perp \cap w_1^\perp \cap v^\perp$ and 
$\langle \xi^2 \rangle>0$.
Since $v_1 \in {\frak I}_1$, we have a decomposition 
$v=\ell v_1+v_2$ where $\langle v_1^2 \rangle=0$,
$\langle v_1,v_2 \rangle=1$ and $\ell=\langle v^2 \rangle/2$.
Then we have a decomposition
$$
\xi^\perp=({\Bbb Z} v_1 \oplus {\Bbb Z}v_2) \perp L.
$$
Since $\langle \xi^2 \rangle>0$,
$L$ is negative definite.
We set 
\begin{equation*}
\begin{split}
w_1:=& x v_1+y v_2+\eta,\\ 
w_2:=& (\ell-x) v_1+(1-y) v_2-\eta,
\end{split}
\end{equation*}
where $x,y \in {\Bbb Z}$ and $\eta \in L$.
By replacing $w_1$ by $w_2$ if necessary,
we may assume that $y \geq 1$.
Since $w_1 \in {\frak W}$,
$\langle w_1^2 \rangle \geq 0, \langle w_2^2 \rangle \geq 0$
and $\langle w_1,w_2 \rangle>0$.
Thus we have $2xy+(\eta^2) \geq 0$,
$2(\ell-x)(1-y)+(\eta^2) \geq 0$ and
$y(\ell-x)+x(1-y)-(\eta^2) >0$.
On the other hand,
\begin{equation*}
\begin{split}
y(\ell-x)+x(1-y)-(\eta^2) \leq &
y(\ell-x)+x(1-y)+2(\ell-x)(1-y)\\
=& \ell(2-y)-x.
\end{split}
\end{equation*}
If $y \geq 2$, then $xy \geq -(\eta^2)/2 \geq 0$
implies that $x \geq 0$.
Hence $\langle w_1,w_2 \rangle \leq 0$.
If $y=1$, then $\langle w_2^2 \rangle \geq 0$ implies
that $\eta=0$.
Hence $w_2=(\ell-x) v_1$ with $0 \leq x <\ell$, which is a contradiction.  
Therefore 
$v_1^\perp \cap w_1^\perp \cap P^+(v^\perp)_{\Bbb R} =\emptyset$.
\end{proof}

\begin{rem}
For $w \in {\frak I}_0$,
we have 
$W_w={\Bbb R}w \cap \overline{P^+(v^\perp)}_{\Bbb R}$.
\begin{NB}
We have
$v^\perp={\Bbb R}w+{\Bbb R}w' \perp L$,
$\langle w,w' \rangle=1$, $\langle {w'}^2 \rangle=0$
and $L$ is negative definite.
Hence the claim holds.
\end{NB}
\end{rem}

We also have the following result.
\begin{prop}\label{prop:intersection2}
Assume that $v$ is a primitive Mukai vector with
$\ell:=\langle v^2 \rangle/2 \geq 4$.
For $v_1, w_1 \in {\frak I}_2$
with $w_1 \not \in {\Bbb Z}v_1$,
$$
v_1^\perp \cap w_1^\perp \cap P^+(v^\perp)_{\Bbb R}=\emptyset.
$$
\end{prop}

\begin{proof}
If $v_1^\perp \cap w_1^\perp \cap P^+(v^\perp)_{\Bbb R}
\ne \emptyset$, we can take $\xi \in H^*(X,{\Bbb Z})_{\alg}$
such that $\xi \in v_1^\perp \cap w_1^\perp \cap v^\perp$
and $\langle \xi^2 \rangle>0$.
\begin{NB}
We set $v_2:=v-v_1$ and $w_2:=v-w_1$.
\end{NB}
Then $v-\ell v_1, v-\ell w_1 \in 
\xi^\perp \cap v^\perp \cap H^*(X,{\Bbb Z})_{\alg}$.
Since $\xi^\perp \cap v^\perp \cap H^*(X,{\Bbb Z})_{\alg}$
is negative definite,
we have $\langle v-\ell v_1,v-\ell w_1 \rangle^2
\leq \langle (v-\ell v_1)^2\rangle \langle (v-\ell w_1)^2 \rangle$.
Moreover if the equality holds, then
$v-\ell v_1=\pm(v-\ell w_1)$ by $\langle (v-\ell v_1)^2 \rangle
=\langle (v-\ell w_1)^2 \rangle=-2\ell$.
In this case, we have $v_1=w_1$ or 
$2v=\ell(v_1+w_1)$.
By the primitivity of $v$ and $\ell \geq 4$, 
the latter case does not occur.
Hence we have
\begin{equation*}
\begin{split}
4 \ell^2=& \langle (v-\ell v_1)^2\rangle \langle (v-\ell w_1)^2 \rangle\\
> & \langle v-\ell v_1,v-\ell w_1 \rangle^2
=(\ell^2 \langle v_1,w_1 \rangle-2\ell)^2.
\end{split}
\end{equation*}
Thus 
$-2\ell <\ell^2 \langle v_1,w_1 \rangle-2\ell<2\ell$, which
implies that $0<\ell \langle v_1,w_1 \rangle<4$.
Since $\ell \geq 4$, this does not occur.
Therefore the claim holds.
\end{proof}

\begin{NB}
\begin{lem}\label{lem:general:I}
Let $v$ be a Mukai vector with $\langle v^2 \rangle>0$.
\begin{enumerate}
\item[(1)]
${\frak I}_0 \ne \emptyset$ if and only if 
there is $\eta \in \NS(X)_{\Bbb Q}$ such that
$(\eta^2)=\langle v^2 \rangle/r^2$ or
$(\eta^2)=0$.
\item[(2)]
Assume that $\sqrt{\ell/n} \not \in {\Bbb Q}$. Then
${\frak I}_k \ne \emptyset$ if and only if 
$\# {\frak I}_k =\infty$.
\end{enumerate}
\end{lem}

\begin{proof}
For $u=e^{\delta+\eta}$,
$\langle u,v \rangle=0$ if and only if 
$(\eta^2)=\langle v^2 \rangle/r^2$.
For $u=(0,\eta,b)$, $u \in v^\perp$
and $\langle u^2 \rangle=0$ if and only if
$(\eta^2)=0$ and $(c_1(v),\eta)=rb$.

By Proposition \ref{prop:stab}, $\Stab_0(v)$ contains an element $g$
of infinite order. Hence (2) is obvious. 
\end{proof}
\end{NB}

\begin{rem}
If $\rk \NS(X) \geq 2$, then 
${\frak I}_0 \ne \emptyset$ 
if and only if $\# {\frak I}_0 =\infty$.
\begin{NB}
If there is $u \in {\frak I}_0$, then
$v^\perp \otimes {\Bbb Q}={\Bbb Q} u+ {\Bbb Q} u' \perp V$,
where $\langle u,u' \rangle=1$, $\langle {u'}^2 \rangle=0$
and $V$ is negative definite.
For $w=x u+y u'+\eta$, $x,y \in {\Bbb Q}, \eta \in V$,
$\langle w^2 \rangle=0$ if and only if
$y=-\frac{1}{2x}\langle \eta^2 \rangle$ or $w=y u'$.  
Since $\rk V=\rk \NS(X)-1$, we get the claim.
\end{NB}
\end{rem}

\begin{defn}
Let $v$ be a primitive Mukai vector.
For $u \in {\frak I}_1 \cup {\frak I}_2$,
%a primitive isotropic Mukai vector $u$ with $\langle v,u \rangle=1,2$,
we set 
$$
d_u:=v-\frac{\langle v^2 \rangle}{\langle v,u \rangle}u.
$$
\end{defn}

\begin{lem}\label{lem:reflection}
Let $u$ be an isotropic Mukai vector with $\langle v,u \rangle=1,2$.
\begin{enumerate}
\item[(1)]
$d_u$ is a primitive vector with 
$\langle d_u^2 \rangle=-\langle v^2 \rangle$.
\item[(2)]
$d_u$ defines a reflection of the lattice $v^{\perp}$:
\begin{equation*}
\begin{matrix}
R_u:& v^{\perp} & \to & v^{\perp}\\
& x & \mapsto & x-\frac{2\langle d_u,x \rangle}
{\langle d_u,d_u \rangle}d_u.
\end{matrix}
\end{equation*}
\end{enumerate}
\end{lem}

\begin{proof}
(1)
We set $d_u=k d_1$, $k \in {\Bbb Z}$.
Then $\langle v^2 \rangle=k^2 \langle d_1^2 \rangle+2k \ell
\frac{2}{\langle v,u \rangle}\langle d_1,u \rangle \in 2k {\Bbb Z}$.
Hence $v=kd_1+\frac{\langle v^2 \rangle}{\langle v,u \rangle}u$ 
is divisible by $k$. 
By the primitivity of $v$, $k=1$.
$\langle d_u^2 \rangle=-\langle v^2 \rangle$ is easy. 

(2)
For $x \in v^\perp$, we have $\langle d_u,x \rangle=
-\frac{\langle v^2 \rangle}{\langle v,u \rangle}\langle x,u \rangle$.
Hence
$$
R_u(x)=x-\frac{2}{\langle v,u \rangle}\langle x,u \rangle
d_u \in H^*(X,{\Bbb Z})_{\alg}.
$$
Obviously $R_u$ preserves the bilinear form.
Therefore $R_u$ is an isometry of 
$v^\perp$.
\end{proof}

\begin{prop}\label{prop:codim0}
Let $v_1$ be an isotropic Mukai vector such that
$\langle v,v_1 \rangle=1$.
Then 
\begin{enumerate}
\item[(1)]
$v=\ell v_1+v_2$, where
$\ell:=\langle v^2 \rangle/2$,
$\langle v_2^2 \rangle=0$ and $\langle v_1,v_2 \rangle=1$.
\item[(2)]
We set $Y:=M_H(v_1)$ and let 
$$
\Phi:{\bf D}(X) \to {\bf D}(Y)
$$
be a Fourier-Mukai transform such that
$\Phi(v_1)=(0,0,-1)$ and 
$\Phi(v_2)=(1,0,0)$.
Then the contravariant Fourier-Mukai
transform $\Psi:=[1] \circ \Phi^{-1} \circ {\cal D}_Y \circ \Phi$
gives an isometry $\Psi$ of $H^*(X,{\Bbb Z})_{\alg}$ such that
$$
\Psi_{|v^\perp}=R_{v_1}.
$$
\end{enumerate}
\end{prop}

\begin{proof}
We have a decomposition 
$$
H^*(X,{\Bbb Z})_{\alg}=({\Bbb Z}v_1 +{\Bbb Z}v_2) \perp L,
$$
where $L$ is a lattice with $\Phi(L)=\NS(Y)$.
Hence we see that 
\begin{equation*}
\begin{split}
\Psi(w)=& w,\;w \in L,\\ 
\Psi(v_1)=& -v_1,\\
\Psi(v_2)=& -v_2.
\end{split}
\end{equation*}
Since $v, d_{v_1} \in ({\Bbb Z}v_1 +{\Bbb Z}v_2)$,
we get the claim.  
\end{proof}

\begin{defn}
Let $W$ be a wall for $v$. 
Let $(\beta,H,t)$ be a point of $W$ and
$(\beta',H',t')$ be a point in an adjacent chamber. 
Then we define \emph{the codimension of the wall} $W$ by 
\begin{equation*}
\codim W:=\min_{v=\sum_i v_i } \left \{ \sum_{i<j}\langle v_i,v_j \rangle
-\left(\sum_i (\dim {\cal M}_H^{\beta'}(v_i)^{ss}-
\langle v_i^2 \rangle) \right)
+1 \right\},
\end{equation*}
where $v=\sum_i v_i$ are decompositions of $v$
such that $\phi_{(\beta,tH)}(v)=\phi_{(\beta,tH)}(v_i)$
and $\phi_{(\beta',t' H')}(v_i)>\phi_{(\beta',t' H')}(v_j)$,
$i<j$.
\end{defn}

If $\codim W \geq 2$, then the proof of \cite[Prop. 4.3.5]{MYY:2011:1}
implies that
$$
\dim \{ E \in M_{(\beta',t' H')}(v) \mid \text 
{$E$ is not $\sigma_{(\beta,tH)}$-stable }\}
\leq \langle v^2 \rangle.
$$

\begin{prop}\label{prop:codim0-div}
Let $v$ be a primitive Mukai vector with $\langle v^2 \rangle \geq 6$.
Let $W$ be a wall for $v$ and take $(\beta,H,t) \in W$
such that $\beta \in \NS(X)_{\Bbb Q}$ and $H \in \Amp(X)_{\Bbb Q}$. 
\begin{enumerate}
\item[(1)]
$W$ is a codimension 0 wall if and only if
$W$ is defined by
$v_1$ such that
\begin{equation*}
v=\ell v_1+v_2,\; \langle v_1^2 \rangle=\langle v_2^2 \rangle=0,\;
\langle v_1,v_2 \rangle=1,\;\ell=\langle v^2 \rangle/2.
\end{equation*}
\item[(2)]
We take $(\beta_\pm,\omega_\pm)$ from
chambers separated by the codimension 0 wall $W$ in (1).
We may assume that
$\phi_{(\beta_+,\omega_+)}(E_1)<\phi_{(\beta_+,\omega_+)}(E)$
for $E_1 \in M_{(\beta,\omega)}(v_1)$ and
$E \in M_{(\beta_+,\omega_+)}(v)$.
Then 
\begin{enumerate}
\item
$\theta_{v,\beta_\pm,\omega_\pm}(\xi(\beta,H,t))$ give divisorial contractions.
\item
Let $D_\pm \subset M_{(\beta_\pm,\omega_\pm)}(v)$
be the exceptional divisors of the contractions.
Then $D_\pm$ are irreducible divisors such that
$$
(D_\pm)_{|K_{(\beta_\pm,\omega_\pm)}(v)}
=\pm 2\theta_{v,\beta_\pm,\omega_\pm}(d_{v_1}) 
\in \NS(K_{(\beta_\pm,\omega_\pm)}(v)).
$$
\end{enumerate}
\end{enumerate}
\end{prop}

\begin{proof}
(1) is a consequence of \cite[Lem. 4.3.4 (2)]{MYY:2011:1}.

(2)
Let $\Phi:{\bf D}(X) \to {\bf D}(Y)$ be the Fourier-Mukai transform
in Proposition \ref{prop:codim0}.
By \cite[Prop. 4.1.4]{MYY:2011:1}, $\Phi$ induces an isomorphism
$$
M_{(\beta_+,\omega_+)}(v) \to \Pic^0(Y) \times \Hilb_Y^{\ell}.
$$
By using the Hilbert-Chow morphism
$\Hilb_Y^\ell \to S^\ell Y$, 
we have a divisorial contraction
$$
M_{(\beta_+,\omega_+)}(v) \to \Pic^0(Y) \times \Hilb_Y^{\ell}
\to \Pic^0(Y) \times S^\ell Y.
$$
By using Lemma \ref{lem:H-C}, we get the claim (a).
Since ${\cal D}_Y \circ \Phi$ gives an isomorphism
$$
M_{(\beta_-,\omega_-)}(v) \to \Pic^0(Y) \times \Hilb_Y^{\ell}
$$ 
by \cite[Prop. 4.1.4]{MYY:2011:1},
we also get the claim (a).
(b) is a consequence of Lemma \ref{lem:H-C} below.
\end{proof}

\begin{lem}\label{lem:H-C} 
Let $v=(1,0,-\ell)$ be a primitive Mukai vector with
$\ell \geq 3$.
We set $v_1=(0,0,-1)$ and $v_2=(1,0,0)$.
Then 
$$
d_{v_1}=v-2\ell v_1=(1,0,\ell).
$$ 
For the Hilbert-Chow morphism
$\Hilb_X^{\ell} \to S^{\ell} X$,
the exceptional divisor $D$ is divisible by 2
and satisfies 
$$
D_{|K_H(v)}=2\theta_v((1,0,\ell))=
2 \theta_v(d_{v_1}).
$$
\end{lem}

\begin{prop}\label{prop:general:codim1-div}
Let $v$ be a primitive Mukai vector with
$\langle v^2 \rangle \geq 6$.
Let $W$ be a wall for $v$ and take $(\beta,H,t) \in W$
such that $\beta \in \NS(X)_{\Bbb Q}$ and $H \in \Amp(X)_{\Bbb Q}$. 
\begin{enumerate}
\item[(1)]
$W$ is a codimension 1 wall if and only if
$W$ is defined by
$v_1$ such that
\begin{enumerate}
\item[(i)]
 $v=v_1+v_2$, $\langle v_1^2 \rangle=0$,
$\langle v_2^2 \rangle \geq 0$,
$\langle v,v_1 \rangle=2$ and $v_1$ is primitive
or 
\item[(ii)]
$v=v_1+v_2+v_3$, $\langle v_i^2 \rangle=0$ ($i=1,2,3$),
$\langle v_1,v_2 \rangle=\langle v_2,v_3 \rangle=
\langle v_3,v_1 \rangle=1$.
\end{enumerate}
For the second case, $\langle v^2 \rangle=6$,
$W_{v_1}$ and $W_{v_2}$ intersect transversely and
$(\beta,H,t) \in W_{v_1} \cap W_{v_2}$.
\item[(2)]
Assume that $(\beta,H,t)$ belongs to exactly one wall $W$.
We take $(\beta_\pm,\omega_\pm)$ from
chambers separated by the codimension 1 wall $W$ in (1).
In the notation 
of (1),
we may assume that
$\phi_{(\beta_+,\omega_+)}(E_1)<\phi_{(\beta_+,\omega_+)}(E)$
for $E_1 \in M_{(\beta,\omega)}(v_1)$ and
$E \in M_{(\beta_+,\omega_+)}(v)$.
We set
\begin{equation*}
\begin{split}
D_+:=\{E \in M_{(\beta_+,\omega_+)}(v) \mid
\Hom(E_1,E) \ne 0 , E_1 \in M_{(\beta,\omega)}(v_1) \},\\
D_-:=\{E \in M_{(\beta_-,\omega_-)}(v) \mid
\Hom(E,E_1) \ne 0 , E_1 \in M_{(\beta,\omega)}(v_1) \}.
\end{split}
\end{equation*} 
Then 
\begin{enumerate}
\item
$D_\pm$ are non-empty and irreducible divisors.
\item
$(D_\pm)_{|K_{(\beta_\pm,\omega_\pm)}(v)}
=\pm \theta_{v,\beta_\pm,\omega_\pm}(d_{v_1}) 
\in \NS(K_{(\beta_\pm,\omega_\pm)}(v))$.
In particular, $D_\pm$ are primitive.
\item
$\theta_{v,\beta_\pm,\omega_\pm}(\xi(\beta,H,t))$ gives 
contractions of $D_\pm$. 
\end{enumerate}
\end{enumerate}
\end{prop}

\begin{proof}
(1)
The classification of codimension 1 walls in 
\cite[Lem. 4.3.4 (2), Prop. 4.3.5]{MYY:2011:1} imply that
$W$ is defined by $v_1$ with the required properties.
It is easy to see that $v_1,v_2,v_3$ spans a lattice of rank 3
and $v_1^\perp \cap v_2^\perp={\Bbb Z}(v_3-v_1-v_2)$.
Hence $W_{v_1}$ and $W_{v_2}$ intersect transversely.
 
(2)
$D_+= \theta_{v,\beta_+,\omega_+}(d_{v_1})$ is a consequence of
\cite[Lem. 4.5.1]{MYY:2011:1} and
$D_-= -\theta_{v,\beta_-,\omega_-}(d_{v_1})$ 
follows from \cite[Prop. 4.5.2]{MYY:2011:1}. 
By Lemma \ref{lem:reflection} (1), 
$D_\pm$ are primitive.

The non-emptiness and
the contractibility of $D_\pm$ are showed in the proof of
\cite[Cor. 5.17] {MYY:2011:2}.
\end{proof}

\begin{rem}\label{rem:isom}
For the wall $W$ in 
Proposition \ref{prop:general:codim1-div},
we have an isomorphism 
$$
f:M_{(\beta_+,\omega_+)}(v)
\to M_{(\beta_-,\omega_-)}(v)
$$
such that $f_* \circ \theta_{v,\beta_+,\omega_+}=
\theta_{v,\beta_-,\omega_-} \circ R_{v_1}$:
Indeed we have a map $f$ as a birational map. 
We set $\omega_+:=tH_+$.
Then $\theta_{v,\beta_+,\omega_+}(\xi(\beta_+,H_+,t)) 
\in \NS(M_{(\beta_+,\omega_+)}(v))$ 
is relatively ample over $X \times \widehat{X}$.
Since $\theta_{v,\beta_-,\omega_-}(R_{v_1}(\xi(\beta_+,H_+,t)))
\in \NS(M_{(\beta_-,\omega_-)}(v))$ is 
relatively ample over $X \times \widehat{X}$,
$f$ is an isomorphism.
\begin{NB}
The relative graded rings are isomorphic. 
\end{NB}
For the wall $W$ in Proposition \ref{prop:codim0-div}, 
we also have a similar isomorphism $f$
by using Fourier-Mukai transform $\Psi$
in Proposition \ref{prop:codim0}.
\end{rem}

%Let
%$\overline{P^+(K_H(v))}_{\Bbb Q}$ be the cone generated by
%rational points of the closure of $P^+(K_H(v))$.

The following results fit in the general result
of Markman \cite{Mark} on the movable cone
of irreducible symplectic manifolds.

\begin{thm}\label{thm:general:movable}
Let $(X,H)$ be a polarized abelian surface $X$. 
Let $v$ be a primitive Mukai vector with $\langle v^2 \rangle \geq 6$.
We take 
$(\beta,H,t) \in {\frak H} 
=\NS(X)_{\Bbb R} \times C(\Amp(X)_{\Bbb R}) \times{\Bbb R}_{>0}$
such that $\xi(\beta,H,t) \in 
P^+(v^\perp)_{\Bbb R}
\setminus \cup_{ u \in {\frak W}} u^{\perp}$,
where ${\frak W}$ is the set of Mukai vectors
satisfying \eqref{eq:wall-cond}.
Thus $(\beta,tH)$ is general (see Definition \ref{defn:wall}).
%\begin{equation*}
%\begin{split}
%(\beta,H,t) \in & {\frak H} \setminus \cup_{v_1 \in {\frak W}}W_{v_1}\\
%= & \NS(X)_{\Bbb R} \times C(\Amp(X)_{\Bbb R}) \times{\Bbb R}_{>0} 
%\setminus \cup_{v_1 \in {\frak M}} \Xi^{-1}(v_1^\perp)
%\end{split}
%\end{equation*}
%such that $\beta \in \NS(X)_{\Bbb Q}$ and $tH \in \Amp(X)_{\Bbb Q}$,
%where ${\frak W}$ is the set of Mukai vectors
%satisfying \eqref{eq:wall-cond} and 
%$W_{v_1}$ is the
%wall defined by $v_1 \in {\frak W}$ (see Definition \ref{defn:wall}). 
\begin{enumerate}
\item[(1)]
Let ${\cal D}(\beta,tH)$ be a connected component
of  $P^+(v^\perp)_{\Bbb R}
\setminus \cup_{ u \in {\frak I}} u^{\perp}$
containing $\xi(\beta,H,t)$ (cf. Definition \ref{defn:I}).
Then 
\begin{equation*}
\begin{split}
\overline{\Mov(K_{(\beta,tH)}(v))}_{\Bbb R}=
& \theta_{v,\beta,tH}(\overline{{\cal D}(\beta,tH)}).
\end{split}
\end{equation*}
Moreover 
$$
\theta_{v,\beta,tH}(H^*(X,{\Bbb Z})_{\alg}
\cap \overline{{\cal D}(\beta,tH)}) \subset 
\Mov(K_{(\beta,tH)}(v)).
$$  
\item[(2)]
We choose $u \in {\frak I}_i$ $(i=1,2)$.
Let $x \in H^*(X,{\Bbb Z})_{\alg}
\cap \overline{{\cal D}(\beta,tH)}$ be a general element
of the boundary
defined by $\langle x, u \rangle=0$.
Then 
\begin{enumerate}
\item
$\theta_v(x)$ defines a divisorial contraction from a birational
model $K_{(\beta',t' H')}(v)$ of $K_{(\beta,tH)}(v)$.
\item
The exceptional divisor of the contraction is primitive 
in $\NS(K_{(\beta,tH)}(v))$ if
$u \in {\frak I}_2$,
and the exceptional divisor is divisible by 2 
in $\NS(K_{(\beta,tH)}(v))$ if
$u \in {\frak I}_1$.
\end{enumerate}
\end{enumerate}
\end{thm}

\begin{proof}
(1)
We note that
\begin{equation*}
\Mov(K_{(\beta,tH)}(v))_{\Bbb R} \subset 
C(\overline{P^+(K_{(\beta,tH)}(v))}_{\Bbb R})=
\theta_{v,\beta,tH}(\overline{P^+(v^\perp)}_{\Bbb R})
\end{equation*}
and
\begin{equation}\label{eq:D-calD}
\bigcup_{\xi(\beta',H',t') \in {\cal D}(\beta,tH)} 
\overline{D(\beta',H',t')}=
\overline{{\cal D}(\beta,tH)}.
\end{equation}
\begin{NB}
Assume that $(\beta,\omega),(\beta',\omega')$
does not lie on walls.
By the assumption,
$M_{(\beta,\omega)}(v) \setminus M_{(\beta',\omega')}(v)$
and
$M_{(\beta',\omega')}(v) \setminus M_{(\beta,\omega)}(v)$
are proper closed subsets of codimension $\geq 2$. 
Hence we have a natural birational map
$M_{(\beta,\omega)}(v) \to M_{(\beta',\omega')}(v)$
with an isomorphism 
$\NS(M_{(\beta,\omega)}(v)) \to \NS(M_{(\beta',\omega')}(v))$.
\end{NB}
Assume that $\xi(\beta',H',t') \in {\cal D}(\beta,tH)$.
Since there is no wall of codimension 0,1, 
we have a natural birational identification
$K_{(\beta,tH)}(v) \cdots \to K_{(\beta',t' H')}(v)$
with an identification
$\NS(K_{(\beta,tH)}(v)) \to \NS(K_{(\beta',t' H'')}(v))$.
Hence 
$$
\bigcup_{\xi(\beta',H',t') \in {\cal D}(\beta,tH)} \theta_{v,\beta,tH}
(D(\beta',H',t'))
 \subset 
\Mov(K_{(\beta,tH)}(v))_{\Bbb R}
$$
and
$$
\theta_{v,\beta,tH}(\overline{{\cal D}(\beta,tH)})=
\bigcup_{\xi(\beta',H', t') \in {\cal D}(\beta,tH)} \theta_{v,\beta,tH}
(\overline{D(\beta',H',t')})
 \subset 
\overline{\Mov(K_{(\beta,tH)}(v))}_{\Bbb R}.
$$
\begin{NB}
$\cup_{(\beta',\omega')} \Nef(K_{(\beta',\omega')}(v)) \subset 
\Mov(K_{(\beta,tH)}(v))$.
\end{NB}
Assume that $\xi(\beta',H',t') \not \in \overline{{\cal D}(\beta,tH)}$.
We set 
$$\eta_x:=x \xi(\beta',H',t')+(1-x)\xi(\beta,H,t),\;
x \in [0,1].
$$
If $\theta_{v,\beta,tH}(\xi(\beta',H',t'))$ is movable,
then $L_x:=\theta_{v,\beta,tH}(\eta_x)$
is movable for $0 \leq x \leq 1$.
We take an adjacent wall $u^\perp$ $(u \in {\frak I})$
of ${\cal D}(\beta,tH)$ separating 
$\xi(\beta',H',t')$ and $\xi(\beta,H,t)$.
Then we find $x_0 \in {\Bbb Q} \cap (0,1)$
such that $\langle \eta_{x_0},u \rangle=0$.
Let $D(\beta_1,H_1,t_1) \subset {\cal D}(\beta,tH)$
be the chamber such that $\eta_x \in \xi(D(\beta_1,H_1,t_1))$ 
for $x'  <x<x_0$ with $x'<x_0$.
Since all walls $W$ between $\eta_0$ and $\eta_{x_0}$
satisfy $\codim W \geq 2$,  
we have a natural birational map
$\varphi:K_{(\beta,tH)}(v) \cdots \to K_{(\beta_1,t_1 H_1)}(v)$
which induces a commutative diagram
\begin{equation*}
\begin{CD}
v^\perp @= v^\perp\\
@V{\theta_{v,\beta,tH}}VV @VV{\theta_{v,\beta_1,t_1 H_1}}V\\
\NS(K_{(\beta,tH)}(v)) @>{\varphi_*}>> \NS(K_{(\beta_1,t_1 H_1)}(v))
\end{CD}
\end{equation*}
Since $(\eta_{x_0}^2)>0$,
$\varphi_*(L_{x_0})$
gives a divisorial contraction of $K_{(\beta_1,t_1 H_1)}(v)$.
Let $C (\subset K_{(\beta_1,t_1 H_1)}(v))$ 
be a general curve contracted by
$\varphi_*(L_{x_0})$.
\begin{NB}
Let $D$ be the exceptional divisor.
Then every fiber of the contraction
$f:E \to f(E)$ is connected and $\dim f(E)<\dim E$.
Hence every fiber contains an irreducible curve $C$.
For a point of $C$ which is not contained the base locus 
of $|m\varphi_*(L_1)|$ $(m \gg 0)$,
we have $D \in |m\varphi_*(L_1)|$ with
$x \not \in D$.
Then $(\varphi_*(L_1),C) \geq 0$.
\end{NB}
Since $L_1=
\theta_{v,\beta,tH}(\xi(\beta',H',t'))$ is movable,
$\varphi_*(L_1)$ is also movable. Hence 
we may assume that $C$ is not contained in its base locus.
Then $(\varphi_*(L_1),C) \geq 0$. 
Since $(\varphi_*(L_x),C)>0$ for $x'<x<x_0$, 
we have $(\varphi_*(L_{x_0}),C)>0$, which is a contradiction.
Hence 
$\theta_{v,\beta,tH}(\xi(\beta',H',t'))$ is not movable.

Assume that $\xi \in H^*(X,{\Bbb Z})_{\alg}$ belongs to
$\overline{{\cal D}(\beta,tH)}$.
We take $\overline{D(\beta',H',t')}$ containing $\xi$
by \eqref{eq:D-calD}.
If $(\xi^2)>0$, then Corollary \ref{cor:positive-cone}
implies that
$\theta_{v,\beta,tH}(\xi)$ 
gives a birational contraction of $K_{(\beta',t' H')}(v)$.
Hence it is movable.
\begin{NB}
If $(\xi^2)>0$, then
there are finitely many walls in a neighborhood of $x$.
Hence we can take $\xi_1, \xi_2 \in P^+(v^\perp)_{\Bbb Q}$ such that
$\xi_1 \in {\cal D}(\beta,tH)$, $\xi_2$ belongs to another chamber
${\cal D}(\beta',t' H')$ and 
the segment connecting $\xi_1,\xi_2$ passes $\xi$.
Then $\theta_{v,\beta,tH}(\xi)$ is nef and big, 
and gives a divisorial contraction.
In particular, $\theta_{v,\beta,tH}(\xi) \in 
\Mov(K_{(\beta,tH)}(v))_{\Bbb Q}$.
\end{NB}

If $(\xi^2)=0$, then see Proposition \ref{prop:Lag}.  
Therefore (1) holds.
\begin{NB}
If $(\beta',t' H')$ and $(\beta,tH)$ 
adjacent chambers separated by a wall of codimension 0,1, then
$\Nef(K_{(\beta',t' H')}(v)) \cong \theta_{v,\beta,tH}(D(\beta',H',t'))$ 
is not contained in
$\Mov(K_{(\beta,tH)}(v))$.
Since $\cup_{(\beta',H',t')} \theta_{v,\beta,tH}(\overline{D(\beta',H',t')})
=\overline{P^+(K_{(\beta,\omega)}(v))}$
by Proposition \ref{prop:positive-cone}, we get
the first claim.
\end{NB}

\begin{NB}
Let $\Phi:{\bf D}(X) \to {\bf D}(X)$ be a Fourier-Mukai transform
preserving $\pm v$.
We set
$(\beta',\omega'):=\Phi((\beta,\omega))$.
Then we have an isomorphism
$\Phi:M_{(\beta,\omega)}(v) \to M_{(\beta',\omega')}(v)$, which induces
a birational map  
$M_{(\beta,\omega)}(v) 
\overset{\Phi}{\to} M_{(\beta',\omega')}(v) \cdots \to 
M_{(\beta,\omega)}(v)$.
Since $\theta_v:v^{\perp} \to \NS(K_{(\beta,\omega)}(v))$
is compatible with respect to the action of $\Phi$,
we have an action of $\Stab_0(v)^*$ on
$\NS(K_{(\beta,\omega)}(v))$. 
\end{NB}

(2) is a consequence of Proposition \ref{prop:codim0-div}
and \ref{prop:general:codim1-div}.
\end{proof}

\begin{rem}
For $u \in {\frak I}_0$,
$u^\perp$ is a tangent of $\overline{P^+(v^\perp)}_{\Bbb R}$.
\end{rem}

\begin{cor}\label{cor:birational-model}
Keep notations in Theorem \ref{thm:general:movable}.
\begin{enumerate}
\item[(1)]
Let $(K,L)$ be a pair of a smooth manifold
$K$ with a trivial canonical bundle and an ample divisor $L$ on $K$.
If $K$ is birationally equivalent to
$K_{(\beta,tH)}(v)$, then there is
$\xi(\beta', H',t') \in {\cal D}(\beta,tH)$
such that $K \cong K_{(\beta',t' H')}(v)$ and
${\Bbb R}_{>0}L$ corresponds to 
${\Bbb R}_{>0}\theta_{v,\beta,\omega}(\xi(\beta',H',t'))$.
\item[(2)]
Let $(M,L)$ be a pair of smooth manifold
$M$ with a trivial canonical bundle
and an ample divisor $L$ on $M$.
If $M$ is birationally equivalent to
$M_{(\beta,tH)}(v)$, then there is $\xi(\beta', H',t')
\in {\cal D}(\beta,tH)$
such that $M \cong M_{(\beta',t' H')}(v)$ and
${\Bbb R}_{>0}L$ corresponds to 
${\Bbb R}_{>0}\theta_{v,\beta,\omega}(\xi(\beta',H',t'))$ up to
line bundles coming from $\Alb(M_{(\beta',t' H')}(v))$.
\end{enumerate}
\end{cor}

\begin{proof}
(1)
Since the canonical bundles of 
$K$ and $K_{(\beta,tH)}(v)$ are trivial,
we have a birational map
$$
f:K \cdots \to K_{(\beta,tH)}(v)
$$
 with an isomorphism
$f_*:\NS(K) \to \NS(K_{(\beta,tH)}(v))$.
\begin{NB}
We have an isomorphism $H^0(K,\Omega_K^2) \cong
H^0(K_{(\beta,tH)}(v),\Omega_{K_{(\beta,tH)}(v)}^2) \cong {\Bbb C}$.
Hence $K$ has a holomorphic 2-form which is non-degenerate
outside of the exceptional locus, which is of codimension 2.
Hence $K$ has a symplectic form. 
\end{NB}
Then $f_*(L) \in \Mov(K_{(\beta,tH)}(v))_{\Bbb Q}$.
We take $(\beta',H',t') \in {\frak H}$ such that 
$f_*(L) \in {\Bbb R}_{>0}\theta_{v,\beta,tH}(\xi(\beta',H',t'))$ 
and $\beta', H' \in \NS(X)_{\Bbb Q}$
(see Remark \ref{rem:positive}).
Assume that $\xi(\beta',H',t') 
\in \overline{D(\beta_1,H_1,t_1)} 
(\subset \overline{{\cal D}(\beta,tH)})$.
We take an ample divisor $L_0$ on $K$
such that $f_*(L_0) \in  
\theta_{v,\beta,tH}(D(\beta_1,H_1,t_1))$.
For the birational map
$g:K_{(\beta,tH)}(v) \to K_{(\beta_1,t_1 H_1)}(v)$,
$L_1:=(g \circ f)_*(L_0)$ is ample.
We note that $g \circ f$ induces an isomorphism
$$
\bigoplus_{n \geq 0} H^0(K,{\cal O}_K(nL_0)) \cong
\bigoplus_{n \geq 0} H^0(K_{(\beta_1,t_1 H_1)}(v),
{\cal O}_{K_{(\beta_1,t_1 H_1)}(v)}(nL_1)).
$$
Then the ampleness of $L_0$ and $L_1$ imply that
$g \circ f:K \to K_{(\beta_1,t_1 H_1)}(v)$ is an isomorphism.
Since $L$ is ample,
$\theta_{v,\beta,tH}(\xi(\beta',H',t'))$
is also ample on $K_{(\beta_1,t_1 H_1)}(v)$, which implies that
$\xi(\beta',H',t') \in D(\beta_1,H_1,t_1)$.
Hence $K \cong K_{(\beta_1,t_1 H_1)}(v) \cong
K_{(\beta',t' H')}(v)$.

(2)
For a birational map $\varphi:M \cdots \to M_{(\beta,tH)}(v)$,
we have a commutative diagram
\begin{equation*}
\begin{CD}
M & \overset{\varphi}{\cdots \longrightarrow} & M_{(\beta,tH)}(v)\\
@VVV @VVV\\
\Alb(M) @>{\eta}>> \Alb(M_{(\beta,tH)}(v))
\end{CD}
\end{equation*}
where $\eta$ is an isomorphism. 

Let $K$ be a smooth fiber of the albanese map
of $M$.
Then the canonical bundle of $K$ is trivial.
For a general smooth fiber $K$,
$\varphi$ induces a birational map
$K \to K_{(\beta,tH)}(v)$.
%Hence $K$ is a holomorphic symplectic manifold.
There is $\xi(\beta',H',t') \in D(\beta,tH)$ such that
for the birational map
$\psi:M_{(\beta,tH)}(v) \cdots \to M_{(\beta',t' H')}(v)$,
$\psi \circ \varphi$ indices an isomorphism
$K  \to K_{(\beta',t' H')}(v)$.
Thus $(\psi \circ \varphi)_*(L)_{|K_{(\beta',t' H')}(v)}$
is ample.
Since the albanese map of $M_{(\beta',t'H')}(v)$
is isotrivial, 
$(\psi \circ \varphi)_*(L)$ is relatively ample over
$\Alb(M_{(\beta',t'H')}(v))$.
As in the proof of (1),
by looking at relative global sections,
we see the 
$\psi \circ \varphi$ is an isomorphism.
\end{proof}

\begin{cor}\label{cor:movable}
Keep notations in Theorem \ref{thm:general:movable}
\begin{enumerate}
\item[(1)]
Assume that ${\frak I}_1={\frak I}_2= \emptyset$, that is,
\begin{equation*}
\min \{ \langle v,w \rangle>0 \mid \langle w^2 \rangle=0 \} \geq 3.
\end{equation*}
Then $$
{\cal D}(\beta,tH)=P^+(v^\perp)_{\Bbb R}.
%\{x \in v^\perp \otimes {\Bbb R} \mid
%\langle x^2 \rangle \geq 0, \langle x,H+(H,\delta)\varrho_X \rangle>0 \}.
$$
In particular, $\Mov(K_{(\beta,tH)}(v))_{\Bbb Q}=
\overline{P^+(K_{(\beta,tH)}(v))}_{\Bbb Q}$.
\item[(2)]
Assume that ${\frak I}_1 = \emptyset$ and ${\frak I}_2 \ne \emptyset$,
that is,
\begin{equation*}
\min \{ \langle v,w \rangle>0 \mid \langle w^2 \rangle=0 \}=2.
\end{equation*}
For every divisorial contraction
from a birational model 
of $K_{(\beta,tH)}(v)$,
the exceptional divisor is primitive in $\NS(K_{(\beta,tH)}(v))$. 
\item[(3)]
Assume that ${\frak I}_1 \ne \emptyset$, that is,
\begin{equation*}
\min \{ \langle v,w \rangle>0 \mid \langle w^2 \rangle=0 \}=1.
\end{equation*}
Then there is a divisorial contraction from a birational model
of $K_{(\beta,tH)}(v)$ such that the exceptional divisor
is a prime divisor and divisible by 2 
in $\NS(K_{(\beta,tH)}(v))$. 
\end{enumerate}
\end{cor}

Let us study the structure of walls in a neighborhood of a rational
point of $\overline{P^+(v^\perp)}_{\Bbb R} \setminus P^+(v^\perp)_{\Bbb R}$.
Let $A$ be a compact subset of $P^+(v^\perp)_{\Bbb R}$
and $u$ an isotropic Mukai vector in the boundary 
of $\overline{P^+(v^\perp)}_{\Bbb R}$.
Let $\overline{u A}$ be the cone spanned by $u$ and $A$.
We note that ${\frak W}_A:=\{ v_1 \in {\frak W} \mid  
A \cap v_1^\perp \ne \emptyset \}$ is a finite set.
\begin{NB}
For $v_1 \in {\frak W}$,
$\langle v,v_1 \rangle \geq \langle v-v_1,v_1 \rangle>0$.
Hence $\langle v^2 \rangle>\langle v,v_1 \rangle>0$.
We set $\delta(v_1):=\langle v^2 \rangle v_1-\langle v,v_1 \rangle v$.
Then $\delta(v_1) \in v^\perp$ and
$\langle \delta(v_1)^2 \rangle=
\langle v^2 \rangle(\langle v^2 \rangle \langle v_1^2 \rangle-
\langle v,v_1 \rangle^2)>-\langle v^2 \rangle^2$.
Hence ${\frak W}_A$ is finite. 
\end{NB}
We fix a point $a \in A$ which does not lie on any wall and 
assume that there is no wall between $u$ and $a$
(cf. Remark \ref{rem:beta}).

\begin{lem}\label{lem:B}
$v_1 \in {\frak W}$ satisfies
$v_1^\perp \cap \overline{u A} \ne \emptyset$
if and only if $v_1^\perp \cap (A \cup \{ u \}) \ne \emptyset$.
\end{lem}

\begin{proof}
Assume that $w \in v_1^\perp \cap \overline{u A}$.
We take $w' \in A$ such that $w$ belongs to the segment
$\overline{u w'}$
connecting $u$ and $w'$.
Assume that $\langle v_1,u \rangle \ne 0$.
If $v_1^\perp \cap A =\emptyset$, then
$\langle v_1,u \rangle \langle v_1,w' \rangle<0$. 
Since there is no wall between $u$ and $a$,
we have $\langle v_1, u \rangle \langle v_1,a \rangle>0$.
Then $w'$ and $a$ are separated by the hyperplane
$v_1^\perp$. Therefore there is a point $x \in A$ with
$x \in v_1^\perp$, which is a contradiction.
Hence $v_1^\perp \cap A  \ne \emptyset$.
\end{proof}

\begin{NB}
Lemma \ref{lem:B} should be replaced by
the following result.

\begin{lem}
Let $\Delta$ be a finite polyhedral cone in 
$\overline{P^+(v^\perp)}_{\Bbb R}$.
Then there are finitely many walls intersection $\Delta$.
\end{lem}
This lemma can be applied 
to the assumption of Proposition \ref{prop:Lag}.

\end{NB}
 
\begin{prop}
\begin{enumerate}
\item[(1)]
$\{ v_1 \in {\frak W} \mid 
v_1^\perp  \cap (\overline{uA} \setminus \{ u \})
\ne \emptyset\}$
is a finite set. 
\item[(2)]
There is an open neighborhood of $u$ 
such that 
$$
\{ v_1 \in {\frak W} \mid 
v_1^\perp  \cap U \cap (\overline{uA} \setminus \{ u \})
\ne \emptyset\} \subset u^\perp.
$$
In particular the set of walls is finite in
$U \cap \overline{uA}$ and all walls pass the point $u$.
\end{enumerate}
\end{prop}

\begin{proof}
By Lemma \ref{lem:B},
$$
\{ v_1 \in {\frak W} \mid 
v_1^\perp \cap (\overline{uA} \setminus \{ u \})
\ne \emptyset\} \subset {\frak W}_A.
$$
Hence (1) holds.
(2) easily follow from (1).
\end{proof}

\begin{rem}\label{rem:beta}
Assume that $e^\beta$ satisfies 
$\langle e^\beta, v \rangle=0$.
Then we can set $v=r e^\beta+\xi+(\xi,\beta)\varrho_X$.
We set 
$$
B:=\{v_1 \in {\frak W} \mid e^\beta \in v_1^\perp \}.
$$ 
For $v_1 \in B$ and $v_2:=v-v_1 \in B$, 
we can set 
\begin{equation*}
\begin{split}
v_1:= & r_1 e^\beta+\xi_1+(\xi_1,\beta)\varrho_X,\\
v_2:= & r_2 e^\beta+\xi_2+(\xi_2,\beta)\varrho_X.
\end{split}
\end{equation*}
Since $0 <\langle v_1,v_2 \rangle=(\xi_1,\xi_2)$,
$\xi_1 \ne 0$ and $\xi_2 \ne 0$.
If $(\beta,\omega')$ belongs to the wall defined by $v_1$,
then $(\xi_1^2),(\xi_2^2) \geq 0$ implies that
both of $n \xi_1$ and $n \xi_2$ are effective, or
both of $- n \xi_1$ and $-n \xi_2$ are effective,
where $n$ is the denominator of $\beta$.
In particular, $|(\xi_1,H)| <|(\xi,H)|$
for all ample divisor $H$.
Then we also see that the set of $\xi_1$ is finite.

If $r_1=r\frac{(\xi_1,H)}{(\xi,H)}$, then
$|r_1|<|r|$. Therefore 
$$
B':=\{v_1 \in B \mid (r_1 \xi-r \xi_1,H)=0 
\text{ for some $H \in \Amp(X)_{\Bbb Q}$ }\}
$$ 
is a finite set.
%In particular $(r_1 \xi-r \xi_1,H) \ne 0$ for a general $H \in \Amp(X)$.
We take $H \in \Amp(X)_{\Bbb Q}$ such that 
$(H,r_1 \xi-r \xi_1) \ne 0$
for all $v_1 \in B'$.
Then $(\beta,H,t)$ $(t \ll 1)$ belongs to a chamber.
\end{rem}

\begin{NB}

We take $w \in {\frak I}_0$
such that $c_1(w)/\rk w=c_1(u)/\rk u$, $\rk w >0$.
We set $Y :=M_H(w)$ and let ${\bf E}$ be the universal family
as twisted objects.
Then we have (twisted) Fourier-Mukai transforms 
$\Phi_{X \to Y}^{{\bf E}^{\vee}[k] }:
{\bf D}(X) \to {\bf D}^{\alpha}(Y)$,
where $\alpha$ are suitable 2-cocycles of ${\cal O}_{Y}^{\times}$
defining ${\bf E}^{\vee}$
and $k=1$ for $(\delta-\beta,H)>0$, $k=2$ for $(\delta-\beta,H)<0$.
We take $v_1 \in {\frak W}$ with $w \in v_1^\perp$.
Then for $E_1 \in M_{(\beta,tH)}(v_1)$,
$\Phi_{X \to Y}^{{\bf E}^{\vee}[k]}(E_1)$ is purely
1-dimensional: We set $c_1(u)/\rk u=\beta'$.
If $\beta'=\beta$, then
the claim holds.
If $E_1$ is $\sigma_{\beta',t' H'}$-semi-stable but
is not $\sigma_{\beta,tH}$-semi-stable, then 
there is a wall $W_{v_2}$ for $v_1$
separating $(\beta,H,t)$ and $(\beta',H',t')$
such that $E_1$ is properly $\sigma_{\beta_1,t_1,H_1}$-semi-stable
for $(\beta_1,H_1,t_1) \in W_{v_2}$.
Since $v_2$ also defines a wall for $v$, 
we have $v_2 \in u^\perp$.
Therefore $E_1$ is generated by $\sigma_{\beta,tH}$-semi-stable
objects $F_i$ with $v(F_i) \in u^\perp$.
In particular, 
$\Phi_{X \to Y}^{{\bf E}^{\vee}[k]}(E_1)$ is a complex whose cohomology
sheaves are purely
1-dimensional sheaves.
If $\phi_{(\beta',t' H')}(E)+k<
\phi_{(\beta',t' H')}({\bf E}_{|X \times \{y \}})\leq 
\phi_{(\beta',t' H')}(E)+k+1$
for all $y \in Y$,
then $\Hom({\bf E}_{|X \times \{y \}},E)=0$
for all $y \in Y$.
Then $\Phi_{X \to Y}^{{\bf E}^{\vee}[k]}(E_1)$
is purely 1-dimensional.    
\end{NB}

\begin{NB}
We set $v:=(0,\xi,a)=
a_\gamma \varrho_X+
(d_\gamma H+D_\gamma+(d_\gamma H+D_\gamma,\gamma)\varrho_X)$.
If there is no wall between $\varrho_X$ and $\xi(\beta,H,t)$, then
$M_{(\beta,tH)}(v) \cong M_H^\beta(v)$. 
\end{NB}

%We set $\beta'=\beta+\lambda$

\begin{NB}
We set $v=a \varrho_X+(dH+D+(dH+D,\beta)\varrho_X)$.
Let $W_{v_1}$ be a wall for $v$ with $v_1 \in \varrho_X^\perp$.
If $E \in M_{(\beta,tH)}(v)$, then $E[k] \in {\frak A}_{(\beta,tH)}$
for some $k$. If $E[k]$ is not a $\beta$-twisted semi-stable
sheaf, then for some $t'>t$ and $v_1 \in {\frak W}$,
$(\beta,H,t') \in W_{v_1}$.
We set $v_1=r_1 e^\beta+a_1 \varrho_X+
(d_1 H+D_1+(d_1 H+D_1,\beta)\varrho_X)$.
Then $2(d_1 a-d a_1)=((t'H)^2)(-d r_1)$.
Hence $r_1 \ne 0$, which implies that $v_1 \not \in \varrho_X^\perp$.
If there is no wall with $v_1 \not \in \varrho_X^\perp$,
then $E$ is $\beta$-semi-stable.

\end{NB}

\begin{prop}\label{prop:Lag}
Let $u$ 
be a primitive and isotropic Mukai vector
with $u \in \overline{P^+(v^\perp)}_{\Bbb R}$. 
We take $(\beta,H,t) \in {\frak H} 
\setminus \cup_{v_1 \in {\frak W}} W_{v_1}$ such that
$\xi(\beta,H,t)$ and $u$ are not separated by a wall.
Then $\theta_{v,\beta,tH}(u)$ gives a Lagrangian fibration
$K_{(\beta,tH)}(v) \to {\Bbb P}^{\langle v^2 \rangle/2-1}$.
\end{prop}

\begin{proof}
%For the Mukai vector $v$ in Proposition \ref{prop:movable(rational)},
%we shall explain the morphisms induced by the boundaries of positive cone. 
%We take a primitive vector $u$
%from ${\Bbb R}_{>0}\xi(\beta,H,t)$.
\begin{NB}
If $\rk u=0$, then $u=pf+q \varrho_X$,
where $f$ is a fiber of an elliptic fibration
$X \to C$ over an elliptic curve $C$.
Since $\rk u=0$,
$(\beta-\delta,f)=0$.
We take $(\beta,H,t)$ which is sufficiently close to
$(\beta,f,t)$.
If $(\beta-\delta,H)<0$, then
$E \in M_{(\beta,tH)}(v)$ is a stable sheaf.
\end{NB}
We take a Fourier-Mukai transform $\Phi:{\bf D}(X)
\to {\bf D}(X)$ such that $\rk \Phi(u) \ne 0$.
Since $\Phi$ induces an isomorphism
$M_{(\beta,tH)}(v) \to 
M_{(\beta_1,t_1 H_1)}(\Phi(v))$
with ${\Bbb R}_{>0}\Phi(\xi(\beta_1,H_1,t_1))=
{\Bbb R}_{>0}\xi(\beta,H,t)$,
we may assume that $\rk u \ne 0$.
Then we have $u=r e^{\beta'}$ with $r \ne 0$.
\begin{NB}
We take
$(\beta',H',t') \in {\frak H} 
\setminus \cup_{v_1 \in {\frak W}} W_{v_1}$ such that
$\xi(\beta',H',t')$ and $u$ are not separated by a wall.
There are at most finitely many walls 
$W_{v_1}$ such that
$v_1 \in {\frak I}_1 \cup {\frak I_2}$,
$u \in v_1^\perp$,
$W_{v_1}$ separate $(\beta,H,t)$ and
$(\beta',H',t')$.
By Proposition \ref{prop:codim0} and Remark \ref{rem:isom},
there is an isomorphism
$M_{(\beta'',t'' H'')}(v) \to M_{(\beta,tH)}(v)$
such that $(\beta'',H'',t'') \in {\cal D}(\beta',t' H')$
and $\theta_{v,\beta'',t'' H''}(u)$ corresponds to
 $\theta_{v,\beta,t H}(u)$.
Hence we may assume that $(\beta,H,t)=(\beta',H',t')$.
In this case, we have
$u \in {\Bbb R}_{>0}\xi(\beta,H,0)$ for any $H$.
\end{NB}
\begin{NB}
$v=r e^\beta+\xi+(\xi,\beta)\varrho_X$,
where $r \beta+\xi=r \delta$.
Hence $((\delta-\beta)^2)=\langle v^2 \rangle/r^2$.
Then $\xi(\beta,H,0)=-r(\delta-\beta,H)e^\beta$.
\end{NB}
%we may assume that $(\beta,H,t)$ $(t \ll 1)$ belongs to a chamber $C$.
We take $w =\pm u$ with $\rk w >0$ .
We set $Y :=M_H(w)$ and let ${\bf E}$ be the universal family
as twisted objects.
Then we have (twisted) Fourier-Mukai transforms 
$\Phi_{X \to Y}^{{\bf E}^{\vee}[k] }:
{\bf D}(X) \to {\bf D}^{\alpha}(Y)$,
where $\alpha$ are suitable 2-cocycles of ${\cal O}_{Y}^{\times}$
defining ${\bf E}^{\vee}$
and $k=1$ for $(\delta-\beta',H)>0$, $k=2$ for $(\delta-\beta',H)<0$.
We set $v'=\Phi_{X \to Y}^{{\bf E}^{\vee}[k] }(v)$.
Since $(\widetilde{\beta},\widetilde{tH},s)$ 
does not lie on any wall for $s \geq 1$,
we have $M_{(\widetilde{\beta},s \widetilde{tH})}(v') \cong
M_{\widetilde{tH}}^\alpha(v')$ for $s \geq 1$,
where $M_{\widetilde{tH}}^\alpha(v')$ is the moduli space
of semi-stable $\alpha$-twisted sheaves on $Y$
(\cite{Y:twisted}).
Thus we get an isomorphism
$M_{(\beta,tH)}(v) \to M_{\widetilde{tH}}^\alpha(v')$ with $\rk v'=0$.
We note that
the scheme-theoretic support $\Div(E)$ of a purely 1-dimensional
sheaf is well-defined even for
a twisted sheaf.
Hence we have a morphism
 $f:M_{\widetilde{tH}}^\alpha(v') \to \Hilb_{Y}^\eta$
by sending $E \in M_{\widetilde{tH}}^\alpha(v')$
to $\Div(E)$, where $\Hilb_{Y}^\eta$
is the Hilbert scheme of effective divisors $D$
on $Y$ with $\eta=c_1(D)$.
For a smooth divisor $D$, $f^{-1}(D) \cong \Pic^0(D)$.
Hence $f$ is dominant, which implies $f$ is surjective. 
Therefore
we get a surjective morphism
$M_{(\beta,tH)}(v) \to \Hilb_{Y}^\eta$.
Combining with the properties of the albanese map,
we have a commutative diagram:
\begin{equation*}
\begin{CD}
M_{(\beta,tH)}(v) @>>> \Hilb_{Y}^\eta \\
@V{\frak a}VV @VVV\\
X \times \widehat{X} @>>> \Pic^0(Y)
\end{CD}. 
\end{equation*}
Hence we get a morphism
$K_{(\beta,tH)}(v) \to {\Bbb P}(H^0(Y,{\cal O}(D)))$, where
$D \in \Hilb_{Y}^\eta$.
Then we see that $\theta_{v,\beta,tH}(u)$ comes from
 ${\Bbb P}(H^0(Y,{\cal O}(D)))$.
Thus $\theta_{v,\beta,tH}(u)$ gives a Lagrangian fibration.  
\end{proof}

As we shall see in appendix, the fiber of  
$K_{(\beta,tH)}(v) \to {\Bbb P}(H^0(Y,{\cal O}(D)))$ is connected.

\subsection{The birational classes of the moduli spaces
of rank 1 sheaves.}\label{subsect:birational}

\begin{prop}\label{prop:Hilb}
Let $(X,H)$ be a polarized abelian surface.
Let $v=(r,\xi,a)$ be a Mukai vector such that 
$2\ell:=\langle v^2 \rangle \geq 6$.
Then $M_H^\beta(v)$ is birationally equivalent
to $\Pic^0(Y) \times \Hilb_Y^{\ell}$
if and only if there is an isotropic Mukai
vector $w \in H^*(X,{\Bbb Z})_{\alg}$ with $\langle v,w \rangle=1$,
where $Y$ is an abelian surface.   
\end{prop}

\begin{proof}
By using a Fourier-Mukai transform,
we may assume that $r>0$.
If there is an isotropic Mukai
vector $w$ with $\langle v,w \rangle=1$,
then $M_H^\beta(w)$ is a fine moduli space and 
the claim follows by \cite[Cor. 0.3]{Y}.

Conversely
if $M_H^\beta(v)$ is birationally equivalent to
$\Pic^0(Y) \times \Hilb_Y^{\ell}$, then
we have a birational map $f:K_H^\beta(v) \to K_{H'}(1,0,-\ell)$,
where $H'$ is an ample divisor on $Y$.
Then we have an isomorphism 
$f_*:\NS(K_H(v)) \to \NS(K_{H'}(1,0,-\ell))$.
By the isomorphism $f_*$,
the movable cones are isomorphic.
By Thoerem \ref{thm:general:movable}, ${\frak I}_1 \ne \emptyset$.
\begin{NB}
We take $(\beta,tH)$ with $M_H^\beta(v)=M_{(\beta,tH)}(v)$.
Then we have $(\beta',t' H') \in {\cal D}(\beta,tH)$
such that $K_{(\beta',t' H')}(v)= K_{H'}(1,0,-\ell)$.
Let $\theta_{v,\beta,t H}(x) \in \NS(K_{(\beta',t' H')}(v))$
be the divisor corresponding the Hilbert-Chow morphism.
Then $x \in \overline{{\cal D}(\beta,tH)}$ and 
satisfies $x \in v_1^\perp$ for $v_1 \in {\frak I}_1 \cup {\frak I}_2$.
Since the exceptional locus is not primitive,
we have $v_1 \in  {\frak I}_1$.
\end{NB}
\end{proof}

\begin{rem}
Proposition \ref{prop:Hilb} also follows from
\cite[Lem. 4.9]{Ma-Me}.
Indeed they characterize the generalized Kummer variety
in terms of the class of the stably prime exceptional divisor
$\delta$ such that $2\delta$ should corresponds to the diagonal divisor.
It implies the existence of an isotropic vector $w$ in 
Proposition \ref{prop:Hilb}.
\end{rem}

\begin{rem}
If
\begin{equation*}
\min \{ \langle v,w \rangle>0 \mid \langle w^2 \rangle=0 \} \geq 3,
\end{equation*}
then
$M_H^\beta(v)$ is not birationally equivalent to
any moduli space $M_{H'}^{\beta'}(v')$
 on an abelian surface
$Y$ with $\rk v'=2$.
\end{rem}

The following result was conjecture by Mukai \cite{Mukai:1980}.

\begin{cor}\label{cor:Hilb}
Let $(X,H)$ be a principally polarized abelian surface
with $\NS(X)={\Bbb Z}H$.
Let $v=(r,dH,a)$ be a Mukai vector with
$\ell:=d^2-ra \geq 3$.
Then $M_H^\beta(v)$ is birationally equivalent
to $X \times \Hilb_X^{\ell}$
if and only if
the quadratic equation 
$$
rx^2+2dxy+a y^2=\pm 1
$$
has an integer valued solution.  
\end{cor}

\begin{proof}
Primitive isotropic Mukai vectors are described as
$w=\pm (p^2,-pq H,q^2)$, $p,q \in {\Bbb Z}$ and 
$$
\langle v,w \rangle=\mp (r q^2 +2rpq+a p^2).
$$
Hence the claim follows from 
Proposition \ref{prop:Hilb}. 
\end{proof}

\begin{rem}
We assume that $(X,H)$ is a principally polarized abelian surface
with $\NS(X)={\Bbb Z}H$.
If $\ell=1,2$, then Mukai proved 
$M_H^\beta(v) \cong X \times \Hilb_X^\ell$
(see \cite[section 7]{YY2}).  
\end{rem}

%%%%%%%%%%%%%%%%%%%%%%%%%%%%%%%%%%%%%%%%%%%%%%%%%%
%%%%%%%%%%%%%%%%%%%%%%%%%%%%%%%%%%%%%%%%%%%%%%%%%%
%%%%%%%%%%%%%%%%%%%%%%%%%%%%%%%%%%%%%%%%%%%%%%%%%%

\subsection{Walls for $X$ with $\rk \NS(X) \geq 2$.}

We shall show that there are many walls
by using Fourier-Mukai transforms.
We set
\begin{equation*}
Q_\ell:=
\{\xi \in \NS(X)_{\Bbb R} \mid (\xi^2)=2\ell\}.
\end{equation*}

\begin{lem}\label{lem:boundary}
We set $v=(r,\xi_0,a)$ ($r \ne 0$) and $\ell:=\langle v^2 \rangle/2$.
An isotropic vector $w \in H^*(X,{\Bbb Z})_{\alg} \otimes {\Bbb R}$
satisfies $\langle w, v \rangle=0$ if and only if
$w=(\rk w) e^{\xi_0/r+\xi}$ with $r \xi \in Q_\ell$ or 
$w=(0,\xi,(\xi,\xi_0)/r)$ with
$(\xi^2)=0$.
\end{lem}

\begin{proof}
Assume that $\rk w \ne 0$ and set
$w=(\rk w)e^{\xi_0/r+\xi}$.
Since $v=r e^{\xi_0/r}-\frac{\ell}{r} \varrho_X$,
the condition is
$$
0=\langle e^{\xi_0/r+\xi},v \rangle=-r\frac{(\xi^2)}{2}+\frac{\ell}{r}.
$$
Thus $r\xi \in Q_\ell$.

If $\rk w=0$, then we set $w=e^{\xi_0/r}(0,\xi,a)$ with $(\xi^2)=0$.
Then the condition is $a=0$, which implies 
$w=(0,\xi,(\xi,\xi_0)/r)$. 
\end{proof}

\begin{NB}

\begin{lem}\label{lem:m-irrational-old}
We can find $\eta \in \NS(X)$ such that
$(\eta^2)>0$ and
$\sqrt{2\ell/(\eta^2)} \not \in {\Bbb Q}$.
\end{lem}

\begin{proof}
It is sufficient to find
$\eta \in \NS(X)_{\Bbb Q}$ such that
$(\eta^2)>0$ and
$\sqrt{2\ell/(\eta^2)} \not \in {\Bbb Q}$.
We take $\eta_1 \in \NS(X)_{\Bbb Q}$ with $(\eta_1^2)>0$.
Assume that $\sqrt{2\ell/(\eta_1^2)} \in {\Bbb Q}$.
Replacing $\eta_1$ by $\sqrt{2\ell/(\eta_1^2)}\eta_1$,
we may assume that $(\eta_1^2)=2\ell$.
We take $0 \ne \eta_2 \in \NS(X)_{\Bbb Q}$ such that
$(\eta_1,\eta_2)=0$.
We may assume that $m:=-(\eta_2^2)/(2\ell)$ is a positive integer.  
Then $((x \eta_1+y \eta_2)^2)=2\ell(x^2-m y^2)$.
If $\sqrt{m} \in {\Bbb Z}$, then
$2\sqrt{m} \eta_1+\eta_2$ satisfies the claim.
If $\sqrt{m} \not \in {\Bbb Z}$, then
$m \eta_1+\eta_2$ satisfies the claim.  
\end{proof}
\end{NB}

\begin{lem}\label{lem:m-irrational}
Assume that $\rk \NS(X) \geq 2$.
For $\xi_0 \in \NS(X)$ and 
an ample divisor $\xi$,
there is $\xi_1 \in \NS(X)$ such that
$\xi_1=\xi_0+r(k\xi+\eta)$ ($k \gg -(\eta^2)$, $\eta \in \NS(X)$) and
$\sqrt{\frac{2\ell}{(\xi_1^2)}} \not \in {\Bbb Q}$.
\end{lem}

\begin{proof}
We take an integer $k_1 \gg 0$ such that
$((\xi_0+rk_1 \xi)^2)>0$.
If $\sqrt{2\ell ((\xi_0+rk_1 \xi)^2)} \not \in {\Bbb Z}$,
then
$\xi_0+rk_1 \xi$ satisfies the claim.
Assume that $a:=\sqrt{2\ell ((\xi_0+rk_1 \xi)^2)} \in {\Bbb Z}$.
We take $\eta \in  \NS(X)$
with $\langle \eta, (\xi_0+rk_1 \xi) \rangle=0$.
%
%If $a:=\sqrt{2\ell ((\xi_0+rk_1 \xi)^2)} \in {\Bbb Z}$, then
Then 
$\xi_1:=(rk_2+1)(\xi_0+rk_1 \xi)-r\eta$ satisfies
$2\ell (\xi_1^2)=(rk_2+1)^2 a^2+2\ell r^2(\eta^2)$.
If $2\ell (\xi_1^2)=x^2$, $x \in {\Bbb Z}$, then
$$
-2\ell r^2(\eta^2)=((rk_2+1)a-x)((rk_2+1)a+x)>(rk_2+1)a.
$$
Hence for $k_2 \gg 0$,
$
\sqrt{\frac{2\ell}{(\xi_1^2)}}= 
\frac{\sqrt{2\ell(\xi_1^2)}}{(\xi_1^2)} \not \in {\Bbb Q}.
$
\end{proof}

\begin{prop}
Assume that $\rk \NS(X) \geq 2$.
For $v=(r,\xi_0,a)$ with $\langle v^2 \rangle=2\ell$,
if $\ell \geq r>0$, then
$\overline{\cup_{u \in {\frak W}} u^\perp}$
contains $\overline{P^+(v^\perp)}_{\Bbb R} \setminus
P^+(v^\perp)_{\Bbb R}$. 
\end{prop}

\begin{proof}
Since $\ell \geq r$,
$u:=(0,0,-1)$ satisfies 
\eqref{eq:wall-cond} and \eqref{eq:wall-cond2}.
Hence $u$ defines a non-empty wall $u^\perp$.

We shall use Lemma \ref{lem:boundary} to show the claim.
For $r\xi \in Q_\ell$,
we take $r \xi_1 \in \NS(X)_{\Bbb Q}$ in a neighborhood of
$r \xi$.
Then $\sqrt{\tfrac{2\ell}{(\xi_1^2)}}\xi_1 \in Q_\ell$
is also sufficiently close to $r\xi$.
Replacing $\xi$ by $-\xi$ if necessary, we assume that
$r\xi_1$ is ample.
We take a primitive and ample
divisor $H$ such that
$dH= \xi_0+r(k\xi_1+\eta)$ $(k \gg 0)$
and $\sqrt{\frac{2\ell}{(H^2)}} \not \in {\Bbb Q}$.
Then $\lim_{k \to \infty} \sqrt{\frac{2\ell}{(H^2)}}H=
\sqrt{\frac{2\ell}{(\xi_1^2)}}\xi_1$.  
In particular, for any open neighborhood $U$ of $r\xi$,
we can take an ample divisor $H$ 
such that $\sqrt{\frac{2\ell}{(H^2)}}H \in U$.
We set $L:={\Bbb Z} \oplus {\Bbb Z}H \oplus {\Bbb Z}\varrho_X$.
For $v':=v e^{k \xi_1+\eta}=(r,dH,a') \in L$, 
Lemma \ref{lem:S_H} implies 
we have an autoequivalence $\Phi$ of
${\bf D}(X)$ such that $\Phi(v')=v'$ and $\Phi_{|L}$ is of infinite order.
We set $\zeta_\pm:=
e^{\frac{1}{r}\left(d \pm \sqrt{\frac{2\ell}{(H^2)}} \right)H}$.
Then ${\Bbb R}_{>0} \zeta_\pm$ are the fixed rays of $\Phi$
in $L_{\Bbb R} \cap \overline{P^+(v^\perp)}_{\Bbb R}$
and the rays defined by $\Phi^n(u)^\perp$
converge to the fixed rays.
Hence
$\zeta_\pm \in 
\overline{\cup_{n \in {\Bbb Z}}\Phi^n(u)^\perp}$.
We set $\Psi:=e^{-(k\xi_1+\eta)} \Phi e^{k\xi_1+\eta}
\in \Eq({\bf D}(X),v)$,
where $\Eq({\bf D}(X),v)$ is the set of autoequivalences of ${\bf D}(X)$
fixing $v$.
Since 
$$
\left(d \pm \sqrt{\frac{2\ell}{(H^2)}}\right)H=\xi_0
\pm \sqrt{\frac{2\ell}{(H^2)}}H
+r(k\xi_1+\eta),
$$
$e^{\frac{1}{r}\left(\xi_0 \pm \sqrt{\frac{2\ell}{(H^2)}}H \right)}
\in \overline{\cup_{n \in {\Bbb Z}}\Psi^n(u)^\perp}$.
Hence $e^{\xi_0/r \pm \xi} \in 
\overline{\cup_{\Phi \in \Eq({\bf D}(X),v)}\Phi(u)^\perp}$.

For $\xi \in \overline{\Amp(X)}$ with $(\xi^2)=0$,
we have
$(0,\xi,(\xi,\xi_0)/r) \in (0,0,-1)^\perp$. 
Therefore the claim holds.
\end{proof}

\begin{lem}\label{lem:nowall-cond}
Assume that $v_1$ defines a wall for $v$.
Then
$(\langle v^2 \rangle+\langle v_1^2 \rangle)/2
 \geq \langle v,v_1 \rangle>\langle v_1^2 \rangle$.
In particular,
$\langle v^2 \rangle>\langle v,v_1 \rangle>\langle v_1^2 \rangle$.
\end{lem}

\begin{proof}
Since $\langle v-v_1,v_1 \rangle >0$,
we have 
$\langle v,v_1 \rangle >\langle v_1^2 \rangle$.
Since $0 \leq \langle (v-v_1)^2 \rangle =
\langle v^2 \rangle+\langle v_1^2 \rangle-2\langle v,v_1 \rangle$,
we get the first claim.
Then we have $\langle v^2 \rangle>\langle v_1^2 \rangle$,
which implies the second claim.
\end{proof}

\begin{lem}\label{lem:no-wall}
Let $X$ be an abelian surface with $\NS(X)={\Bbb Z}H \perp L$,
where $H$ is an ample divisor and 
$L$ is a negative definite lattice. 
Assume that $(H^2)=2\ell(4 \ell ra+1)$, $r,a \in {\Bbb Z}_{>0}$.
We set $v:=(2 \ell r,H,2 \ell a)$.
Then
$\langle v^2 \rangle=(H^2)-8 \ell^2 ra =2\ell$
and there is no wall for $v$.
In particular, $\Amp(K_H(v))$ coincides with its positive cone.
\end{lem}

\begin{proof}
We note that
$\langle v,w \rangle \in 2\ell {\Bbb Z}$
for all $w \in H^*(X,{\Bbb Z})_{\alg}$.
By Lemma \ref{lem:nowall-cond}, there is no wall for $v$.
\end{proof}

\begin{rem}
If $\sqrt{\langle v^2 \rangle(H^2)}=2\ell \sqrt{4\ell ra+1}
\not \in {\Bbb Q}$
or $\rk \NS(X) \geq 2$, then 
there are infinitely many autoequivalences of ${\bf D}(X)$
preserving $v$. 
\begin{NB} 
If $ra=\ell$, then $\sqrt{4\ell ra+1}=\sqrt{4 \ell^2+1} \not \in {\Bbb Q}$.
Indeed if $4 \ell^2+1=y^2$, $y \in {\Bbb Z}_{>0}$, then
$1=(y-2\ell)(y+2\ell) \geq y+2\ell>2\ell$, which is a 
contradiction.
\end{NB}
\end{rem}

\begin{prop}\label{prop:infinite-auto}
Assume that $\rk \NS(X) \geq 2$ or $\sqrt{\langle v^2 \rangle/(H^2)}
\not \in {\Bbb Q}$.
If there is no wall for $v$,
then $\Aut(M_H(v))$ contains an automorphism of infinite order.
%In this case, the ample cone of $K_H(v)$ coincides with the positive cone.
\end{prop}

\begin{proof}
We may assume that $v=(r,\xi,a)$, where $\xi$ is ample.
We take $g \in \Stab_0(v)^*$ of infinite order.
Then there is an autoequivalence $\Phi$ of ${\bf D}(X)$ 
which induces $g$.
Then $\Phi$ induces an isomorphism
$M_H(v) \to M_H(v)$.
\begin{NB}
Since $\theta_v:v^{\perp} \to H^2(M_H(v),{\Bbb Z})$ is
$g$-equivariant, $\Phi^m$ $(m \ne 0)$ acts non-trivially.  
Let ${\cal L}$ be an ample divisor of $K_H(v)$. 
By the action of $g$, $\lim_{m \to \pm \infty}g^m ({\cal L})$
are the rays $R_\pm$ with $q_{K_H(v)}(R_\pm)=0$.  
\end{NB}
\end{proof}

If $M_{(\beta,tH)}(v)$ is birationally equivalent to
$M_H(1,0,-\ell)$, we can get a more precise description of 
the stabilizer group. Since there is an autoequivalence
$\Phi:{\bf D}(X) \to {\bf D}(X)$ with
$\Phi(v)=(1,0,-\ell)$, it is sufficient to treat the case
$v=(1,0,-\ell)$.

\begin{prop}\label{prop:decomp}
We set $v=(1,0,-\ell)$.
\begin{enumerate}
\item[(1)]
A Fourier-Mukai transform $\Phi:{\bf D}(Y) \to {\bf D}(X)$
with the condition $\Phi((1,0,-\ell))=(1,0,-\ell)$
corresponds to a decomposition 
\begin{equation}\label{eq:decomp}
v=\ell u_1+u_2,\;
\langle u_1,u_2 \rangle=1,\;
\langle u_1^2 \rangle=\langle u_2^2 \rangle=0,
\end{equation}
where $u_1=\Phi(-\varrho_Y)$ and $u_2=\Phi(v({\cal O}_Y))$.
\item[(2)]
For the decomposition \eqref{eq:decomp}, one of the following holds.
\begin{enumerate}
\item[(i)]
$$
u_1=(p^2 s,pq \xi,q^2 t),\;
u_2=-(q^2 t,\ell pq \xi,\ell^2 p^2 s),
$$
where $p,q,s,t \in {\Bbb Z}$ satisfy
$\ell s p^2-t q^2=1$
and $\xi$ is a primitive divisor with
$(\xi^2)=2st>0$.
\item[(ii)] 
$$
\ell=1,\;u_1=(0,\xi,-1),\; u_2=(1,-\xi,0),
$$
where $(\xi^2)=0$.
\end{enumerate}
\item[(3)]
For the case (i) of (2),
if $t=1$, then $Y \cong X$.
\end{enumerate}
\end{prop}

\begin{proof}
(1)
Since $v=\ell (-\varrho_Y)+v({\cal O}_Y)$,
we have $v=\Phi(v)=\ell \Phi(-\varrho_Y)+v({\cal O}_Y)$.
Since $\varrho_Y$ and $v({\cal O}_Y)$ are isotropic vector with
$\langle -\varrho_Y,v({\cal O}_Y) \rangle=1$,
we get the decomposition \eqref{eq:decomp}.
Conversely for the decomposition \eqref{eq:decomp}, we have 
an equivalence $\Phi:{\bf D}(Y) \to {\bf D}(X)$
such that $u_1=\Phi(-\varrho_Y)$ and $u_2=\Phi(v({\cal O}_Y))$,
where $Y=M_H(\pm u_1)$ and $H$ is an ample divisor on $X$.

(2)
We first assume that $\rk u_1 \ne 0$.
We set $u_1=(px,pq \xi,y)$, where $\xi \in \NS(X)$ is primitive and
$\gcd(x,q \xi)=1$.
Since $u_1$ is primitive, $\gcd(p,y)=1$.
Assume that $q^2(\xi^2) \ne 0$.
Then $p^2 q^2 (\xi^2)=2pxy$ implies that
$p \mid x$ and $q^2 \mid y$.
So we write $u_1=(p^2 s,pq \xi,q^2 t)$ with
$(\xi^2)=2st$.
Since $1=\langle v, u_1 \rangle$,
we have $p^2 s \ell-q^2 t=1$.
We set $u_2:=v-\ell u_1$. Then
$$
u_2=(1-\ell s p^2,-\ell pq \xi,-\ell(1+q^2 t))
=-(q^2 t,\ell p q \xi,\ell^2 p^2 s).
$$
If $st<0$, then we have $(p^2 s \ell,q^2 t)=(1,0)$ or
$(p^2 s \ell,q^2 t)=(0,-1)$.
Since $\rk u_1 \ne 0$, we have $p^2=s=\ell=1$ and $q=0$.
Since $q^2(\xi^2) \ne 0$, this case does not occur.
Therefore $st>0$.

If $q^2 (\xi^2)=0$, then
$y=0$ and $p=\pm 1$. Hence 
$u_1=\pm (x,q \xi,0)$.
Since $1=\langle v,u_1 \rangle= \pm \ell x$,
$x=\pm 1$ and $\ell=1$.
Thus $u_1=(1,q\xi,0)$ and $u_2=(0,-q\xi,-1)$.
By exchanging $u_1$ by $u_2$,
we have (ii).
\begin{NB}If $q=0$, then $u_1=(1,0,0)$ and fits in the case (i),
where we exchange $u_1$ and $u_2$.
If $(\xi^2)=0$, then $u_2=(0,-q\xi,-1)$ and fits in the case (ii). 
\end{NB}

We next assume that $\rk u_1=0$.
We set $u_1=(0,D,y)$.
Then $1=\langle v,u_1 \rangle=-y$.
Hence $u_1=(0,D,-1)$ with $(D^2)=0$.
Then (i) holds, where $t=-1$, $q^2=1$ and $pq$ is the multiplicity of
$D$. 
(3) If $(\xi^2)=2s$, then Lemma \ref{lem:n} implies the claim.
\end{proof}

%\begin{rem}
%For case (i), 
%$s,t>0$ if and only if $\xi \in \Amp(X)$.
%\end{rem}

\begin{rem}
In \cite{YY1},
the condition $v=\ell u_1+u_2$ with
$\langle u_1,u_2 \rangle=1$,
$\langle u_1^2 \rangle=\langle u_2^2 \rangle=0$
is called {\it numerical equation} and plays important role
for the study of stable sheaves.
\end{rem}

For $m \in {\Bbb Q}$,
we set
\begin{equation*}
Q_{\ell,m}:=
\{\xi \in \NS(X)_{\Bbb R} \mid \xi \in \sqrt{m}\NS(X)_{\Bbb Q},
(\xi^2)=2\ell\}.
\end{equation*}
$Q_{\ell,m}=Q_{\ell,m'}$ if and only if
$\sqrt{m} \in {\Bbb Q}\sqrt{m'}$.
For $\xi \in Q_{\ell,m}$,
we take an ample divisor $H$ with
${\Bbb Q}\tfrac{\xi}{\sqrt{m}} \cap \NS(X)={\Bbb Z}H$.
If $p^2 \ell (H^2)/2 -q^2=1$ has an integral solution $(p,q)$,
then there is an autoequivalence $\Phi$
in Proposition \ref{prop:decomp}.
In particular, if $\sqrt{m} \not \in {\Bbb Q}$, then
there are infinitely many $\Phi$ in Proposition \ref{prop:decomp}.

\begin{lem}\label{lem:Q_ell}
\begin{enumerate}
\item[(1)]
Each $Q_{\ell,m}$ is a dense or empty subset
of $Q_\ell$.
\item[(2)]
$\cup_{\sqrt{m} \not \in {\Bbb Q}} Q_{\ell,m}$
is dense in $Q_\ell$. In particular,
there are many autoequivalences $\Phi$
in Proposition \ref{prop:decomp}.
\end{enumerate}
\end{lem}

\begin{proof}
(1)
We set $Q_\ell^+:=Q_\ell \cap \Amp(X)_{\Bbb R}$ and
$Q_{\ell,m}^+=Q_{\ell,m} \cap Q_\ell^+$.
Then $Q_\ell=Q_\ell^+ \cup -Q_\ell^+$
and $Q_{\ell,m}=Q_{\ell,m}^+ \cup -Q_{\ell,m}^+$.
Assume that $Q_{\ell,m} \ne \emptyset$.
We take $\xi_0 \in Q_{\ell,m}^+$ and
set
$$
B_{\xi_0}:=\{ \eta \in \xi_0^\perp \mid -(\eta^2) < 2\ell \}.
$$
Then we have a bijective correspondence
\begin{equation*}
\begin{matrix}
Q_\ell^+ & \to & B_{\xi_0}\\
\xi & \mapsto & \eta,
\end{matrix}
\end{equation*}
where 
\begin{equation*}
\begin{split}
\eta=& \frac{2\ell \xi-(\xi,\xi_0)\xi_0}{2\ell+(\xi_0,\xi)},\\
\xi=&\frac{4\ell}{(\eta^2)+2\ell}\eta+
\frac{2\ell-(\eta^2)}{2\ell+(\eta^2)}\xi_0.
\end{split}
\end{equation*}
\begin{NB}
$\eta$ is the intersection of 
the line $\overline{\xi (-\xi_0)}$ with $\xi_0^\perp$.
Then $\eta=t(\xi+\xi_0)-\xi_0$ and $(\eta,\xi_0)=0$,
$t=\frac{2\ell}{(\xi+\xi_0,\xi_0)}$.
Thus $\eta=\frac{2\ell}{(\xi+\xi_0,\xi_0)}\xi-
\frac{(\xi,\xi_0)}{(\xi+\xi_0,\xi_0)}\xi_0$.
\end{NB}
Then $Q_{\ell,m}^+$ corresponds to
$\eta \in \sqrt{m} \NS(X)_{\Bbb Q}$.
Therefore $Q_{\ell,m}$ is dense in $Q_\ell$.

(2) is a consequence of Lemma \ref{lem:m-irrational} and (1).
\end{proof}

\begin{NB}
\begin{proof}

(1)
Assume that $Q_{\ell,m} \ne \emptyset$.
We take $\xi_0 \in Q_{\ell,m}$ and set
$\eta_0:=\xi_0/\sqrt{m}$.
For $\xi \in Q_\ell$,
we set
$tu:=(\xi-\xi_0)/\sqrt{m} \in \NS(X)_{\Bbb Q}$.
Then we have
$t(t(u^2)+2(\eta_0,u))=0$.
If $(u^2)=0$, then $t(\eta_0,u)=0$.
Since $(\eta_0^2)>0$,
we see that $tu=0$.
Thus $\xi=\xi_0$.
So we may assume that $(u^2) \ne 0$.  
Then we see that $t=-2(\eta_0,u)/(u^2)$ and
$$
\xi=\sqrt{m}
\left(\eta_0-2\frac{(\eta_0,u)}{(u^2)}u \right),\;
u \in \NS(X)_{\Bbb Q}.
$$
Therefore $Q_{\ell,m}$ is dense in $Q_\ell$.
\end{proof}
\end{NB}

\begin{NB}
\begin{prop}
Assume that $\rk \NS(X) \geq 2$.
Then $\overline{\cup_{u \in {\frak I}_1} u^\perp}$
contains $\overline{P^+(v^\perp)}_{\Bbb R} \setminus
P^+(v^\perp)_{\Bbb R}$. 
\end{prop}

\begin{proof}
For $\xi \in Q_{\ell,m}$,
we take a primitive $H \in {\Bbb Q}_{>0}\frac{\xi}{\sqrt{m}}$.
Then $\sqrt{2\ell/(H^2)} \not \in {\Bbb Q}$.
Hence we have infinitely many solutions $(p,q)$ of
$p^2 \ell (H^2)/2-q^2=1$.
Then there are infinitely many walls in any neighborhood of
$e^{\sqrt{2\ell/(H^2)}H}$.
By using this fact, we shall prove the claim.

We first note that
an isotropic vector $u \in H^*(X,{\Bbb Z})_{\alg} \otimes {\Bbb R}$
belongs to $w \in v^\perp$ if and only if
$w=(\rk w) e^\xi$ with $\xi \in Q_\ell$ or $w=(0,\xi,0)$ with
$(\xi^2)=0$.
For $w=(\rk u) e^\xi$ with $\xi \in Q_\ell$, 
Lemma \ref{lem:Q_ell} implies that
$w \in \overline{\cup_{u \in {\frak I}_1} u^\perp}$. 
For $\xi \in \overline{\Amp(X)}$ with $(\xi^2)=0$,
we have
$(0,\xi,0) \in (0,\xi,-1)^\perp$. 
Therefore the claim holds.
\end{proof}

\end{NB}

\section{The case where $\rk \NS(X)=1$.}\label{sect:rk1}

\subsection{The walls and chambers on the $(s,t)$-plane.}

From now on, let $X$ be an abelian surface such that
$\NS(X)={\Bbb Z}H$, where $H$ is an ample generator.
We set $(H^2)=2n$.
Let
${\Bbb H}:=\{(s,t) \mid t>0 \}$ be the upper half plane and
$\overline{\Bbb H}:=\{(s,t) \mid t \geq 0 \}$
its closure. 
We identify $\NS(X)_{\Bbb R} \times C(\overline{\Amp(X)}_{\Bbb R})
\times {\Bbb R}_{\geq 0}$ with $\overline{\Bbb H}$ via
$(sH,H,t) \mapsto (s,t)$. 
We shall study the set of walls.
\begin{lem}[{\cite{Mac}, \cite[Cor. 5.10]{YY2}}]
If
$\sqrt{\ell/n} \in {\Bbb Q}$, there are finitely many 
chamber. 
\end{lem}
By this lemma, it is sufficient to treat
the case where $\sqrt{\ell/n} \not \in {\Bbb Q}$.

\begin{prop}
Assume that $n \leq 4$.
Let $v:=(r,dH,a)$ be a primitive Mukai vector  with
$\langle v^2 \rangle>0$.
Then there is an isometry $\Phi$ of
$H^*(X,{\Bbb Z})_{\alg}={\Bbb Z}^{\oplus 3}$ 
such that $\Phi(v)=(r',d' H,a')$
with $r' a' \leq 0$.
\end{prop}

\begin{proof}
We may assume that $r \ne 0$.
Replacing $v$ by $-v$, we may assume that
$r>0$.
By the action of $e^{mH}$, we may assume that
$|d| \leq r/2$. 
Since $d^2 n-ra>0$,
we have $n/4 \geq (d/r)^2 n>a/r$.
By our assumption, we have
$a<r$.
If $a>0$, then we apply the isometry
$\varphi:(r,dH,a) \mapsto (a,-dH,r)$.
Applying the same arguments to $\varphi(v)$,
we finally get an isometry
$\Phi$ such that $\Phi(v)=(r',d' H,a')$
with $r' a' \leq 0$. 
\end{proof}

\begin{cor}\label{cor:n-leq4}
Assume that $n \leq 4$.
Let $v:=(r,dH,a)$ be a primitive Mukai vector  with
$\langle v^2 \rangle>0$.
If $\sqrt{\ell/n} \not \in {\Bbb Q}$,
then there are infinitely many
isotropic Mukai vectors
$u$ such that $\langle u,v \rangle>0$ and
$\langle (v-u)^2 \rangle \geq 0$.
In particular, 
there are infinitely many walls for $v$. 
\end{cor}

\begin{proof}
We first show that there is a Mukai vector $u$ satisfying
the requirements.
We may assume that $ra \leq 0$.
If $ra=0$, then $\langle v^2 \rangle= d^2(H^2)$.
Hence $ra<0$. 
Then $u=(0,0,1)$ or $(0,0,-1)$ satisfies the requirements.

Since $\sqrt{\ell/n} \not \in {\Bbb Q}$, 
$\Stab_0(v)^*$ contains an element $g$ of 
infinite order.
Then $g^{n} (u)$ 
$(n \in {\Bbb Z})$ also satisfies the requirements.
\end{proof}

\begin{rem}
The condition $(H^2)\leq 8$ in Corollary \ref{cor:n-leq4}
is necessary by Lemma \ref{lem:no-wall}.
\end{rem}

\begin{NB}
\begin{prop}
Assume that $\ell=\langle v^2 \rangle/2 \geq 3$
and $\sqrt{\ell/n} \not \in {\Bbb Q}$.
If there is an isotropic Mukai vector $w$ with
$\langle v,w \rangle=1,2$,
then there are infinitely many walls of codimension 0,1.
\end{prop}

\begin{proof}
Since $\sqrt{\ell/n} \not \in {\Bbb Q}$, 
$\Stab_0(v)^*$ contains a cyclic subgroup of
infinite order.
Hence if there is an isotropic Mukai vector $w$ with
$\langle v,w \rangle=1,2$,
then there are infinitely many walls of codimension 0,1.
\end{proof}
\end{NB}

\subsection{An example}

Assume that $n:=(H^2)/2=1$.
We set $v:=(2,H,-2)$.
Then $\ell:=\langle v^2 \rangle/2=5$.
We set
\begin{equation*}
g:=
\begin{pmatrix}
0 & 1\\
1 & 1
\end{pmatrix},\;
h:=
\begin{pmatrix}
0 & -1\\
1 & 0
\end{pmatrix}.
\end{equation*}
In this case, we have
\begin{equation*}
\begin{split}
\Stab_0(v)=& \{ \pm g^n \mid n \in {\Bbb Z} \},\\
\Stab(v)=& \Stab_0(v) \rtimes \langle h \rangle,\\
h^{-1} g h=& -g^{-1}.
\end{split}
\end{equation*}

Let $C_0$ is a line defined by
\begin{equation*}
s=\frac{1}{2}.
\end{equation*}
It is the wall defined by $v_0:=(0,0,-1)$.
Let $C_{-1}$ is a circle defined by
\begin{equation*}
t^2+s(s+2)=0.
\end{equation*}
It is the wall defined by $v_{-1}:=(1,0,0)= v_0 \cdot g^{-1}$.
We set $v_n:= v_0 \cdot g^n $ and let 
$C_n$ be the wall defined by $v_n$.

\begin{lem}
$\{ C_n \mid n \in {\Bbb Z} \}$ is the set of walls
for $v$.
\end{lem}

\begin{proof}
It is sufficient to prove that
there is no wall between $C_{-1}$ and $C_0$.
If a Mukai vector $w:=(r',d'H,a')$ defines a wall for $v$ and
the wall $W_w$ lies between $C_{-1}$ and $C_0$.
Since $C_{-1}$ passes $(0,0)$,
$W_w$ intersects with the line $s=0$.
Hence we get $d',1-d'>0$, which is a contradiction.
\end{proof}

\begin{prop}
\begin{enumerate}
\item[(1)]
For $(s,t) \not \in \cup_{n \in {\Bbb Z}} C_n$, 
$$
M_{(sH,tH)}(2,H,-2) \cong M_H(2,H,-2).
$$
\item[(2)]
Let $(r,dH,a)$ be a primitive Mukai
vector such that $2 \mid r, 2 \mid a$ and $d^2-ra=5$.
Then 
$$
M_H(r,dH,a) \cong M_H(2,H,-2).
$$
\end{enumerate}
\end{prop}

\begin{proof}
(1) Let ${\cal C}$ be the chamber between $C_0$ and $C_{-1}$.
Then $\cup_{n \in {\Bbb Z}} g^n({\cal C})=
{\Bbb H} \setminus \cup_{n \in {\Bbb Z}} C_n$.
Let $\Psi:{\bf D}(X) \to {\bf D}(X)$ be a contravariant Fourier-Mukai
transform inducing $g$.
Since $\Psi$ preserves the stability,
the claim holds.

(2) By the action of $\GL(2,{\Bbb Z})$,
$w:=(r,dH,a)$ is transformed to $(1,0,-5)$ or $(2,H,-2)$
( \cite[Prop. 7.12]{YY1}). 
Since $2 \mid \langle w,u \rangle$ for all $u \in H^*(X,{\Bbb Z})_{\alg}$,
$w:=(r,dH,a)$ is transformed to $(2,H,-2)$.
Then the claim follows from (1). 
\end{proof}

\subsection{Divisors on the moduli spaces
$M_{(sH,tH)}(v)$.}

\begin{defn}
Let $v=(r,dH,a)$ be a Mukai vector with $r>0$
and set $\ell:=\langle v^2 \rangle/2=d^2 n-ra$.
We set
$$
s_\pm:=\frac{d}{r}\pm \frac{1}{r} \sqrt{\frac{\ell}{n}}.
$$
\end{defn}

\begin{lem}
We take 
$g \in \Stab_0(v)^*$ such that the order is infinite.
Then $(s_\pm,0)$ are the fixed points of the action of
$g$ on $(s,t)$-plane.
In particular, if there is a wall, then
$(s_\pm,0)$ are the accumulation points of the set of walls.
\end{lem}

We set $(\beta,\omega):=(sH,tH)$.
Then $v=(r,dH,a)$ is written as 
$$
v=r e^\beta+d_\beta(H+(\beta,H)\varrho_X)+a_\beta \varrho_X,
$$
where
\begin{equation*}
\begin{split}
d_\beta= & d-rs,\\
a_\beta= & -\langle e^\beta,v \rangle=a-(dH,\beta)+\frac{(\beta^2)}{2}r.
\end{split}
\end{equation*}
\begin{defn}
We set
\begin{equation}\label{eq:xi(s,t)}
\begin{split}
\xi(s,t):= &
\xi(sH,H,t)\\
=&
\left( r(s^2+t^2)n-a \right)
\left(H+\frac{2dn}{r}\varrho_X \right)-
2n(d-rs)
\left(1-\frac{a}{r}\varrho_X \right).
\end{split}
\end{equation}
\end{defn}

We consider the circle
\begin{equation*}
C_{v,\lambda}:
t^2+(s-\lambda)
\left(s-\frac{1}{r \lambda-d}
\left(\lambda d-\frac{a}{n} \right) \right)=0,\; \lambda \in {\Bbb R},
\end{equation*}
that is, ${\Bbb R} Z_{(sH,tH)}(v)={\Bbb R}Z_{(sH,tH)}(e^{\lambda H})$, where
$\lambda \ne d/r$.
We note that $(\lambda,0) \in C_{v,\lambda}$ and
$(d-rs)(d-r\lambda)> 0$ for $(s,t) \in C_{v,\lambda}$.
For $(s,t) \in C_{v,\lambda}$,
we see that
\begin{equation*}
\xi(s,t)=
(d-rs)
\left(
\frac{r\lambda^2 n-a}{d-r\lambda}
\left(H+\frac{2dn}{r}\varrho_X \right) -
2n
\left(1-\frac{a}{r}\varrho_X \right)\right).
\end{equation*}
Hence ${\Bbb R}_{>0}\xi(s,t)={\Bbb R}_{>0}\xi(\lambda,0)$
and is determined by $\lambda$.
If 
$$
C_{v,\lambda}
=\{(s,t) \mid {\Bbb R} Z_{(sH,tH)}(v)={\Bbb R}Z_{(sH,tH)}(v_1)\},
$$
that is, $C_{v,\lambda}$
is the wall defined by a Mukai vector $v_1:=(r_1,d_1 H,a_1)$, then
$\frac{r\lambda^2 n-a}{d-r\lambda}
=\frac{r a_1-r_1 a}{r_1 d-r d_1} \in {\Bbb Q}$.
Thus $\xi(s,t)/(d-rs) \in H^*(X,{\Bbb Q})_{\alg}$.

\begin{lem}
\begin{equation*}
\xi(s_\pm,0)=
2\left(\frac{\ell}{r} \pm
n\frac{d}{r}\sqrt{\frac{\ell}{n}} \right)
\left(H+\frac{(dH,H)}{r}\varrho_X \right)\pm
2n\sqrt{\frac{\ell}{n}}\left(1-\frac{a}{r}\varrho_X \right)
\end{equation*}
and satisfy $\langle \xi(s_\pm,0)^2 \rangle=0$.
Thus $\xi(s_\pm,0)$ define isotropic vectors
in $v^\perp \subset 
H^*(X,{\Bbb Z})_{\alg} \otimes {\Bbb R}$.
\end{lem}

\begin{proof}
\begin{equation*}
\begin{split}
\langle \xi(s_\pm,0)^2 \rangle
= & 4\left(\frac{\ell}{r} \pm
n\frac{d}{r}\sqrt{\frac{\ell}{n}} \right)^2(2n)
+4 \ell n \frac{2a}{r}
\mp 8 \left(\frac{\ell}{r} \pm
n\frac{d}{r}\sqrt{\frac{\ell}{n}} \right)
\sqrt{n \ell}\frac{d}{r}(2n)\\
=& 4 \left(\frac{\ell}{r} \pm
n\frac{d}{r}\sqrt{\frac{\ell}{n}} \right)
\left(\frac{\ell}{r} \mp
n\frac{d}{r}\sqrt{\frac{\ell}{n}} \right)(2n)
+4 \ell n \frac{2a}{r}\\
=& 
4\frac{\ell^2}{r^2}(2n)-4
n\frac{d^2 }{r^2}\ell (2n)
+4 \ell n \frac{2a}{r}\\
=& 4\frac{\ell}{r^2}(\ell-d^2 n)(2n)
+4  \frac{\ell}{r^2}(ra)(2n)
= 0.
\end{split}
\end{equation*}
Therefore we get the claim.
\end{proof}

Proposition \ref{prop:positive-cone} is written as follows.
\begin{prop}%[{\cite[Prop. 4.3.2]{MYY:2011:2}}]\label{prop:positive-cone}
\begin{enumerate}
\item[(1)]
If $(s,t)$ belongs to a chamber and
$s, t^2 \in {\Bbb Q}$, then
$\theta_{v,sH,tH}(\xi(s,t))$ is an ample ${\Bbb Q}$-divisor
of $K_{(sH,tH)}(v)$.
\item[(2)]
We have a bijective map
\begin{equation*}
\varphi: [s_-,s_+] \to 
C(\overline{P^+(K_{(sH,tH)}(v))}_{\Bbb R}) 
\end{equation*}
such that 
$$
\varphi(\lambda):={\Bbb R}_{>0} \theta_{v,sH,tH}(\xi(\lambda,0)).
$$
\item[(3)]
$\Nef(K_{(sH,tH)}(v))_{\Bbb R}
=\varphi(\overline{D(sH,H,t)} \cap [s_-,s_+])$.
\end{enumerate}
\end{prop}

\begin{proof}
(1) is obvious.
(2) We note that
$f(\lambda):=(r\lambda^2 n-a)/(d-r \lambda)$ gives a bijective map
$$
f:[s_-,s_+] \to [\tfrac{2\sqrt{n \ell}}{r}-\tfrac{2dn}{r},\infty]
\cup [-\infty,-\tfrac{2\sqrt{n \ell}}{r}-\tfrac{2dn}{r}]
$$
and
 $$
\bigcup_{\lambda \in [s_-,s_+]}C_{v,\lambda}=
{\Bbb R}^2 \setminus \{(s_-,0),(s_+,0)\},
$$
where we identify $\infty$ with $-\infty$,
$\lambda=d/r$ corresponds to $\pm \infty$
and
$C_{v,d/r}$ is the line $s=d/r$.
Since $C_{v,\lambda} \cap [s_-,s_+]=\{ \lambda \}$,
$[s_-,s_+]$ is the parameter space of $C_{v,\lambda}$.
Since $\xi(s,t)$ is determined by $f(\lambda)$,
Proposition \ref{prop:positive-cone} implies
$\varphi$ is bijective.
\begin{NB}
We set $x:=d-r\lambda$.
Then $\frac{r \lambda^2 n-a}{d-r \lambda}=
\frac{nx}{r}+\frac{\ell}{r x}-\frac{2dn}{r}$
gives a bijection
$[s_-,s_+] \to [\frac{2\sqrt{n \ell}}{r}-\frac{2dn}{r},\infty]
\cup [-\infty,-\frac{2\sqrt{n \ell}}{r}-\frac{2dn}{r}]$,
where $\lambda=d/r$ corresponds to $\pm \infty$.
\end{NB}
  
(3) is a consequence of 
Proposition \ref{prop:positive-cone}
(1) and (2).
\end{proof}

\begin{NB}
If $s=\frac{d}{r}$, then
$(d-rs)\theta_v(\xi_\omega)=
\frac{\langle v^2 \rangle+r^2 t^2 (H^2)}{2r}
(H+\frac{d}{r}(H^2)\varrho_X)$.
\end{NB}

\subsection{The movable cone of $K_{(sH,tH)}(v)$.}

\begin{NB}
\begin{rem}
For $w \in {\frak I}_0$,
we have 
$W_w=\{(\lambda,0)  \}$, where $\lambda H=c_1(w)/\rk w$.
\end{rem}
\end{NB}

\begin{lem}\label{lem:I}
Let $v$ be a Mukai vector with $\langle v^2 \rangle=2 \ell$.
\begin{enumerate}
\item[(1)]
${\frak I}_0 \ne \emptyset$ if and only if $\sqrt{\ell/n} \in {\Bbb Q}$.
\item[(2)]
Assume that $\sqrt{\ell/n} \not \in {\Bbb Q}$. Then
${\frak I}_k \ne \emptyset$ if and only if 
$\# {\frak I}_k =\infty$.
In this case, $(s_+,0)$ and $(s_-,0)$ are the accumulation points
of $\cup_{w \in {\frak I}_k} W_w$. 
\end{enumerate}
\end{lem}

\begin{proof}
Since $x \in H^*(X,{\Bbb Z})_{\alg} \otimes {\Bbb R}$
satisfies $\langle x^2 \rangle=\langle x,v \rangle=0$
if and only if $x \in {\Bbb R}\xi(s_\pm,0)$.
Hence (1) holds.
By Proposition \ref{prop:stab}, $\Stab_0(v)$ contains an element $g$
of infinite order. Hence (2) is obvious. 
\end{proof}

\begin{prop}\label{prop:codim1-div}
Let $v$ be a primitive Mukai vector with
$\langle v^2 \rangle \geq 6$.
Let $W \subset \overline{\Bbb H}$ be a wall for $v$ and take $(s,t) \in W$
such that $s \in {\Bbb Q}$. 
Then $W$ is a codimension 1 wall if and only if
$W$ is defined by
$v_1$ such that
$v=v_1+v_2$, $\langle v_1^2 \rangle=0$,
$\langle v,v_1 \rangle=2$, $v_1$ is primitive and
there are $\sigma_{(\beta,\omega)}$-stable objects $E_i$ with
$v(E_i)=v_i$ for $i=1,2$.
\end{prop}

\begin{proof}
Since $\NS(X)={\Bbb Z}H$,
there is no decomposition
$v=u_1+u_2+u_3$ such that
$\langle u_i^2 \rangle=0$ $(i=1,2,3)$ and
$\langle u_i,u_j \rangle=1$ $(i \ne j)$.
\begin{NB}
$v \perp u_1-u_2,u_1-u_3$ and
$u_1-u_2,u_1-u_3$ spans a negative definite lattice of type $A_2$.
Since the signature of $H^*(X,{\Bbb Z})_{\alg}$
is $(2,1)$, it is a contradiction.
\end{NB}  
Then the classification of codimension 1 walls in 
\cite[Lem. 4.3.4 (2), Prop. 4.3.5]{MYY:2011:1} imply that
$W$ is defined by $v_1$ with the required properties.
\end{proof}

\begin{NB}
\begin{prop}\label{prop:movable}
Assume that ${\frak I} \ne \emptyset$ and $\sqrt{\ell/n} \not \in {\Bbb Q}$.
\begin{enumerate}
\item[(1)]
There are $w_1,w_2 \in {\frak I}$ such that
$W_{w_i}$ $(i=1,2)$ are the boundary of $\overline{{\frak D}(sH,tH)}$.
\item[(2)]
$\Mov(K_{(sH,tH)}(v))$ is spanned by
$\theta_v(\xi(\lambda_i,0))$, where $\lambda_i H =c_1(w_i)/\rk w_i$.
\item[(3)]
There are birational models $K_{(s_i H,t_i H)}(v)$
($(s_i,t_i) \in {\frak D}(sH,tH)$) such that
$\theta_v(\xi(\lambda_i,0))$ give   
divisorial contractions. 
Moreover the exceptional divisor
$D_i$ is primitive in $\NS(K_{(s_i H,t_i H)}(v))$
if and only if $\langle v,w_i \rangle=2$.
\end{enumerate}
\end{prop}

\begin{proof}
(1) is obvious by Lemma \ref{lem:I}.
So we prove (2) and (3).
We take $(s_i,t_i) \in {\frak C}$ which are close to
$W_{w_i}$. Then
Proposition \ref{prop:codim0-div} and
Proposition \ref{prop:codim1-div} (2) imply that
$\theta_v(\xi(\lambda_i,0))$
defines a divisorial contraction
of $K_{(s_i H,t_i H)}(v)$.
Moreover $D_i=\pm \frac{2}{\langle v,w_i \rangle}\theta_v(d_{w_i})$.
Since $\codim (M_{(s H,t H)}(v) \setminus
(M_{(s_1 H,t_1 H)}(v) \cap M_{(s_2 H,t_2 H)}(v))) \geq
2$, we can identify
$\NS(K_{(s H,t H)}(v))$ with $\NS(K_{(s_i H,t_i H)}(v))$
via the natural birational identification. 
Thus (2) and (3) holds.
\end{proof}
\end{NB}

\begin{thm}\label{thm:movable}
Let $(X,H)$ be a polarized abelian surface $X$ with
$\NS(X)={\Bbb Z}H$. 
Let $v$ be a primitive Mukai vector with $\langle v^2 \rangle \geq 6$.
Assume that $\sqrt{\ell/n} \not \in {\Bbb Q}$.
\begin{enumerate}
\item[(1)]
Assume that ${\frak I}_1={\frak I}_2= \emptyset$, that is,
\begin{equation*}
\min \{ \langle v,w \rangle>0 \mid \langle w^2 \rangle=0 \} \geq 3.
\end{equation*}
Then the movable cone of $K_{(sH,tH)}(v)$ is the same as the positive cone
of $K_{(sH,tH)}(v)$.
In this case, there is an action of birational automorphisms
such that a fundamental domain is 
a cone spanned by rational vectors. 
\item[(2)]
Assume that ${\frak I}_1 = \emptyset$ and ${\frak I}_2 \ne \emptyset$,
that is,
\begin{equation*}
\min \{ \langle v,w \rangle>0 \mid \langle w^2 \rangle=0 \}=2.
\end{equation*}
Then the movable cone of $K_{(sH,tH)}(v)$ is spanned by
two vectors, which give divisorial contractions.
Moreover the exceptional divisors are primitive 
in $\NS(K_{(sH,tH)}(v))$. 
\item[(3)]
Assume that ${\frak I}_1 \ne \emptyset$, that is,
\begin{equation*}
\min \{ \langle v,w \rangle>0 \mid \langle w^2 \rangle=0 \}=1.
\end{equation*}
Then the movable cone of $K_{(sH,tH)}(v)$ is spanned by
two vectors, which give divisorial contractions.
Moreover one of the exceptional divisors is divisible by 2 
in $\NS(K_{(sH,tH)}(v))$. 
\end{enumerate}
\end{thm}

\begin{proof}
(1)
\begin{NB}
Assume that $(\beta,\omega),(\beta',\omega')$
does not lie on walls.
By the assumption,
$M_{(\beta,\omega)}(v) \setminus M_{(\beta',\omega')}(v)$
and
$M_{(\beta',\omega')}(v) \setminus M_{(\beta,\omega)}(v)$
are proper closed subsets of codimension $\geq 2$. 
Hence we have a natural birational map
$M_{(\beta,\omega)}(v) \to M_{(\beta',\omega')}(v)$
with an isomorphism 
$\NS(M_{(\beta,\omega)}(v)) \to \NS(M_{(\beta',\omega')}(v))$.
\end{NB}
\begin{NB}
Assume that $(s,t)$ and $(s',t')$ do not lie on walls.
Since there is no wall of codimension 0,1, 
by the proof of Proposition \ref{prop:movable},
we have a natural birational identification
$K_{(sH,tH)}(v) \cdots \to K_{(s'H,t'H)}(v)$
with an identification
$\NS(K_{(sH,tH)}(v)) \to \NS(K_{(s'H,t'H)}(v))$.
Hence $\cup_{(s,t)} \Nef(K_{(sH,tH)}(v)) \subset 
\Mov(K_{(sH,tH)}(v))$.
Since $\cup_{(s,t)} \Nef(K_{(sH,tH)}(v))=
\overline{P^+(K_{(sH,tH)}(v))}$
by Proposition \ref{prop:positive-cone}, we get
the first claim.
\end{NB}
Let $\Phi:{\bf D}(X) \to {\bf D}(X)$ be a Fourier-Mukai transform
preserving $\pm v$.
We set $(s'H,t'H):=\Phi((sH,tH))$.
Then we have an isomorphism
$\Phi:M_{(sH,tH)}(v) \to M_{(s' H,t' H)}(v)$, which induces
a birational map  
$$
M_{(sH,tH)}(v) \overset{\Phi}{\to} M_{(s' H,t' H)}(v) 
\cdots \to M_{(sH,tH)}(v).
$$
Since $\theta_v:v^{\perp} \to \NS(K_{(sH,tH)}(v))$
is compatible with respect to the action of $\Phi$,
we have an action of $\Stab_0(v)^*$ on
$\NS(K_{(sH,tH)}(v))$. 
Since $\sqrt{\ell/n} \not \in {\Bbb Q}$, Lemma \ref{lem:S_H}
implies that
$\Stab_0(v)^*$ contains an element $g$ of infinite order. 
Hence the claim holds.

(2) and (3) are consequence of Corollary
\ref{cor:movable}.
\end{proof}

\begin{NB}
\begin{lem}
Assume that $\sqrt{\ell/n}\in {\Bbb Q}$ and $\ell \geq 3$.
Then there is at most one isotropic Mukai vector
$w$ such that $\langle v,w \rangle=1$.
In this case, $\langle v,w' \rangle \ne \pm 2$
for any isotropic Mukai vector $w'$ with $w' \ne 2w$. 
\end{lem}

\begin{proof}
By the action of $\widehat{G}$, 
we may assume that $v=(1,0,-\ell)$.
Assume that there is a primitive isotropic
Mukai vector $w$ with
$\langle v,w \rangle=1,2$.
We can set $w=\pm (p^2 r,pq H,q^2 s)$
with $2rs=(H^2)$, $r,s>0$.
Then $\ell p^2 r-q^2 s= \pm 1, \pm 2$.
By our assumption, $\ell rs$ is a square number. 

If $s$ is odd, then $\gcd(s,\ell r)=1$, which implies that
$s$ and $\ell r$ are square numbers.
So we can set $\ell p^2 r=x^2$ and $q^2 s=y^2$.
Then $\pm (x-y)(x+y)=1,2$.
It is easy to see that $(x^2,y^2)=(1,0), (0,1)$.
Hence $\ell=r=p^2=1$ and $q=0$, or
$p=0$ and $s=q^2=1$.
Since $\ell \geq 3$, the first case does not hold.
For the second case, 
we get $w=(0,0,-1)$.

If $s$ is even,
we have $2|\ell p^2 r$.
Since $\frac{s}{2} \frac{\ell p^2 r}{2}$ is a
square number and $\gcd(\frac{s}{2}, \frac{\ell p^2 r}{2})=1$,
we can set 
$\frac{\ell p^2 r}{2}=y^2$ and
 $\frac{s}{2}q^2=y^2$.
Since $x^2-y^2=\pm 1$ and $\ell \geq 3$, by the same argument,
we get $w=(0,0,-1)$.
\end{proof}
\end{NB}

\begin{rem}
By Lemma \ref{lem:no-wall}, there are $v$ satisfying (1).
For $v=(2,H,-2k)$, we have $\langle v^2 \rangle=2(n+4k)$ and
case (2) holds, if $\sqrt{\ell/n} \not \in {\Bbb Q}$.
\begin{NB}
$n(n+4k)=\ell (\ell-4k)$, where $n=\ell-4k \geq 1$.
\begin{NB2}
Assume that $k=1$. Then we have $\ell>4$ and $n=\ell-4$.
If $\ell(\ell-4)=y^2$, $y \in {\Bbb Z}_{>0}$, then
$(\ell-2-y)(\ell-2+y)=4$. Hence 
$\ell-2-y>0$ and $4 \geq \ell-2+y>\ell-2$.
Then $\ell \leq 5$.
If $\ell=5$, then $y^2=5(5-4)=5$. Therefore $\ell \ne 5$.
\end{NB2}
Assume that $k=-1$. Then $n=\ell+4$.
If $\ell(\ell+4)=y^2$, $y \in {\Bbb Z}_{>0}$, then
$(\ell+2-y)(\ell+2+y)=4$. Hence 
$\ell+2-y>0$ and $4 \geq \ell+2+y>\ell+2$.
Then $\ell \leq 1$.
If $\ell=1$, then $1(1+4)=y^2$, which is a contradiction.
Therefore $\sqrt{n \ell}=
\sqrt{\ell(\ell+4)} \not \in {\Bbb Q}$.
\end{NB}
If $\rk v=1$, then case (3) holds.
\end{rem}

\begin{prop}\label{prop:movable(rational)}
Let $v$ be a primitive Mukai vector with $\langle v^2 \rangle \geq 6$. 
Assume that $\sqrt{\ell/n} \in {\Bbb Q}$.
\begin{enumerate}
\item[(1)]
There is at most one isotropic Mukai vector $v_1$
with $\langle v,v_1 \rangle=1,2$.
\item[(2)]
If there is a vector $v_1$ of (1), then 
$$
\overline{P^+(v^\perp)}_{\Bbb R}=
\Mov(K_{(sH,tH)}(v))_{\Bbb R} \cup 
R_{v_1}(\Mov(K_{(sH,tH)}(v)_{\Bbb R})
$$
and the two chambers are separated by $d_{v_1}^{\perp}$.
\item[(3)]
If there is no $v_1$ of (1), then 
$\overline{P^+(v^\perp)}_{\Bbb R}=
\Mov(K_{(sH,tH)}(v))_{\Bbb R}$.
\end{enumerate}
\end{prop}

\begin{proof}
(1)
Since $\sqrt{\ell/n} \in {\Bbb Q}$,
there are two isotropic Mukai vectors
$w_1, w_2$ such that
$$
\{ x \in H^2(X,{\Bbb Z})_{\alg} \mid
 \langle x,v \rangle=\langle x^2 \rangle =0
\}
={\Bbb Z}w_1 \cup {\Bbb Z}w_2.  
$$
Then $v^\perp \otimes {\Bbb Q}={\Bbb Q}w_1 + {\Bbb Q} w_2$.
We may assume that $\langle w_1,w_2 \rangle<0$.
Let $v_1$ be an isotropic Mukai vectors such that
$\langle v,v_1 \rangle=1,2$.
Since $\langle v,d_{v_1} \rangle=0$,
we set $d_{v_1}:=a w_1+ b w_2$ ($a,b \in {\Bbb Q}$).
Then we have $2ab \langle w_1,w_2 \rangle=-\langle v^2 \rangle<0$.
By Lemma \ref{lem:reflection} (2), $R_{v_1}$ preserves
$\{ \pm w_1,\pm w_2 \}$.
Since 
$$
R_{v_1}(w_1)=w_1-\frac{1}{a}(a w_1+bw_2)=-\frac{b}{a}w_2,
$$
$R_{v_1}(w_1)=\pm w_2$.
If $R_{v_1}(w_1)=w_2$, then we have $R_{v_1}(w_1+w_2)=w_1+w_2$.
Hence $\langle d_{v_1},w_1+w_2 \rangle=0$.
Since $\langle (w_1+w_2)^2 \rangle=2\langle w_1,w_2 \rangle<0$
and $\langle d_{v_1}^2 \rangle<0$,
$v^\perp$ is negative definite, which is a contradiction.
Therefore $R_{v_1}(w_1)=-w_2$.
Then we see that $w_1+w_2 \in {\Bbb Z}d_{v_1}$.
If $\langle v,v_2 \rangle=1,2$, then 
the primitivities of $d_{v_1}$ and $d_{v_2}$ imply that
$d_{v_1}=\pm d_{v_2}$.
If $d_{v_1}=-d_{v_2}$, then
we see that 
$$
v=\frac{\langle v^2 \rangle}{4}
\left(
\frac{2}{\langle v,v_1 \rangle}v_1+\frac{2}{\langle v,v_2 \rangle}v_2
\right).
$$
Since $\langle v^2 \rangle \geq 6$ and $v$ is primitive, 
this case does not occur.
Therefore $d_{v_1}=d_{v_2}$, which implies that
$v_1=v_2$.

(2) and (3) are obvious.
\end{proof}

\begin{rem}\label{rem:oguiso}
Theorem \ref{thm:movable} and Proposition \ref{prop:movable(rational)}
are compatible with Oguiso's general results
\cite[Thm. 1.3]{Oguiso}.
\end{rem}

\begin{NB}

For the Mukai vector $v$ in Proposition \ref{prop:movable(rational)},
we shall explain the morphisms induced by the boundaries of positive cone. 
We first note that there are finitely many chambers,
so we can take chambers $C_\pm$ 
such that $(s_\pm,0) \in \overline{C_\pm}$.
We take $w_\pm \in {\frak I}_0$
such that $c_1(w_\pm)/\rk w_\pm=s_\pm H$, $\rk w_\pm >0$ .
We set $Y_\pm:=M_H(w_\pm)$ and let ${\bf E}_\pm$ are the universal families
as twisted objects.
Then we have (twisted) Fourier-Mukai transforms 
$\Phi_{X \to Y_\pm}^{{\bf E}_\pm^{\vee}[k] }:
{\bf D}(X) \to {\bf D}^{\alpha_\pm}(Y_\pm)$,
where $\alpha_\pm$ are suitable 2-cocycles of ${\cal O}_{Y_\pm}^{\times}$
defining ${\bf E}_\pm^{\vee}$
and $k=1$ for ${\bf E}_-$, $k=2$ for ${\bf E}_+$.
For $(s,t) \in C_\pm$,
we have isomorphisms $M_{(sH,tH)}(v) \to M_{H_\pm}^\alpha(v')$
such that $\rk v'=0$.
We note that
the scheme-theoretic support $\Div(E)$ of a purely 1-dimensional
sheaf is well-defined even for
a twisted sheaf.
Hence we have a morphism
 $f:M_{H_\pm}^\alpha(v') \to \Hilb_{Y_\pm}^\eta$
by sending $E \in M_{H_\pm}^\alpha(v')$
to $\Div(E)$, where $\Hilb_{Y_\pm}^\eta$
is the Hilbert scheme of effective divisors $D$
on $Y_\pm$ with $\eta=c_1(D)$.
For a smooth divisor $D$, $f^{-1}(D) \cong \Pic^0(D)$.
Hence $f$ is dominant, which implies $f$ is surjective. 
Therefore
we get a surjective morphism
$M_{(sH,tH)}(v) \to \Hilb_{Y_\pm}^\eta$.
Combining with the properties of the albanese map,
we have a commutative diagram:
\begin{equation*}
\begin{CD}
M_{(sH,tH)}(v) @>>> \Hilb_{Y_\pm}^\eta \\
@VVV @VVV\\
\Pic^0(X) \times X @>>> \Pic^0(Y_\pm)
\end{CD}. 
\end{equation*}
Hence we get a morphism
$K_{(sH,tH)}(v) \to {\Bbb P}(H^0(Y_\pm,{\cal O}(D)))$, where
$D \in \Hilb_{Y_\pm}^\eta$.
Then we see that $\theta_v(\xi(s_\pm,0))$ comes from
 ${\Bbb P}(H^0(Y_\pm,{\cal O}(D)))$.
Thus $\theta_v(\xi(s_\pm,0))$ give Lagrangian fibrations.  
\end{NB}

\begin{NB}
Let $M_1,M_2$ be smooth projective manifolds with
trivial canonical bundles, and
$f:M_1 \to M_2$ a birational map.
Then there are closed subsets $Z_i \subset M_i$
of codimension $\geq 2$ such that
$f$ induces an isomorphism
$f:M_1 \setminus Z_1 \to M_2 \setminus Z_2$.
Then we have an isomorphism
$f_*:\NS(M_1) \to \NS(M_2)$.
Since $\codim Z_i \geq 2$,
we also have an isomorphism
$f_*:H^0(M_1,L) \to H^0(M_2,f_*(L))$.
Therefore $f_*$ induces a bijection
of movable divisors, and we get 
$f_*(\Mov(M_1))=\Mov(M_2)$.
\end{NB}

\begin{NB}
\begin{lem} 
Let $v=(r,dH,a)$ be a primitive Mukai vector with
$\ell:=\langle v^2 \rangle/2 \geq 3$.
\begin{enumerate}
\item[(1)]
Assume that $r=1$.
For the Hilbert-Chow morphism
$\Hilb_X^{\ell} \to S^{\ell} X$,
the exceptional divisor $E$ is divisible by 2
and is given by $E_{|K_H(v)}=2\theta_v((1,0,\ell))$.
\item[(2)]
Assume that $r=2$.
For the Gieseker-Uhlenbeck morphism,
the exceptional divisor $E$ is primitive
and $E_{|K_H(v)}=\theta_v(v+(0,0,\ell))$
\end{enumerate}
\end{lem}

\begin{proof}
The first claim is well-known.
So we assume that $r=2$.
For the Gieseker-Uhlenbeck morphism,
the exceptional divisor $E$ parameterizes
non-locally free sheaves.
By \cite[Lem. 4.4.1]{MYY:2011:1}, we have 
$E_{|K_H(v)}=\theta_v(v+(0,0,\ell))$.
We shall prove that 
$v+(0,0,\ell)$ is primitive.
If $v+(0,0,\ell)=(2,dH,a+\ell)$
is not primitive, then
$2|d$ and $2|(a+\ell)$.
Since $\ell=d^2 n-2a$,
we have $2|\ell$. Then $2|(a+\ell)$ implies $2|a$.
Since $v$ is primitive, it is impossible.
Therefore $v+(0,0,\ell)$ is primitive.
\end{proof}

Let $v$ be a primitive Mukai vector such that
$\langle v^2 \rangle \geq 6$.
Assume that there is an isotropic Mukai vector
$w$ such that $\langle v,w \rangle=2$.
Then $w$ defines a wall $W$. Let $C^\pm$ be the chambers
adjacent to $W$.  
For $(s,t) \in W$, let $Y:=M_{(sH,tH)}(w)$ be the moduli space
of stable complexes and ${\bf E}$ a universal family
as a twisted object.
By the Fourier-Mukai transform 
$\Phi_{X \to Y}^{{\bf E}^{\vee}[1]}:{\bf D}(X) \to
{\bf D}^\alpha(Y)$,
$M_{C^\pm}(v)$ are transformed to the moduli spaces
$M_{H'}^\alpha(u)$ and 
the moduli space 
$\{F|F^{\vee} \in M_{H'}^{-\alpha}(u^{\vee})\}$,
where $\alpha$ is a 2-cocycle of ${\cal O}_X^{\times}$. 

There are irreducible divisors $E_\pm \subset M_{C^\pm}(v)$
which correspond to the divisors of non-locally free sheaves.
Hence $E_\pm$ are primitive in $M_{C^\pm}(v)$.
By \cite[Cor. 4.3.3 (2)]{MYY:2011:2},
$\theta_v((d-rs)\xi_{tH})$ contracts
$E_\pm$.

\end{NB}

\begin{NB}
\begin{prop}
Let $(X,H)$ be a polarized abelian surface with
$\NS(X)={\Bbb Z}H$.
Let $v=(r,dH,a)$ be a Mukai vector such that 
$r>0$ and $2\ell:=\langle v^2 \rangle \geq 6$.
Then $M_H(v)$ is birationally equivalent
to $\Pic^0(Y) \times \Hilb_Y^{\ell}$
if and only if there is an isotropic Mukai
vector $w \in H^*(X,{\Bbb Z})_{\alg}$ with $\langle v,w \rangle=1$,
where $Y$ is an abelian surface.   
\end{prop}

\begin{proof}
If there is an isotropic Mukai
vector $w$ with $\langle v,w \rangle=1$,
then $M_H(w)$ is a fine moduli space and 
the claim follows by \cite[Thm. 0.1]{Y}.

Conversely
if $M_H(v)$ is birationally equivalent to
$\Pic^0(Y) \times \Hilb_Y^{\ell}$, then
we have a birational map $f:K_H(v) \to K_{H'}(1,0,-\ell)$,
where $H'$ is an ample divisor on $Y$.
Then we have an isomorphism 
$f_*:\NS(K_H(v)) \to \NS(K_{H'}(1,0,-\ell))$.
Since $\rk \NS(K_{H'}(1,0,-\ell))=\rk \NS(Y)+1$,
we get $\rk \NS(Y)=1$.
By the isomorphism $f_*$,
the movable cones are isomorphic.
By Theorem \ref{thm:movable}, ${\frak I}_1 \ne \emptyset$.
\end{proof}

\end{NB}

\begin{NB}
\subsection{}

Assume that $(H^2)=2$.
Then every
isotropic Mukai vector
is written as $w=(p^2,pqH,q^2)$, $(p,q)=1$.

For $v=(3,H,-3)$,
there is no isotropic Mukai vector $w$ with
$\langle v,w \rangle=\pm 1$.

$v$ is equivalent to $(2,0,-5)$.

\end{NB}

\section{Relations with Markman's results.}\label{sect:Markman}

In this section, we shall explain the relation between 
our results and Markman's general results \cite{Mark}, \cite{Ma}, \cite{Ma2}.

\begin{defn}
Let $M$ be an irreducible symplectic manifold and $h$ an ample class.
\begin{enumerate}
\item[(1)]
An effective divisor $E$ is prime exceptional, if
$E$ is reduced and irreducible of $q_M(E^2)<0$.
\item[(2)]
Let $e \in \NS(M)$ be a primitive class.
$e$ is stably prime exceptional,
if $q_M(e,h)>0$ and there is a projective irreducible symplectic
manifold $M'$, a parallel-transport operator
$$
g:H^2(M,{\Bbb Z}) \to H^2(M',{\Bbb Z}),
$$
 and an integer $k$,
such that $kg(e)$ is the class of a prime exceptional divisor
$E \subset M'$.
\end{enumerate}
\end{defn}

Let $\Spe_M$ be the set of stably prime exceptional 
divisors of $M$.
Then Markman described the interior of the movable cone in terms of
$\Spe_M$.

\begin{thm}[{\cite[Prop. 1.8, Lem. 6.22]{Ma}}]\label{thm:Markman}
Let $M$ be an irreducible symplectic manifold and $\Mov(M)^0$ 
the interior of $\Mov(v)_{\Bbb R}$. Then 
\begin{equation*}
\Mov(M)^0=\{ x \in P^+(M) \mid q_M(x,e)>0 \text{ for all $e \in \Spe_M$} \}. 
\end{equation*}
\end{thm}

Let $M$ be an irreducible symplectic manifold
which is deformation equivalent to
the generalized Kummer variety $K_H(1,0,-\ell)$
of dimension $2\ell-2$.
We shall explain the description of $\Spe_M$. 
For $e \in H^2(M,{\Bbb Z})$ with $q_M(e^2)=-2\ell$,
we set 
$$
\divisor q_M(e,\bullet):=\min \{ q_M(e,x)>0 \mid x \in H^2(M,{\Bbb Z}) \}.
$$
As an abstract lattice,
the cohomological Mukai lattice $H^{2*}(X,{\Bbb Z})=
\bigoplus_{i=0}^2 H^{2i}(X,{\Bbb Z})$
is independent of the choice of an abelian surface $X$.
We denote this lattice by $\widetilde{\Lambda}$.
It is a direct sum of 4 copies of
the hyperbolic lattice.
Since $M$ is deformation equivalent to
$K_H(1,0,-\ell)$,
by using a parallel-transport, 
we have a primitive embedding
$H^2(M,{\Bbb Z}) \to \widetilde{\Lambda}$.
The $\LieO(\widetilde{\Lambda})$-orbit of the embedding is 
independent of the choice of a parallel-transport 
by similar claims to \cite[Thm. 9.3, Thm. 9.8]{Ma}.
For $M=K_H(v)$, it is the embedding
$$
H^2(K_H(v),{\Bbb Z}) \overset{\theta_v^{-1}}{\to} v^\perp
\subset H^{2*}(X,{\Bbb Z})
$$
 (\cite[Example 9.6]{Ma}).
We fix an embedding and 
regard $H^2(M,{\Bbb Z})$ as a sublattice of
$\widetilde{\Lambda}$. 
Let ${\Bbb Z}v$ be the orthogonal compliment of
$H^2(M,{\Bbb Z})$ in $\widetilde{\Lambda}$.
Then $\langle v^2 \rangle =2\ell$.
Since $\langle v,e \rangle=0$,
$e \pm v$ are isotropic.
We define $(\rho,\sigma) \in {\Bbb Z}_{>0} \times {\Bbb Z}_{>0}$
by requiring that $(e+v)/\rho$ and $(e-v)/\sigma$ are primitive and isotropic.
We also set $r:=\rho/\gcd(\rho,\sigma)$ and
$s:=\sigma/\gcd(\rho,\sigma)$.
If $\ell \mid \divisor q_M(e,\bullet)$, then
$r$ and $s$ are relatively prime integers with
$rs=\ell,\ell/2,\ell/4$. 
We set $rs(e):=\{r,s \}$.

\begin{prop}[{\cite{Ma2}}]
Let $M$ be an irreducible symplectic manifold of 
$\dim M=2\ell-2$ which is deformation equivalent
to $K_H(1,0,-\ell)$.
\begin{enumerate}
\item[(1)]
For $e \in \Spe_M$, $q_M(e^2)=-2\ell$ and 
$\ell \mid \divisor q_M(e,\bullet)$.
\item[(2)]
For $e \in H^2(M,{\Bbb Z})$ with $q_M(e^2)=-2\ell$ and 
$\ell \mid \divisor q_M(e,\bullet)$, the orbit of $e$ of the monodromy
group action is classified by
$rs(e):=\{r,s \}$ and $\divisor q_M(e,\bullet)$.
\end{enumerate}
\end{prop}
For each value of $\{r,s \}$ with $rs \in \{\ell,\ell/2,\ell/4 \}$, 
the same examples in \cite[sect. 10, 11]{Mark} show that
there are $K_H(v)$ with
$\dim K_H(v)=2\ell-2$ and $e \in \NS(K_H(v))$
such that 
$\divisor q_M(e,\bullet)=\ell, 2\ell$ and $rs(e)=\{r,s \}$.
Then we also get the following description of $\Spe_M$.

\begin{prop}
Let $M$ be an irreducible symplectic manifold of 
$\dim M=2\ell-2$ which is deformation equivalent
to $K_H(1,0,-\ell)$ and $h$ an ample divisor on $M$.
A divisor
$e$ with $q_M(e,e)=-2\ell$ and $q_M(e,h)>0$
is stably prime exceptional if and only if
\begin{enumerate}
\item[(1)]
$\divisor q_M(e, \bullet)=2\ell$ and $\{r,s\}=\{1,\ell\}$ or
\item[(2)]
$\divisor q_M(e,\bullet)$ and $\{ r,s \}$ are one of the following.
\begin{enumerate}
\item
$\divisor q_M(e,\bullet)=2 \ell$ and 
$\{r,s \}=\{2,\ell/2\}$, $\ell \geq 6$, $\ell \equiv 2 \mod 4$. 
\item
$\divisor q_M(e,\bullet)=\ell$ and
$\{r,s \}=\{1,\ell \}$, $\ell \geq 3$, $2 \nmid \ell$. 
\item
$\divisor q_M(e,\bullet)=\ell$ and
$\{r,s \}=\{1,\ell/2\}$, $\ell \geq 2$, $2 \mid \ell$.
\end{enumerate} 
\end{enumerate}
\end{prop}

If $M=K_{(\beta,\omega)}(v)$ for some $v$, then
we shall explain that
the case (1) corresponds to the codimension 0 wall
$u^\perp \;(u \in {\frak I}_1)$
and the case (2) corresponds to the codimension 1 wall
$u^\perp \;(u \in {\frak I}_2)$.

(1)
We first assume that $u \in {\frak I}_1$.
Since $d_u=v-2\ell u$,
$v+d_u=2(v-\ell u)$ and $v-d_u=2\ell u$.
Since
$\langle d_u,\;\;\rangle=-2\ell \langle u,\;\; \rangle$ on $v^\perp$,
$\divisor q_{K_H(v)}(d_u,\bullet)=2\ell$.
Since $\langle v+d_u,u \rangle=2$,
$v+d_u$ is primitive, which implies that
$\{\rho,\sigma\}=\{2,2\ell\}$.
Therefore $\{r,s \}=\{1,\ell \}$. 

(2)
We next assume that $u \in {\frak I}_2$. 
In this case, we have $d_u=v-\ell u$.
Thus
$v+d_u=2v-\ell u$ and $v-d_u=\ell u$.
We shall compute $\divisor q_{K_H(v)}(d_u,\bullet)$.
We first note that
$\langle d_u,\;\; \rangle=-\ell \langle u,\;\; \rangle$
on $v^\perp$.
Since $H^*(X,{\Bbb Z})$ is a 4 copies of hyperbolic lattice,
there is an isotropic Mukai vector $\lambda \in H^{2*}(X,{\Bbb Z})$
with $\langle u,\lambda \rangle=1$.
Then 
$$
H^{2*}(X,{\Bbb Z})=({\Bbb Z}u+{\Bbb Z}\lambda) \oplus
 ({\Bbb Z}u+{\Bbb Z}\lambda)^\perp.
$$
We set 
\begin{equation}\label{eq:div(q)}
v:=2 \lambda+a u+\xi, \; a \in {\Bbb Z}, \xi \in  
({\Bbb Z}u+{\Bbb Z}\lambda)^\perp.
\end{equation}
Then $x \lambda+y u+z \eta$ $(x,y,z \in {\Bbb Z}, \eta \in 
({\Bbb Z}u+{\Bbb Z}\lambda)^\perp)$
belongs to $v^\perp$ if and only if
$xa+2y+z\langle \xi,\eta \rangle=0$.
If $2 \nmid \xi$, then the unimodularity of
$({\Bbb Z}u+{\Bbb Z}\lambda)^\perp$ implies that
we can take $\eta$ with $2 \nmid \langle \xi,\eta \rangle$.
We take $z \in {\Bbb Z}$ such that
$y=-(a+z \langle \xi,\eta \rangle)/2 \in {\Bbb Z}$.
Then $\lambda+y u+z \eta \in v^\perp$ and
$\langle d_u,\lambda+y u+z \eta \rangle=
-\ell \langle u,\lambda+y u+z \eta \rangle=-\ell$.
Therefore $\divisor q_{K_H(v)}(d_u,\bullet)=\ell$.
If $2 \mid \xi$, then the primitivity of $v$ implies that
$2 \nmid a$. Hence $x \lambda+y u+z \eta \in v^\perp$
satisfies $2 \mid x$.
Then we have 
$$
\langle d_u,x\lambda+yu+z \eta \rangle=
-\ell \langle u,x \lambda+yu+z \eta \rangle=
-\ell x \langle u,\lambda \rangle \in 2\ell {\Bbb Z}.
$$
Hence $\divisor q_{K_H(v)}(d_u,\bullet)=2\ell$.
Therefore $\divisor q_{K_H(v)}(d_u,\bullet)=2\ell$ if and only if
there is $a \in {\Bbb Z}$ such that $2 \mid (v-a u)$.

We next compute $\{r,s\}$.
We take $w$ such that ${\Bbb Z}u+{\Bbb Z}w$ is a saturated
sublattice of $H^*(X,{\Bbb Z})_{\alg}$ containing $v$.
For the notation of \eqref{eq:div(q)},
$w=2\lambda+\xi$ or $2w=2\lambda+\xi$.
Thus they are distinguished by
$\divisor q_{K_H(v)}(d_u,\bullet)$.
We set $v=a u+b w$ ($b=1,2$).
Then $v+d_u=(2a-\ell)u+2b w$.
We note that $\ell=2a+b^2 \langle w^2 \rangle/2 $ by
$b \langle u,w \rangle=2$.
\begin{NB}
Old version:
We note that $2=\langle v,u \rangle=b\langle u,w \rangle$.
Thus $b=\pm 1, \pm 2$.
\end{NB}
If $2 \nmid \ell$, then 
$v+d_u$ is primitive, which implies that
$\{\rho,\sigma\}=\{1,\ell \}$.
In this case, $b=1$ and $\divisor q_{K_H(v)}(d_u,\bullet)=\ell$.
Assume that $2 \mid \ell$.
Then $(v+d_u)/2=(a-\ell/2)u+b w \in H^*(X,{\Bbb Z})_{\alg}$.
If $b=1$, then
$(v+d_u)/2$ is primitive, which implies 
$\{\rho,\sigma\}=\{2,\ell \}$.
If $b =2$, then
the primitivity of $v$ implies that $2 \nmid a$.
Then $\ell \equiv 2 \mod 4$ and
$(v+d_u)/4 \in H^*(X,{\Bbb Z})_{\alg}$ is primitive, which implies 
$\{\rho,\sigma\}=\{4,\ell \}$.
%If $\ell \equiv 0 \mod 4$, then
%$(v+d_u)/2$ is primitive, which implies 
%$\{\rho,\sigma\}=\{2,\ell \}$.
Therefore $\{r,s\}$ satisfies (a),  (b) or (c).

\begin{NB}
\section{An Example.}

Let $C_1$ and $C_2$ are non-isogenous elliptic curves and 
set $X=C_1 \times C_2$.
Then $\NS(X)={\Bbb Z}\sigma \oplus {\Bbb Z}f$,
$(f^2)=(\sigma^2)=0$ and $(\sigma,f)=1$.
We set $v:=(1,0,-\ell)$.
(1)
$\ell=3$.
We have decompositions
\begin{equation*}
\begin{split}
v=& (1,0,0)+3(0,0,-1)\\
v=& (1,-\sigma-f,1)+(0,\sigma,-2)+(0,f,-2)\\
v=& (1,-f,0)+(0,f,-3)\\
v=& (1,-\sigma,0)+(0,\sigma,-3)
\end{split}
\end{equation*}

Codimension 0 walls:
\begin{equation*}
\begin{split}
v=& (1,0,0)+3(0,0,-1)\\
v=& (1,-3f,0)+3(0,f,-1)\\
v=& (1,-3\sigma,0)+3(0,\sigma,-1)\\
v=& (-2,3(\sigma+2f),-9)+3(1,-(\sigma+2f),2)\\
v=& (-2,3(2\sigma+f),-9)+3(1,-(2\sigma+f),2)\\
v=& (4,-6(\sigma+f),9)+3(-1,2(\sigma+f),-4)
\end{split}
\end{equation*}
These defines a hexagon.whose vertices are
$(0,f,0),(-1,\sigma+3f,-3),(-2,3\sigma+4f,-6),(-2,4\sigma+3f,-6),
(-1,3\sigma+f,-3),(0,\sigma,0)$ with center
$(-1,2\sigma+2f,-3)$.   

\begin{lem}
$\Mov(K_H(1,0,-3))$ is defined by
$(0,f,0)$, $(0,\sigma,0)$ and $(-1,2\sigma+2f,-3)$.   
\end{lem}

(2)
$\ell=4$.
We set $v_1=(r,a\sigma+bf,s)$. Replacing $v_1$ by
$v-v_1$, we may assume that $r \geq 1$.
Since $e^{2(\sigma+f)}=(1,2(\sigma+f),4)$,
we have $v e^{2(\sigma+f)}=(1,2(\sigma+f),0)$
and $v_1 e^{2(\sigma+f)}=(r,(a+2r)\sigma+(b+2r)f,s+2(a+b)+4r)$.
Assume that $(-1,y(\sigma+f),-4)$ with $y \geq 2$ belongs to
$v_1^\perp$. Then
$y(a+b)+4r+s=0$.
Thus $(a+b)((4r+s)+2(a+b)) \leq 0$.
Hence we see that
\begin{equation*}
\begin{split}
&(a+b+4r)(4-(a+b+4r))\langle v_1,v-v_1 \rangle\\
=&(4-(a+b+4r))^2 \langle v_1^2 \rangle/2+
(a+b+4r)^2 \langle (v-v_1)^2 \rangle/2
-(a+b)(2(a+b)+4r+s) \geq 0.
\end{split}
\end{equation*}
Therefore $(a+b+4r)(4-(a+b+4r)) \geq 0$.
Thus $0 \leq a+b+4r \leq 4$.
Since $r \geq 1$, $a+b \leq 0$.
If $(a+b+4r)(4-(a+b+4r))=0$, then
$(a+b)(2(a+b)+4r+s)=0$.
If $a+b=-4r$, then we have $s=4r$.
In this case, $\langle v_1^2 \rangle/2=ab-rs=-(a+2r)^2 \leq 0$.
Hence $a=b=-2r$. Therefore $v_1=r(1,-2(\sigma+f),4)$.
Since $\langle v_1,v \rangle=\langle v_1^2 \rangle=0$, 
we have $\langle v_1,v-v_1 \rangle=0$, which is a contradiction.
If $a+b=4-4r$, then $r=1$ or $s=4r-8$.
If $s=4r-8$, then $\langle (v-v_1)^2 \rangle/2=-(a+2(1-r))^2 \leq 0$.
Hence we see that
$v-v_1=(1-r)(1,-2(\sigma+f),4)$, which is a contradiction.
If $r=1$, then $v-v_1=(0,a(\sigma-f),0)$.
Since $\langle (v-v_1)^2 \rangle \geq 0$,
$a=0$, which is a contradiction.
Therefore $0 < a+b+4r < 4$.
In particular $a+b<0$ and $2(a+b)+4r+s \geq 0$.

Assume that $2(a+b)+4r+s> 0$.
Then we have
\begin{equation}\label{eq:example:v_1}
0 \leq \langle v_1^2 \rangle= \frac{1}{2}
\left((a+b+4r)^2-(a-b)^2-4r(s+2(a+b)+4r) \right) 
\leq \frac{1}{2}((a+b+4r)^2-4r).
\end{equation}
Hence $r=1,2$.
If $r=2$, then $a+b+4r=3$ and $\langle v_1^2 \rangle=0$.
Since 
$$
0 \leq \langle (v-v_1)^2 \rangle= \frac{1}{2}
\left((a+b+4r-4)^2-(a-b)^2-4(r-1)(s+2(a+b)+4r) \right) 
\leq \frac{1}{2}(1-4r)<0,
$$
we have $r=1$.

If $a+b=-1$, then $a-b=\pm 1$ and $s=-1,0$.
Thus $v_1=(1,-f,-1), (1,-f,0), (1,-\sigma,-1), (1,-\sigma,0)$.
If $a+b=-2$, then $a=b=1$ and $s=1$.
Hence $v_1=(1,-(\sigma+f),1)$.
If $a+b=-3$, then \eqref{eq:example:v_1} does not hold. 

Assume that $s+2(a+b)+4r=0$.
Then $(a+2r)(b+2r) \geq 0$, $(2-(a+2r))(2-(b+2r)) \geq 0$
and $(a+2r)(2-(b+2r))+(b+2r)(2-(a+2r))>0$.
Hence $4>(a+2r)+(b+2r)>2(a+2r)(b+2r) \geq 0$, which implies that
$a+2r,b+2r \geq 0$.  
We also have $2-(a+2r),2-(b+2r) \geq 0$.
If $a \geq b$, then $0 \leq b+2r \leq a+2r \leq 2$.
If $b+2r>0$, then $a+2r=b+2r=1$, which implies that
$(a+2r)+(b+2r)=2(a+2r)(b+2r)$.
Hence $b+2r=0$. Then $a+2r=1,2$.
Therefore $v_1=(r,(1-2r)\sigma-2r f,4r-2)$ or
$v_1=(r,(2-2r)\sigma-2rf,4r-4)$.
If $x(0,f,0)+(1-x)(0,\sigma,0) \in v_1^\perp$
for $0 \leq x \leq 1$, then
$x-2r=0$ for the first case and 
$2x-2r=0$ for the second case.
Therefore the first case does not occur and the second case 
occurs only for $r=1$ and $v_1=(1,-2f,0)$.
If $b \geq a$, then we also have $v_1=(1,-2\sigma,0)$.  

codimension 0 wall:
\begin{equation*}
v=(1,0,0)+4(0,0,-1).
\end{equation*}
codimension 1 walls:
\begin{equation*}
\begin{split}
v=&(1,-2f,0)+2(0,f,-2),\\
v=&(1,-2\sigma,0)+2(0,\sigma,-2).
\end{split}
\end{equation*}
They define a triangle with vertices
$(1,-2(\sigma+f),4)$, $(0,f,0)$ and $(0,\sigma,0)$.

\begin{lem}
$\Mov(K_H(1,0,-4))$ is spanned by 
$(1,-2(\sigma+f),4)$, $(0,f,0)$ and $(0,\sigma,0)$.
\end{lem}

\begin{lem}
Let $v_1$ be an isotropic Mukai vector.
There is an isotropic Mukai vector $w$ such that
$w=(x,a\sigma+bf,\ell x) \in v^\perp \cap v_1^\perp$.
\end{lem}

We first assume that $\rk v_1 \ne 0$.
We set $v_1=r e^{p \sigma+q f}=(r,r(p \sigma+q f),rpq)$.
Then $(pb+qa)-(pq+\ell)x=0$ and $\ell x^2=ab$.
Hence $x=(pb+qa)/(pq+\ell)$ and 
$\ell(pb+qa)^2=(pq+\ell)^2 ab$.
Then we have $(\ell b-q^2 a)(p^2 b-\ell a)=0$.
Therefore 
$$
w=(\rk w) e^{\frac{\ell}{q}\sigma+q f},
(\rk w) e^{p \sigma+\frac{\ell}{p}f}.
$$
We next assume that $v_1=r(0,f,q)$. Then
we see that $w=\rk w e^{q \sigma+\frac{\ell}{q}f}$
or $w=(0,f,0)$.
If $v_1=r(0,\sigma,p)$,
then we also 
see that $w=\rk w e^{\frac{\ell}{q}\sigma+q f}$
or $w=(0,\sigma,0)$.
If $v_1=(0,0,-1)$, then $w=(0,f,0), (0,\sigma,0)$.

For $(0,f,-2) \in {\frak I}_2$, 
$(-2,4\sigma+\ell f,-2 \ell),(0,f,0) \in (0,f,-2)^\perp$
are isotropic.
Thus $(0,f,-2)$ defines a codimension 1 wall with vertices
$(-2,4\sigma+\ell f,-2 \ell),(0,f,0)$.  

(3) $\ell=5$.
For $(1,-2(\sigma+f),4) \in {\frak I}_1$,
$(-2,4\sigma+5f,-10), (-2,5\sigma+4f,-10) \in (1,-2(\sigma+f),4)^\perp$
are isotropic.

\begin{lem}
$\Mov(K_H(1,0,-5))$ is spanned by $(0,\sigma,0),(0,f,0),
(-2,4\sigma+\ell f,-2\ell), (-2,\ell \sigma+4 f,-2\ell)$.
\end{lem}

(4) $\ell=6$.

For $(4,-10(\sigma+f),25) \in {\frak I}_1$,
$(-10,25 \sigma+24 f, -60), (-10,24 \sigma+25 f, -60)
\in (4,-10(\sigma+f),25)^\perp$ are isotropic.

\begin{lem}
$\Mov(K_H(1,0,-6))$ is spanned by $(0,\sigma,0),(0,f,0),
(-2,4\sigma+\ell f,-2\ell), (-2,\ell \sigma+4 f,-2\ell)$.
\end{lem}

(5)
$\ell=7$.
For $(1,-(2\sigma+3f),6) \in {\frak I}_1$,
$(1,-(2\sigma+3f),6)^\perp$ intersects with the boundary at
$(-2,4\sigma+7f,-14)$ and $(-3,7 \sigma+9f,-21)$.
For $(-1,3(\sigma+f),-9) \in {\frak I}_2$,
$(1,-3(\sigma+f),9)^\perp$ intersects with 
the boundary at $(-3,\ell \sigma +9 f,-3 \ell),
(-3,9 \sigma +\ell f,-3 \ell)$.

\begin{lem}
$\Mov(K_H(1,0,-7))$ is spanned by $(0,\sigma,0),(0,f,0),
(-2,4\sigma+\ell f,-2\ell), (-2,\ell \sigma+4 f,-2\ell),
(-3,\ell \sigma +9 f,-3 \ell),
(-3,9 \sigma +\ell f,-3 \ell)$.
\end{lem}

(6)
$\ell=8$.
For $(1,-(2\sigma+3f),6) \in {\frak I}_2$,
$(1,-(2\sigma+3f),6)^\perp$ intersects with the boundary at
$(-2,4\sigma+\ell f,-2 \ell)$ and $(-3,\ell \sigma+9f,-3 \ell)$.
For $(-1,3(\sigma+f),-9) \in {\frak I}_1$,
$(1,-3(\sigma+f),9)^\perp$ intersects with 
the boundary at $(-3,\ell \sigma +9 f,-3 \ell),
(-3,9 \sigma +\ell f,-3 \ell)$.

\begin{lem}
$\Mov(K_H(1,0,-8))$ is spanned by $(0,\sigma,0),(0,f,0),
(-2,4\sigma+\ell f,-2\ell), (-2,\ell \sigma+4 f,-2\ell),
(-3,\ell \sigma +9 f,-3 \ell),
(-3,9 \sigma +\ell f,-3 \ell)$.
\end{lem}

\end{NB}

\begin{NB}
\begin{lem}
$\xi \in \NS(X)$ is effective if and only if $(\xi^2) \geq 0$
and $(\xi,H)>0$ for an ample divisor $H$.
\end{lem}

\begin{proof}
If $(\xi^2)>0$, then the Riemann-Roch theorem implies that
$\xi$ or $-\xi$ is effective. Hence 
$(\xi,H)>0$ implies $\xi$ is effective.
If $(\xi^2)=0$, then for a line bundle with 
$c_1(L)=\xi$,
$\Phi_{X \to \Pic^0(X)}^{{\bf P}}(L) \ne 0$.
Hence $H^1(L \otimes {\bf P}_{|X \times \{y \}}) \ne 0$
for a point $y \in \Pic^0(X)$.
Then $H^0(L \otimes {\bf P}_{|X \times \{y \}}) \ne 0$
or $H^0(L^{\vee} \otimes {\bf P}_{|X \times \{y \}}^{\vee}) \ne 0$.
Hence 
$(\xi,H)>0$ implies $\xi$ is effective.
\end{proof}

We set $v:=(0,\xi,a)$ and assume that $\xi$ is effective.

\begin{lem}
We set $v_1:=(r_1,\xi_1,a_1)$.
\begin{enumerate}
\item[(1)]
$v_1 \in \varrho_X^\perp$
if and only if $r_1=0$.
\item[(2)]
If $v_1 \in \varrho_X^\perp$ 
defines a wall, then $\xi_1$ and $\xi-\xi_1$ are effective.
In particular the choice of $\xi_1$ is finite.
\end{enumerate}
\end{lem}

\begin{proof}
(1) is obvious. (2)
Since $v_2:=v-v_1 \in \varrho_X^\perp$,
$(\xi_1,\xi-\xi_1)>0$, $(\xi_1^2) \geq 0$ and
$((\xi-\xi_1)^2) \geq 0$.
Assume that $(\xi_1,H) \leq 0$. 
If the equality holds, then the Hodge index theorem implies that
$(\xi_1^2) \leq 0$. 
Then $(\xi_1^2)=0$, which implies that $\xi_1=0$.
In this case, $(\xi_1,\xi-\xi_1)=0$.
Therefore $(\xi_1,H) < 0$. Then  
$-\xi_1$ is effective. 
If $(\xi-\xi_1,H)>0$, then
$\xi-\xi_1$ is effective and $(\xi_1,\xi-\xi_1) \leq 0$, which is a 
contradiction. If $(\xi-\xi_1,H) \leq 0$, then
$(\xi,H)=(\xi_1,H)+(\xi-\xi_1,H)<0$.
Therefore we have $(\xi_1,H) > 0$, which implies that
$\xi_1$ is effective.
In the same way, we see that $\xi-\xi_1$ is effective. 
Thus the claim holds. 
\end{proof}

For $D \in \NS(X)$,
$T_D:H^*(X,{\Bbb Z})_{\alg} \to H^*(X,{\Bbb Z})_{\alg}$ 
is the multiplication 
by $e^D$. We set
$$
G:=\{T_D
\mid D \in \NS(X) \cap \xi^{\perp}\}.
$$
Then $v$ and $\varrho_X$ are invariant under 
the action of $G$.

\begin{lem}
There are finitely many chambers ${\cal C}$
such that $\overline{\cal C}$ contains $\varrho_X$
up to the action of $G$. 
\end{lem}

$(x,\eta,b) \in v^\perp$ if and only if 
$(\xi,\eta)=xa$.
We take $w \in v^\perp \cap H^*(X,{\Bbb Z})_{\alg} \otimes {\Bbb Q}$ 
such that $\langle w,\varrho_X \rangle=1$.
Replacing $w$ by $w+y \varrho_X$, we may assume that
$\langle w^2 \rangle=0$.
We set $w:=e^{\eta_0}$.
Then $v^\perp={\Bbb Q}w+{\Bbb Q}\varrho_X+L$,
where $L$ is generated by $(0,\eta,(\eta,\eta_0))=\eta e^{\eta_0}$, 
$\eta \in \xi^\perp$.
$T_D(r w+\eta e^{\eta_0}+b \varrho_X)=
e^D(r w+\eta e^{\eta_0}+b \varrho_X)
=r w+(\eta+rD)e^{\eta_0}+(b+(\eta,D)+\frac{(D^2)}{2}r )\varrho_X$.

\end{NB}

\section{Appendix}\label{sect:appendix}

\subsection{The base of Lagrangian fibrations}

Let
$\Phi_{X \to Y}^{{\bf E}^{\vee}[k]}:{\bf D}(X) \to {\bf D}^\alpha(Y)$
be the Fourier-Mukai transform in 
the proof of Proposition \ref{prop:Lag}.
Then we have an isomorphism 
$$
\phi:M_{(\beta,tH)}(v) \to M_{\widetilde{tH}}^\alpha(v')
$$
and a morphism
$$
f:M_{\widetilde{tH}}^\alpha(v') \to \Hilb_Y^\eta.
$$
We set  
$v':=(0,\eta,b), b \in {\Bbb Z}$ and $H':=\widetilde{tH}$.
Since ${\frak a}:M_{(\beta,tH)}(v) \to X \times \widehat{X}$
is the albanese map,
$$
{\frak a}':
M_{H'}^\alpha(0,\eta,b) \to M_{(\beta,tH)}(v) \overset{\frak a}{\to} 
 X \times \widehat{X}
$$
 is the albanese map.
Then 
$$
M_{H'}^\alpha(0,\eta,b) \to \Hilb_Y^\eta \to
\Pic^0(Y)
$$
 induces a morphism
$g:X \times \widehat{X} \to \Pic^0(Y)$ with a 
commutative diagram
\begin{equation*}
\begin{CD}
M_{H'}^\alpha(0,\eta,b) @>{f}>> \Hilb_Y^\eta\\
@V{{\frak a}'}VV @VVV\\
X \times \widehat{X} @>{g}>> \Pic^0(Y).
\end{CD}
\end{equation*}
Let $K_{H'}^\alpha(0,\eta,b)$ be a fiber of 
$M_{H'}^\alpha(0,\eta,b) \to
 X \times \widehat{X}$.
Since $\Hilb_Y^\eta \to \Pic^0(Y)$ is a 
${\Bbb P}^{(\eta^2)/2-1}$-bundle,
we have a morphism $K_{H'}^\alpha(0,\eta,b) \to {\Bbb P}^{(\eta^2)/2-1}$.
We shall prove the following.

\begin{prop}\label{prop:connected} 
The fiber of $K_{H'}^\alpha(0,\eta,b) \to {\Bbb P}^{(\eta^2)/2-1}$
is connected.
\end{prop}

For $D \in \Hilb_Y^\eta$,
$f^{-1}(D)$ consists of $\alpha$-twisted stable sheaves $E$ 
such that $E$ is an ${\cal O}_D$-module.
We take an effective divisor $D \in \Hilb_Y^\eta$.
Since a fiber of $K_{H'}^\alpha(0,\eta,b) \to {\Bbb P}^{(\eta^2)/2-1}$
is a fiber of $f^{-1}(D) \to X \times \widehat{X}$,
we shall study the map $f^{-1}(D) \to X \times \widehat{X}$.
For the connectivity of fibers,
we may assume that $D$ is a general member of $\Hilb_Y^\eta$.
Indeed since ${\Bbb P}^{(\eta^2)/2-1}$ is a normal variety
over a field of characteristic 0,
the connectivity of the generic fiber implies the connectivity
for all fibers.

\begin{NB}
It is sufficient to prove the claim for
a general fiber of
$K_{({\cal X},{\cal H})/T}^\alpha(0,\eta,b) \to \Hilb_{{\cal Y}/T}^\xi 
\times_{\widehat{\cal Y}} T$.
Indeed if the claim holds, then the
Stein factorization is trivial by
the smoothness of $\Hilb_{{\cal Y}/T}^\xi 
\to {\widehat{\cal Y}} $.

For the morphism
$\zeta:M_{({\cal X},{\cal H})/T}^\alpha(0,\eta,b) \to \Hilb_{{\cal Y}/T}^\xi$,
the smoothness of $M_{({\cal X},{\cal H})/T}^\alpha(0,\eta,b)$ and
$\Hilb_{{\cal Y}/T}^\xi$ over $T$ imply that
for any $t \in T$,
$\zeta$ is smooth over a general point $D \in (\Hilb_{{\cal Y}/T}^\xi)_t$.
For this fiber $\zeta_D$,
we have a homomorphism
$H_1(\zeta_D,{\Bbb Z}) \to H_1({\cal Y}_t,{\Bbb Z})$
which is locally constant.
So for the surjectivity of this map, it is sufficient to
prove the claim for a special $D$ such that 
$\zeta$ is smooth over $D$.

\end{NB}

\begin{NB}
Let $C$ be a curve in an abelian surface $Y$ such that $C$ is a
hyperplane section of $Y$ in a projective space.
\begin{rem}
Assume that $C=d C_0$, $C_0$ is primitive.
If $d \geq 3$, or $d=2$ and $(C_0^2)>2$ or
$d=1$ and $(C^2) \geq 10$, then
$C$ is very ample. 
\end{rem}
Then Lefschetz hyperplane section theorem
implies that $H_1(C,{\Bbb Z}) \to H_1(Y,{\Bbb Z})$ is surjective. 
Then the kernel of $f:\Jac(C) \to Y$ is connected.
Indeed for the abelian variety $Z:=\Jac(C)/(\ker f)_0$,
we have a surjective homomorphism
$H_1(C,{\Bbb Z}) \to H_1(Z,{\Bbb Z}) \to H_1(Y,{\Bbb Z})$.
Then $\dim Z \geq \dim Y$.
Since $Z \to Y$ is finite,
we have $\dim Z \leq 2$. 
Therefore $Z \to Y$ is an isogeny.
Then the surjectivity of   
$H_1(Z,{\Bbb Z}) \to H_1(Y,{\Bbb Z})$ implies that
it is an isomorphism. Thus $\ker f$ is connected.

For ${\cal O}_C(D) \in \Pic^0(C)$ with $D=\sum_i n_i p_i$, 
we have ${\cal O}_C(D)={\cal O}_C+\sum_i n_i {\Bbb C}_{p_i}$
in $K(C)$.
Hence 
$$
\Phi_{Y \to X}^{{\bf E}}({\cal O}_C(D))=
\Phi_{Y \to X}^{{\bf E}}({\cal O}_C)+\sum_i n_i {\bf E}_{p_i}
$$ 
in $K(X)$.
Then we have 
${\frak a}(\Phi_{Y \to X}^{{\bf E}}({\cal O}_C(D)))
={\frak a}(\Phi_{Y \to X}^{{\bf E}}({\cal O}_C))
+\sum_i n_i {\frak a}({\bf E}_{p_i})$.
This morphism is the same as 
$\Pic^0(C)\cong \Jac(C) \overset{f}{\to} 
Y  \overset{{\frak a}}{\to} X \times \widehat{X}$ sending
${\cal O}_C(D)$ to the image of $\sum_i n_i p_i \in Y$
by ${\frak a}$. 

In particular $\ker(\Pic^0(C) \to X \times \widehat{X})$ is connected.

Let ${\cal X} \to T$ be a family of abelian surface
with a family of moduli spaces ${\cal Y} \to T$.
Let ${\cal C} \subset {\cal Y}  \times_T \Hilb_{{\cal Y}/T}^\xi$
be the universal family of curves.
If ${\cal C}_s$ ($s \in \Hilb_{{\cal Y}/T}^\xi$) is smooth
then
$(\Pic^0_{{\cal C}/\Hilb_{{\cal Y}/T}^\xi})_s$ is the fiber of
${\cal M}_{{\cal Y}/T}^\alpha(v) \to \Hilb_{{\cal Y}/T}^\xi$.

We have a morphisms
$\Pic^0_{{\cal C}/\Hilb_{{\cal Y}/T}^\xi} \to 
{\cal X} \times_T \widehat{\cal X}$.
\end{NB}

The following well-known result is due to Reider.
\begin{prop}\label{prop:base-point-free}
Let $X$ be an abelian surface defined over an algebraically closed field
$k$ and $D$ a divisor on $X$.
If $(X,D) \not \cong (C_1 \times C_2,C_1+k C_2)$ and
$(D^2)>4$, then $|D|$ is base point free.
Moreover $|D|$ is fixed point free, if 
$(X,D) \not \cong (C_1 \times C_2,C_1+k C_2)$ and $(D^2)=4$.
\end{prop}

\begin{lem}\label{lem:tree-config}
For a general point of $\Hilb_Y^\eta$,
$D$ is a normal crossing divisor such that
each component $D_i$ is smooth and
the configuration is tree.
\end{lem}

\begin{proof}
We take a line bundle $L$ on $Y$ with $c_1(L)=\eta$.
If $|L|$ is base point free, then
Bertini's theorem implies that
a general member of $|L|$ is a smooth divisor.
Assume that $|L|$ has a base point.
Then Proposition \ref{prop:base-point-free}
implies that 
(i) $(\eta^2)=4$ and $L$ does not have a fixed component
or (ii) there is an elliptic curve $C$ on $Y$
with $(\eta,C)=1$.
In the first case,
$$
K(L):=\{x \in Y \mid T_x^*(L) \cong L \}
$$
is a subgroup of order 4.
Let $\mathrm{Bs}(L)$ be the set of base points.
Then by the action of $K(L)$, $\mathrm{Bs}(L)$ is invariant.
Therefore $\# \mathrm{Bs}(L) \geq \# K(L)$.
Since $L$ does not have a fixed component,
$4 \geq \# \mathrm{Bs}(L)$.
Therefore $\mathrm{Bs}(L)$ consists of 4 points.
For two $D,D' \in |L|$, $D$ and $D'$ intersect transversally.
Therefore $D$ is smooth at base points.
By using Bertini's theorem, $D$ is smooth for a general
member of $|L|$. 
For case (ii), there is an elliptic curve $C'$ such that
$(C,C')=1$ and 
$\eta=C+n C'$, where $n=( \eta^2 )/2$.
Since $nC'$ is linear equivalent to $\sum_{i=1}^n C_i$ with
$C_i \cap C_j=\emptyset$ $(i \ne j)$, we get the claim.  
\end{proof}

\subsection{Moduli of twisted-stable sheaves on $D$} 

By Lemma \ref{lem:tree-config},
we shall study $f^{-1}(D)$ for a normal crossing divisor
$D=\sum_{i=0}^m D_i$ such that
$D_i$ are smooth curves and the configuration of
$D_i$ is tree.
We may assume that $D=D_0+D_1+\cdots+D_m$ and 
$p_1,p_2,...,p_m$ are the singular points of $D$
such that $p_i=D_{\varphi(i)} \cap D_i$ with $\varphi(i)<i$.

%For each $p_i$, we have a decomposition $D=A^i+B^i$ with 
%$A^i \cap B^i=\{p_i \}$.

By looking at the dual graph of irreducible components,
we have the following lemma.
\begin{lem}\label{lem:decomposition}
For each singular point $p_i$, we have a unique 
decomposition $D=A^i+B^i$ with 
$A^i \cap B^i=\{p_i \}$.
\end{lem}

\begin{NB}

There is a unique decomposition
$D=A+B$ such that $A \cap B =\{x \}$.

\begin{proof}
We set $U:=X \setminus \{ x \}$.
If there is a desired decomposition
$D=A+B$, then $A':=A \cap U$ and $B'=B \cap U$ are the
connected components of $D \cap U$.
\begin{NB2}
$A' \cap B'=A \cap B \cap U= \emptyset$.
For two point $y,z \in A$,
there is a path $I$ in $A$ connecting $y$ and $z$.
By perturbing $I$, we may assume that $x \not \in I$.
Hence $A$ is connected. 
\end{NB2}
Hence the decomposition is unique.

If $D \cap U$ has connected components $A_1,A_2,...,A_k$,
then the closure $\overline{A}_i$ intersects at $x$.
Hence $k=1,2$.
If $D\cap U$ is connected, then the configuration of
$D_i$ is not tree. Therefore $D \cap U$ has two
connected components  
$A'$ and $B'$.
Then $A=\overline{A}'$ and $B=\overline{B}'$
are the desired decomposition of $D$.
\end{proof}
\end{NB}

\begin{lem}\label{lem:independent}
In the free abelian group generated by
$D_0,D_1,..,D_m$, we have
$$
{\Bbb Z}D+{\Bbb Z}A^1+\cdots + {\Bbb Z}A^m
={\Bbb Z}D_0 \oplus {\Bbb Z}D_1 \oplus \cdots \oplus 
{\Bbb Z}D_m.
$$
\end{lem}

\begin{proof}
Before proving this lemma, we note that
${\Bbb Z}D+{\Bbb Z}A^i={\Bbb Z}D+{\Bbb Z}B^i$. Thus
the left hand side is independent of the choice of $A^i$ in the
decomposition $D=A^i+B^i$.

For each $D_i$, let $p_{n_1},...,p_{n_t} \in D_i$ be the singular points of
$D$.
Then we have the decompositions
$D=A^{n_j}+B^{n_j}$ with $A^{n_j} \cap B^{n_j}=\{p_{n_j} \}$.
We may assume that $D_i \subset A^{n_j}$ for all $j$.
Then $D=D_i+\sum_j B^{n_j}$. 
Hence $\sum_{i=0}^m {\Bbb Z}D_i \subset {\Bbb Z}D+\sum_{i=1}^m A^i$, 
which implies the claim.
\end{proof}

\begin{defn}
Let $\beta$ be a ${\Bbb Q}$-Cartier divisor on $D$.
\begin{enumerate}
\item[(1)]
For a subdivisor
$D' \subset D$,
$\beta(D'):=\int_{D'} \beta  \in {\Bbb Q}$ 
denotes the degree of $\beta_{|D'} \in H^2(D',{\Bbb Q})$.
\item[(2)]
For a coherent sheaf $E$ on $D$,
we set $\chi(E(-\beta)):=\chi(E)-\beta(\Div(E))$.
\end{enumerate}
\end{defn}

\begin{defn}
We have a surjective homomorphism
\begin{equation*}
\begin{matrix}
\deg^m :& \Pic(D)& \to & \bigoplus_{i=0}^m {\Bbb Z}\delta_i\\
& E & \to & \sum_{i=0}^m \deg(E_{|D_i})\delta_i.
\end{matrix}
\end{equation*}
Indeed for a smooth point $p_i \in D_i$,
we have a Cartier divisor and get
a line bundle ${\cal O}_D(p_i)$ on $D$. 
Then $\deg^m({\cal O}_D(p_i))=\delta_i$.
\end{defn}

\begin{defn}
Let $H'$ be an ample divisor on $D$.
A purely 1-dimensional sheaf
$E$ on $D$ is $\beta$-twisted semi-stable, if
\begin{equation*}
\frac{\chi(F(-\beta))}{(\Div(F),H')} 
\leq \frac{\chi(E(-\beta))}{(\Div(E),H')}
\end{equation*}
for all $0 \ne F \subset E$.
If the inequality is strict for all proper
subsheaf $F$, then
$E$ is $\beta$-twisted stable.
\end{defn}

Since $H^2_{\text{\'{e}t}}(D,{\cal O}_D^{\times})=0$,
by refining the covering of $D$,
we have an $\alpha_{D}$-twisted line bundle $L$ on $D$ 
which induces an equivalence
\begin{equation*}
\begin{matrix}
\Coh^{\alpha}(D)& \to & \Coh(D)\\
E & \mapsto & E \otimes L^{\vee}.
\end{matrix}
\end{equation*}
Let $G$ be a locally free $\alpha$-twisted sheaf defining 
twisted semi-stability of $f^{-1}(D)$.
Then $G':=G_{|D} \otimes L^{-1}$ is a locally free sheaf on 
$D$. We set $\beta:=c_1(G')/\rk G' \in H^2(D,{\Bbb Q})$.
Thus we have an isomorphism
$f^{-1}(D) \to M_D^\beta(v)$, where
$M_D^\beta(v)$ is the moduli space of
$\beta$-twisted sheaves on $D$ with $v(E)=v$
and the polarization is $H'_{|D}$.
We shall describe $M_D^\beta(v)$.

Let $x$ be a smooth point of $D$.
Then the stalk $E_x$ is a free ${\cal O}_{D,x}$-module.
Since $\Div(E)=D$,
the classification of finitely generated ${\cal O}_{D,x}$-module
implies that $E_x \cong {\cal O}_{D,x}$.

\begin{NB}
Let ${\cal Z}$ be the fiber of 
${\cal M}_{{\cal Y}/T}^\alpha(v) \to  \Hilb_{{\cal Y}/T}^\xi$.
If ${\cal C}_s$ is smooth, then
${\cal Z}_s \cong \Pic^0({\cal C}_s)$ is an abelian variety.
If ${\cal Z}_s$ $(s \in  \Hilb_{{\cal Y}/T}^\xi)$ is smooth, then 
$$
\varphi_s :H^1(({\cal X} \times_T \widehat{\cal X})_s,{\Bbb Z})
\to H^1({\cal Z}_s,{\Bbb Z})
$$
is independent of $s  \in \Hilb_{{\cal Y}/T}^\xi$.

Assume that
${\cal X}_t$ is a product of elliptic curves $C$ and $D$ such that 
${\cal C}_s=C_0+C_1+\cdots +C_n$, where $e \in D$ 
is the zero of the group action,
$C_0=C \times \{e \}$ 
and $C_i=\{p_i \} \times D$ for $i>0$.
Then ${\cal Z}_s$ is smooth and it is isomorphic
to $C_0 \times C_1 \times \cdots  \times C_n$.
In this case, $C_0 \times C_1 \times \cdots  \times C_n \to 
X \times \widehat{X}$
factors through $Y$ and 
$H_1(C_0 \times C_1 \times \cdots  \times C_n,{\Bbb Z}) \to
H_1(Y,{\Bbb Z})$ is surjective.
Hence $\coker \varphi_s$ is torsion free.

We set $A:=C_0 \times C_1 \times \cdots  \times C_n$.

Let $E$ be a stable sheaf whose support is
$A$.
Then we have an exact sequence
\begin{equation*}
0 \to E_{|C_0}(-\sum_{i>0} p_i) \to E \to
\oplus_{i>0} E_{|C_i} \to 0.
\end{equation*}

\begin{rem}
Indeed the kernel $K$ of the restriction map
$E \to \oplus_{i>0} E_{|C_i}$ is a purely 1-dimensional sheaf
and the multiplication map
$K \to K(C_0)$ is zero on $C \setminus \{p_1,...,p_n \}$.
Hence $K$ is an ${\cal O}_{C_0}$-module.
Thus we have a surjection 
$E(-\sum_{i>0} C_i) \to E(-\sum_{i>0} C_i)_{|C_0} \to K$.
If $E_{|C_0}$ is purely 1-dimensional, then
$E(-\sum_{i>0} C_i)_{|C_0} \to K$ is injective and get the claim.  
\end{rem}

\end{NB}

\begin{prop}\label{prop:Pic_D}
Assume that $\beta$ is general, that is,
$M_D^\beta(v)$ consists of $\beta$-twisted stable sheaves.
\begin{enumerate}
\item[(1)]
$M_D^\beta(v)$ is non-empty and consists of line bundles on $D$.
\item[(2)]
$M_D^\beta(v)$ is isomorphic to $\prod_i \Pic^0(D_i)$.
In particular, $M_D^\beta(v)$ is an abelian variety.
\end{enumerate}
\end{prop}

For the proof of Proposition \ref{prop:Pic_D},
we first prove the following.
 
\begin{lem}\label{lem:locally-free}
$M_D^\beta(v)$ consists of locally free ${\cal O}_D$-modules.  
\end{lem}

\begin{proof}
For $E \in M_D^\beta(v)$,
assume that $E_{|D_i}$ is torsion free.
Then
$E_{|D_i}$ is a locally free sheaf of rank 1 on $D_i$.
By using Nakayama's lemma, we have a surjective homomorphim
${\cal O}_{X, x} \to E_x$ for all $x \in D_i$.
Then we have a surjective homomorphism
${\cal O}_{U} \to E_{|U}$ for a neighborhood $U$ of $x$.
Since $E$ is an ${\cal O}_D$-module,
we have a surjective homomorphism
$\psi:{\cal O}_{D \cap U} \to E_{|U}$. 
Since $E$ is a locally free ${\cal O}_D$-module
over $X \setminus \cup_{j \ne k} D_j \cap D_k$,
$\Supp \ker \psi \subset \cup_{j \ne k} D_j \cap D_k$.
Hence $\psi$ is an isomorphism. 
Therefore it is sufficient to show the torsion freeness of
$E_{|D_i}$.
Assume that the torsion module $T$ of $E_{|D_i}$ is not zero.
Then there is a component $D_j$ such that $T_{p_i} \ne 0$ at
$p_i \in D_i \cap D_j$.
We take the decomposition $D=A^i+B^i$ with $A^i \cap B^i=\{ p_i \}$
in Lemma \ref{lem:decomposition}.
We may assume that $D_i \subset A^i$ and $D_j \subset B^i$.
Let $T'$ be the torsion submodule of $E_{A^i}$.
Then $T$ is a direct summand of $T'$ with $T_{p_i}=T_{p_i}'$.
For the morphism $E \to E_{|A^i}/T'$,
the kernel contains a submodule $F$ fitting in an exact sequence
\begin{equation*}
0 \to  (E_{|B^i}/T'')(-p_i) \to F \to T' \to 0, 
\end{equation*}
where $T''$ is the torsion submodule of $E_{|B^i}$.
Then we have
\begin{equation*}
\frac{\chi((E_{|B^i}/T'')(-p_i-\beta))+\chi(T')}{(B^i,H')}=
\frac{\chi(F(-\beta))}{(B^i,H')}<
\frac{\chi(E(-\beta))}{(D,H')}<
\frac{\chi((E_{|B^i}/T'')(-\beta))}{(B^i,H')}
\end{equation*}
Since $\chi((E_{|B^i}/T'')(-p_i-\beta))+\chi(T') \geq 
\chi((E_{|B^i}/T'')(-\beta))$, we get a contradiction.
Therefore $E_{|D_i}$ is torsion free, and we complete the proof.
\end{proof}

We next characterize the Mukai vectors of $E_{|D_i}$
for $E \in M_D^\beta(v)$. 
We set $v(E_{|A^i})=(0,A^i,a_i)$ and $v(E_{|B^i})=(0,B^i,b_i)$.
Since $E_{|B^i}(-p_i)$ is a subsheaf of $E$ and
$E_{|B^i}$ is a quotient of $E$,
we have
\begin{equation}\label{eq:(D,H')}
\frac{b_i-1-\beta(B^i)}{(B^i,H')}
<\frac{\chi(E)-\beta(D)}{(D,H')}<
\frac{b_i-\beta(B^i)}{(B^i,H')}.
\end{equation}
Hence
\begin{equation}\label{eq:b}
b_i=\min\left\{ n \in {\Bbb Z} \left|
n>\frac{(\chi(E)-\beta(D))(B^i,H')}{(D,H')}+\beta(B^i)
\right. \right\}.
\end{equation}

Conversely if \eqref{eq:b} holds for a line bundle $E$ on $D$, 
then we shall prove that
$E$ is $\beta$-twisted stable.

We first note that \eqref{eq:(D,H')} holds and
\begin{equation}\label{eq:a}
\frac{a_i-1-\beta(A^i)}{(A^i,H')}
<\frac{\chi(E)-\beta(D)}{(D,H')}<
\frac{a_i-\beta(A^i)}{(A^i,H')}.
\end{equation}
For an exact sequence
\begin{equation*}
0 \to E_1 \to E \to E_2 \to 0
\end{equation*}
such that $E_1$ and $E_2$ are purely 1-dimensional and
$E_1$ is $\beta$-twisted stable,
$\Div(E_1)$ is connected.
We set 
$$
\Div(E_1) \cap \Div(E_2):=\{ p_{n_1},p_{n_2},...,p_{n_s} \}.
$$
We note that 
$p_{n_j}=D_{a_j} \cap D_{b_j}$ with
$\{a_j,b_j \}=\{ n_j,\varphi(n_j) \}$.
We may assume that $D_{a_j} \subset E_1$ and 
$D_{b_j} \subset E_2$. 
Replacing $B^i$ by $A^i$,
we may assume that 
$D_{a_j} \subset A^{n_j} \cap \Div(E_1)$
and
$D_{b_j} \subset \Div(E_2) \cap B^{n_j}$
for $1 \leq j \leq s$.
Since $A^{n_j} \setminus \{ p_{n_j} \}$
is a connected component of $D \setminus \{ p_{n_j} \}$,
connectivity of $\Div(E_1) \setminus \{p_j \}$ implies that
$\Div(E_1) \subset A^{n_j}$. 
Since $D$ is a tree configuration, we also have
$B^{n_j} \cap B^{n_k} =\emptyset$ for $j \ne k$.  
Hence we have a decomposition of $\Div(E_2)=D-\Div(E_1)$
into connected components $B^j$:
$\Div(E_2)=\sum_j B^{n_j}$.
By \eqref{eq:(D,H')}, we have
\begin{equation*}
\frac{\chi(E(-\beta))}{(D,H')}(B^{n_j},H')<
\chi(E_{|B^{n_j}}(-\beta)).
\end{equation*}
Then we have 
\begin{equation*}
\sum_j \chi(E_{|B^{n_j}}(-\beta)) >
\sum_j \frac{\chi(E(-\beta))}{(D,H')}(B^{n_j},H')=
\frac{\chi(E(-\beta))}{(D,H')}(\Div(E_2),H').
\end{equation*} 
Since $\cup_j B^{n_j}$ is a disjoint union,
we have a surjective homomorphism
$E \to \oplus_j E_{|B^{n_j}}$.
Since $E_1 \to E \to \oplus_j E_{|B^{n_j}}$ is a zero map,
we have a surjective morphism
$E_2 \to \oplus_j E_{|B^{n_j}}$.
Since $\Div(E_2)=\sum_j B^{n_j}$ and $E_2$ is pure,
it is an isomorphism.
Therefore 
\begin{equation*}
\chi(E_2(-\beta))>
\frac{\chi(E(-\beta))}{(D,H')}(\Div(E_2),H'),
\end{equation*} 
 which implies $E$ is $\beta$-twisted stable.

\begin{rem}\label{rem:chi(D_i)}
By the proof of Lemma \ref{lem:independent},
we also have an exact sequence
\begin{equation*}
0 \to {\cal O}_{D_i}(-\sum_j p_{n_j}) \to
{\cal O}_D \to \oplus_j {\cal O}_{B^{n_j}} \to 0.
\end{equation*}
Hence
\begin{equation*}
\chi(E_{|D_i})=(D-D_i,D_i)+\chi(E)-
\sum_j b_{n_j}.
\end{equation*}
\end{rem}

\begin{defn}
For a sequence of smooth curves $C_1,C_2,...,C_s$ in 
$X$ and a sequence of integers $d_1,d_2,...,d_s$,
$\Pic^{d_1,d_2,...,d_s}(\sum_{j=1}^s C_j)$
denotes the moduli spaces of line bundles $E$ on
$\sum_{j=1}^s C_j$ such that 
$$
\chi(E_{|C_i})=d_i+(1-g(C_i)).
$$  
\end{defn}

\begin{NB}
Since $H^0(\sum_i C_i,{\cal O}_{\sum_i C_i})={\Bbb C}$,
$\Pic^{d_1,d_2,...,d_s}(\sum_{j=1}^s C_j)$ exists as an algebraic space.
\end{NB}

By Lemma \ref{lem:independent},
we have a bijective correspondence between
$(\chi(E_{|D_0}),\chi(E_{|D_1}),...,\chi(E_{|D_m}))$
and 
$(\chi(E),b_1,b_2,...,b_m)$.
Hence Proposition \ref{prop:Pic_D} follows from the following
claim.

\begin{lem}\label{lem:Picard-scheme}
For $E \in M_D^\beta(v)$, 
we set $d_i:=\deg(E_{|D_i})=
\chi(E_{|D_i})-(1-g(D_i))$.
\begin{NB}
Then $u_i:=v(E_{|D_i}(-p_i))=(0,D_i,d_i-1)$.
We set
$$
v_i:=v(E_{|D_0})+\sum_{j=1}^i u_j=
(0,\sum_{j \leq i} D_j,d_0+\sum_{j=1}^i (d_j-1)).
$$  
$\Pic^{(d_0,d_1,...,d_i)}(\sum_{j=0}^i D_j)$ 
denotes the moduli space of line bundles
$E$ on $\sum_{j=0}^i D_j$ with $\chi(E_{|D_j})=d_j$
for all $j$. 
\end{NB}
Then we have an isomorphism
$$
\Pic^{d_0,d_1,...,d_i}(\sum_{j=0}^i D_j) \cong 
\Pic^{d_0,d_1,...,d_{i-1}}(\sum_{j=0}^{i-1} D_j) 
\times \Pic^d_i(D_i).
$$
In particular, 
$\Pic^{d_0,d_1,...,d_m}(\sum_{j=0}^m D_j)
\cong \prod_j \Pic^{d_j}(D_j)$. 
\end{lem}

\begin{proof}
For $E \in \Pic^{d_0,d_1,...,d_i}(\sum_{j=0}^i D_j)$, 
we have 
$$
(E_{|\sum_{j<i} D_j},E_{|D_i}) \in 
\Pic^{d_0,d_1,...,d_{i-1}}(\sum_{j=0}^{i-1} D_j) 
\times \Pic^d_i(D_i)$$ 
and $E$ fits in an exact sequence
\begin{equation*}0 \to E_{|D_i}(-p_i) \to E
\to E_{|\sum_{j<i} D_j} \to 0.
\end{equation*}
Since $\Ext^k(E_{|\sum_{j<i} D_j},E_{|D_i}(-p_i))=0$ for $k \ne 1$
and 
\begin{equation*}
\begin{split}
\Ext^1(E_{|\sum_{j<i} D_j},E_{|D_i}(-p_i))\cong &
H^0(X,{\cal E}xt^1_{{\cal O}_X}(E_{|\sum_{j<i} D_j},E_{|D_i}(-p_i)))\\
\cong & H^0(X,{\Bbb C}_{p_i}),
\end{split}
\end{equation*}
$E$ is uniquely determined by 
$(E_{|\sum_{j<i} D_j},E_{|D_i}(-p_i))$.
\end{proof}

\subsection{Proof of Proposition \ref{prop:connected}}

\begin{lem}\label{lem:H_1}
Let $C$ be a smooth curve of $(C^2)>0$ in an abelian surface $Y$.
Then $H_1(C,{\Bbb Z}) \to H_1(Y,{\Bbb Z})$ is surjective.
\end{lem}

\begin{proof}
If it is not surjective, then
$f:\Pic^0(Y) \to \Pic^0(C)$ is not injective.
For the abelian surface $f(\Pic^0(Y))$,
we set $Y':=\Pic^0(f(\Pic^0(Y)))$. 
Since $C \to Y$ factors through $\Alb(C)$, 
we have the following diagram
\begin{equation*}
\begin{CD}
C @>>> Y'\\
@| @VV{g}V\\
C @>>> Y
\end{CD}
\end{equation*}
Let $y$ be a point of $\ker g$.
Since $T_y^*(C)$ is algebraically equivalent to
$C$, we have $(T_y^*(C),C)=(C,C)>0$.
Thus $T_y^*(C) \cap C \ne \emptyset$.
For a point $s \in T_y^*(C) \cap C$,
$g(s)=g(s+y)$ for the points
$s, s+y \in C$.
Thus $g_{|C}$ is not injective.
Therefore $H_1(C,{\Bbb Z}) \to H_1(Y,{\Bbb Z})$ is surjective.
\end{proof}

For the divisor $D \in \Hilb_Y^\eta$ in Lemma \ref{lem:tree-config},
we take $E \in f^{-1}(D)$.
Let ${\frak d}$ be a Cartier divisor of $D$ 
such that
${\frak d}=\sum_i {\frak d}_i$, 
${\frak d}_i=\sum_j n_{ij} p_{ij}$,
$p_{ij} \in D_i \setminus \cup_{k \ne i} D_k$
and $\deg({\frak d}_i)=\sum_{j} n_{ij}=0$ for all $i$.
For ${\cal O}_D({\frak d}) \in \Pic(D)$, 
we have ${\cal O}_D({\frak d})={\cal O}_D+
\sum_{i,j} n_{ij} {\Bbb C}_{p_{ij}}$
in $K(D)$.
Hence 
$$
\Phi_{Y \to X}^{{\bf E}}(E({\frak d}))=
\Phi_{Y \to X}^{{\bf E}}(E)+\sum_{i,j} n_{ij} {\bf E}_{p_{ij}}
$$ 
in $K(X)$.
Then we have 
$$
{\frak a}(\Phi_{Y \to X}^{{\bf E}}(E(\frak{d})))
={\frak a}(\Phi_{Y \to X}^{{\bf E}}(E))
+\sum_{i,j} n_{ij} {\frak a}({\bf E}_{p_{ij}}).
$$
This morphism is the same as 
$$
\prod_i \Pic^0(D_i)\cong \prod_i \Jac(D_i) \overset{\mu}{\to} 
Y  \overset{{\frak a}}{\to} X \times \widehat{X}
$$
 sending
${\cal O}_D(\sum_{i,j} n_{ij} p_{ij})$ 
to the image of $\sum_{i,j} n_{ij} p_{ij} \in Y$
by ${\frak a}$.

\begin{lem}
${\frak a}:Y \to X \times \widehat{X}$ sending $y \in Y$ to
${\frak a}({\bf E}_y) \in X \times \widehat{X}$ is injective.
\end{lem}

\begin{proof}
We set $u=(r,\xi,a)$.
Replacing $u$ by $-u$, we may assume that $r>0$.
\begin{NB}
Replacing $E \in M_H(u)$ by 
$\Phi(E)$, we may assume that $r>0$ and $\xi$ is ample, where 
$\Phi$ is a suitable autoequivalence of ${\bf D}(X)$.
Indeed we have a commutative diagram:
\begin{equation*}
\begin{CD}
{\bf D}(X) @>{\Phi}>> {\bf D}(X)\\
@V{{\frak a}}VV @VV{{\frak a}}V\\
X \times \widehat{X} @>>> X \times \widehat{X}
\end{CD}
\end{equation*}
\end{NB}
We set $p=(r,\xi)$. Since $v$ is primitive,
$(p,a)=1$.
Since $ra=(\xi^2)/2 \in p^2 {\Bbb Z}$,
we may set $r=p^2 t$ and $\xi=pqH$, where $H$ is primitive.
Since $q^2 (H^2)/2=ta$ and $(q,pt)=1$,
we can set $a=q^2 s$. Thus we get
$v=(p^2 t,pqH,q^2 s)$, where $(pt,q)=1$, $(p,q^2 s)=1$ and the type of
$H$ is $(1,ts)$.
We take $E \in Y$.
We have a morphism
$f:X \to Y$ by sending $x \in X$ to $T_x^*(E)$. 
Then $X/K(E) \cong Y$, where 
$K(E)=\im(K(pqH) \overset{p^2 t}{\to} X)$.
We note that 
$g:X \to Y \to X \times \widehat{X}$ is
the morphism sending $x$ to $(q^2 s x,\phi_{pqH}(x))$.
We shall prove that $\ker g=K(E)$.
 
We set $K(pq H)=\frac{1}{pq}V_1/V_1 \oplus \frac{1}{pqts} V_2/V_2$.
Then 
$$
K(E)=p^2 t K(pqH)=
\frac{p^2 t}{pq}V_1/V_1 \oplus \frac{p^2 t}{pqts} V_2/V_2  
=\frac{1}{q}V_1/V_1 \oplus \frac{1}{qs} V_2/V_2,
$$ 
where we used $(pt,q)V_1=V_1$ and $(p,qt)V_2=V_2$.
For $(\frac{x_1}{pq},\frac{x_2}{pqts}) \in \ker (q^2 s) \cap K(pqH)$,
we have $(\frac{q^2 s x_1}{pq},\frac{q^2 s x_2}{pqts}) \in V_1 \oplus V_2$.
Then we have $x_1 \in p V_1$ and $x_2 \in pt V_2$.
Hence  
$$
\ker (q^2 s) \cap K(pqH)=\frac{1}{q}V_1/V_1 
\oplus \frac{1}{qs}V_2/V_2=K(E).
$$
 \end{proof}

By Proposition \ref{prop:Pic_D}, Proposition \ref{prop:connected}
follows from the following claim.
\begin{lem}\label{lem:connected}
If $D$ is a normal crossing divisor
of smooth curves $D_i$, then
$\ker \mu$ is connected.
\end{lem}

\begin{proof}
Since $D_i$ are smooth,
it is sufficient to prove that
$H_1(\prod_i \Pic^0(D_i), {\Bbb Z}) \to
H_1(Y,{\Bbb Z})$ is surjective.
Since 
$$
H_1(\prod_i \Pic^0(D_i), {\Bbb Z}) \cong
\bigoplus_i H_1(\Pic^0(D_i),{\Bbb Z}) \cong 
\bigoplus_i H_1(D_i,{\Bbb Z}),
$$
Lemma \ref{lem:H_1} implies the claim unless
all $D_i$ are elliptic curves.
If all $D_i$ are elliptic curves, then
$(D_0,D_1)=1$ implies that
the natural homomorphism
$D_0 \times D_1 \to Y$ is an isomorphism.
Hence 
$$
H_1(D_0,{\Bbb Z}) \oplus H_1(D_1,{\Bbb Z})
\to H_1(Y,{\Bbb Z})
$$ 
is an isomorphism.
Therefore the claim holds. 
\end{proof}

\end{document}